\documentclass[11pt,preprint]{article}

\usepackage{geometry}
\usepackage{authblk}

\usepackage{framed}
\usepackage{amsthm}

\usepackage{mathtools}
\usepackage{psfrag}
\usepackage{amssymb}
\usepackage{amsmath}
\usepackage{graphicx}
\usepackage{xcolor}
\usepackage[color=green!40]{todonotes}
\usepackage{enumitem}
\usepackage{pifont}

\usepackage{hyperref}

\newtheorem{theorem}{Theorem}[section]

\newtheorem{prob}[theorem]{Problem}
\newtheorem{remark}[theorem]{Remark}
\newtheorem*{vis}{Visibility condition}
 
\newtheorem{alg}{Algorithm}

\theoremstyle{definition}
\newtheorem{definition}[theorem]{Definition}

\numberwithin{equation}{section}
\numberwithin{figure}{section}
\numberwithin{theorem}{section}

\DeclareMathOperator{\sinc}{sinc}

\newcommand{\Hs}{\mathbb X}
\newcommand{\Ds}{\mathbb Y}
\newcommand{\mB}{\mathrm B}

\newcommand{\R}{\mathbb R}

\newcommand\edot{\,\cdot\,}

\newcommand{\Fo}{\mathcal  F}
\newcommand{\Ko}{\mathbf K}
\newcommand{\Lo}{\mathbf L}
\newcommand{\No}{\mathbf N}
\newcommand{\Io}{\mathbf I}
\newcommand{\To}{\mathbf T}

\newcommand{\supp}{\mbox{supp}}

\newcommand{\Hr}{ \mathcal H}
\newcommand{\Om}{\Omega}
\newcommand{\cT}{ \mathbb T}

\newcommand{\eps}{\epsilon}
\newcommand{\coloneqq}{:=}
\newcommand{\sg}{\sigma}

\newcommand\abs[1]{\left\vert#1\right\vert}
\newcommand\sabs[1]{\vert#1\vert}
\newcommand\norm[1]{\left\Vert#1\right\Vert}

\newcommand\set[1]{\left\{#1\right\}}
\newcommand\kl[1]{\left(#1\right)}

\newcommand{\fnum}{{\tt f}}

\newcommand{\gnum}{{\tt g}}
\newcommand{\Lnum}{\boldsymbol{\tt{L}}}
\newcommand{\Wnum}{\boldsymbol{\tt{W}}}
\newcommand{\Num}{\boldsymbol{\tt{R}}}
\newcommand{\Dnum}{\boldsymbol{\tt{D}}}
\newcommand{\interior}{\operatorname{Int}}

\newcommand{\WF}{ \mbox{WF}}
\newcommand{\mR}{ \mathcal R}
\newcommand{\mD}{ \mathcal D}
\newcommand{\pd}{{\partial}}

\numberwithin{theorem}{section}
\numberwithin{figure}{section}
\numberwithin{equation}{section}
\title{Analysis of Iterative Methods in Photoacoustic Tomography\\ with Variable Sound Speed}

\author{Markus Haltmeier} 
\affil{Department of Mathematics, University of Innsbruck\\
Technikestra{\ss}e 13, A-6020 Innsbruck, Austria\\
E-mail: {\tt markus.haltmeier@uibk.ac.at}}

\author{Linh V. Nguyen} 
\affil{Department of Mathematics, University of Idaho\\
875 Perimeter Dr, Moscow, ID 83844, US\\
E-Mail: {\tt lnguyen@uidaho.edu}}

\date{}

\begin{document}
\maketitle

\begin{abstract}
In this article, we revisit iterative methods for solving the inverse problem of
photoacoustic tomography in free space. Recently, there have been interesting developments on
explicit formulations of the adjoint operator, demonstrating that iterative methods is an attractive
choice for photoacoustic image reconstruction. In this work, we propose several modifications of current formulations
 of the adjoint operator which help  speed up the convergence  and yield improved error estimates. We establish
  a stability analysis and show that, with our choices of the adjoint operator, the iterative methods  can achieve a linear rate of convergence, in the $L^2$-norm as well as in the $H^1$-norm.
  In addition, we analyze the normal operator from the
  microlocal analysis point of view. This gives insight into the convergence speed of the iterative methods and choosing proper weights for the mapping spaces. Finally, we present numerical results using various iterative reconstruction methods for full as well as limited view data. Our results demonstrate that Nesterov's fast gradient and the CG  methods converge faster than Landweber's and iterative time reversal methods in the visible as well as the invisible case.


\end{abstract}

\section{Introduction}

Photoacoustic tomography (PAT) is an novel coupled-physics method for non-invasive imaging. It combines the high contrast of optical imaging  with the good resolution of ultrasound tomography. The biological object of interest is scanned with a laser light pulse. The photoelastic effect induces an acoustic pressure wave propagating in space. One measures the  pressure on an observation surface.  The aim of PAT is to recover the initial pressure inside the tissue from the measured data. This quantity contains helpful internal information of the object and is the image to be reconstructed.

The mathematical model for PAT is the acoustic wave equation
\begin{eqnarray} \label{E:PAT}
\left\{\begin{array}{l} c^{-2}(x) \, p_{tt}(x,t) -  \Delta p(x,t) =0, \quad (x,t) \in \R^d \times \R_+, \\[6 pt] p(x,0) =f(x), \quad p_t(x,0) =0, \quad x \in \R^d, \end{array} \right.
\end{eqnarray}
where $c \colon \R^d \to \R$ is the sound speed and $f \colon \R^d \to \R$ the initial pressure. Let us denote by $S$ the observation surface. We will assume that $S$ is a closed subset of $\partial \Om$ with nonempty interior $\interior(S)$. Here, $\Om$ is an open subset of $\R^d$ that contains the support of $f$. The mathematical problem of PAT is to invert the map $\Lo\colon f \mapsto g := p|_{S \times (0,T)}$. We will refer to this problem as the inverse source problem of PAT.

For  our conveniences, let us fix several geometric conventions. We will always assume that the sound speed $c$ is smooth and bounded from below by a positive constant. The space $\R^d$ is considered as a Riemannian manifold with the metric $c^{-2}(x) \, dx^2$ and $\Om$ is assumed to be strictly convex with respect to this metric. Then, all the geodesic rays originating inside $\Om$ intersect the boundary $\pd \Om$ at most once. We say that the speed $c$ is nontrapping if all such geodesic rays intersect with $\pd \Om$; otherwise, the  speed $c$ is called trapping. We will denote by $\cT^* \Om$ the cotangent bundle of $\Om$. It can be identified with $\Om \times  \R^d$. Also, $\cT^* \Om \setminus 0$ is the cotangent bundle of $\Om$ minus the zero section, which can be identified with $\Om \times (\R^d \setminus \set{0})$. A set $V \subset \cT^* \Om \setminus 0$ is said to be conic if $(x,\xi) \in V$ implies $(x,\alpha \xi) \in V$ for all $\alpha>0$.

Let us assume that $\supp(f) \subset \Om_0$, where $\Om_0 \Subset \Om$. Then, $\Lo= \Lo_+ + \Lo_-$, where $\Lo_\pm$ are Fourier integral operators (FIOs) of order zero (see \cite[Proposition 3]{US}). This fact, in particular, implies that $\Lo$ is a well-defined linear bounded operator from $H^{s_1}(\Om_0)$ to $H^{s_2}(S \times [0,T])$ for  all $s_1 \geq s_2$. In this article, we will identify the correct mapping spaces for $\Lo$ in order to stabilize the inverse problem of PAT and design proper algorithms. For the sake of simplicity, we will assume that $\Lo$ is injective. Necessary and sufficient conditions for this assumption to hold can be found in \cite[Proposition 2]{US}.

An essential feature of an inverse problem is its well-posedness or ill-posedness (see \cite{engl1996regularization} and Section~\ref{S:IterativeH}).
The inversion of a linear operator $\To$ is called well-posed if the ratio $\|\To x\|/\|x\|$ is bounded from below by a positive constant and ill-posed otherwise. The inverse problem of PAT can be either well-posed or ill-posed, as can be seen in the following two scenarios:
\begin{enumerate}[label=(\roman*)]
\item There is a closed subset $S_0 \subset \partial \Om$ such that $S_0 \subset \interior(S)$ and the following condition holds: for any element $(x,\xi) \in \cT^*\Om_0 \setminus 0$, at least one of the unit speed geodesic rays originating from $x$ at time $t=0$ along the  direction of $\pm \xi$ intersects with $S_0$ at a time $t<T$. This is the so-called {\bf visibility condition} \cite{LQ00,XWAK08,HKN,nguyen2011singularities,US}.
\item There is an open conic set $V \subset \cT^*\Om_0 \setminus 0$  such that for all $(x,\xi) \in V$ none of the unit speed geodesic rays originating from $x$ at time $t=0$ along the direction of $\pm \xi$ intersects with $S$ at a time $t \leq T$. This is called the \emph{\bf invisibility condition}.
\end{enumerate}

The visibility and invisibility conditions are almost, but not exactly, complementary. Under the visibility condition, it is shown in \cite[Theorem 3]{US} that the inversion of $\Lo\colon H^1_0(\Om_0) \to H^1(S \times [0,T])$ is well-posed. In a similar manner, one can the same result for $\Lo\colon L^2(\Om_0) \to L^2(S \times [0,T])$. On the other hand, when the invisibility condition holds, the inversion of $\Lo\colon H^{s_1}(\Om_0) \to H^{s_2} (S \times [0,T])$ is ill-posed for all $s_1,s_2$ (see \cite{nguyen2011singularities}). In this article, we solve the inverse source problem of PAT for both well-posed and ill-posed settings by iterative (regularization) methods. They  include Landweber's, Nesterov's and the conjugate gradient (CG)  methods.  These iterative methods are theoretically convergent to the exact solution in the absence  of noise. However, the convergence speed as well as error estimates with respect to the noise level are significantly different for the ill-posed and the well-posed settings.

There exist several methods to solve  blem of PAT such as explicit
inversion formulas \cite{FPR,Kun07,FHR,Kun07,IPI,Halt2d,Halt-Inv,natterer2012photo,Pal-Uniform}, series solutions \cite{kunyansky2007series,agranovsky2007uniqueness}, time reversal \cite{FPR,HKN,HTRE,US}, and quasi-reversibility \cite{clason2008quasi}. Reviews on these
methods can be found in \cite{HKN,kuchment2014radon,kuchment2008mathematics,RosNtzRaz13}.\footnote{Other setups of PAT that use integrating detectors have been studied in, e.g., \cite{haltmeier2004thermo,paltauf2005thermo,bur6086thermo,zangerl2009circular}. However, we do not consider these setups in the present article.} To the best of our knowledge, the series solutions, time reversal method, and quasi-reversibility methods only apply to the case when $S$ is a closed surface.  Inversion formulas only work for certain closed or flat observation surfaces. 

Let us mention that algebraic reconstruction methods have been frequently used for photoacoustic tomography and achieve high quality images (see, e.g., \cite{DeaBueNtzRaz12,huang2013full,modgil2010image,wang2011imaging,wang2011photoacoustic,wang2012investigation}). For example, in \cite{huang2013full}, the full-wave iterative image reconstruction for photoacoustic tomography with inhomogeneous media was successfully implement and tested with synthetic as well as experimental data. The majority of works on the algebraic methods employ the discretise-then-adjoint approach. In this article we, instead, follow the spirit of \cite{belhachmi2016direct,arridge2016adjoint}, where the adjoint operator was explicitly described in the continuous form. Our goal is to systematically analyze the convergence behavior of iterative methods in a continuous framework.

Our approach  is most closely related to  \cite{belhachmi2016direct}, where Landweber's method was proposed to solve the inverse source problem for PAT. However, we make several changes in order to preserve the well-posedness (when it holds) of the problem and speeds up the convergence speed. For example, instead of considering $\Lo$ as a mapping from $H^1_0(\Om_0)$ to $L^2(S \times [0,T])$, we consider $\Lo$  as a mapping between two Sobolev spaces of the same order  (see Section~\ref{S:Adjoint}). This change not only  preserves the well-posedness of the inverse problem under the visibility condition but also  makes it simpler to compute the adjoint operator and further speed up the convergence. Additionally, by introducing the weighted norm on the image space, we supply the flexibility to the iterative method.  Our choices of  mapping spaces are most similar to those in \cite{arridge2016adjoint}, where  acoustic measurements on an open set are considered (see also Section~\ref{S:Open}, where we derive theoretical results for this setup).  Our adjoint operator is slightly different from that in \cite{arridge2016adjoint} (although they agree on the $C_0^\infty$ framework). Moreover, in addition to the $L^2$-type product (as considered in \cite{arridge2016adjoint}), we also consider the adjoint in the $H^1$-type product. The analysis for this additional case makes it easier to compare the proposed iterative methods with the well-known Neumann series (i.e., iterative time-reversal) method proposed in \cite{US}.

Let us note that our established  iterative algorithms converge linearly for the partial data problem under the visibility condition. This convergence rate, to the best of our knowledge, has not been obtained by any previous  method. Comparable  results, for  a different setup of PAT where the acoustic wave is contained in a bounded domain, have recently been obtained in \cite{Acosta,nguyen2016dissipative,StefanovYang2}.

The article is organized as follows. In Section~\ref{S:Iterative}, we briefly review several iterative methods that  will be used for solving the inverse  source problem of PAT.  As we will see, the knowledge of the adjoint operator is crucial for those iterative methods. In Section~\ref{S:Adjoint} we derive and analyze the adjoint operator $\Lo^*$ of $\Lo$.  We will revisit the PAT with open observation domain in Subsection~\ref{S:Open}.  In Section~\ref{S:Num}, we  describe our numerical implementations and present  results in various scenarios including the full and partial data cases. 

\section{Iterative methods for solving linear equations} \label{S:Iterative}
In this section, we briefly review several common iterative methods to solve linear equations in Hilbert spaces; and how to apply them to the inverse source problem of PAT. 

\subsection{Iterative methods in Hilbert spaces} \label{S:IterativeH} 

Let $\To\colon X \to Y$ be a linear operator mapping between two Hilbert spaces $X$ and $Y$.  We denote by $\mR(\To)$ its range. Assuming that $\To$ is injective, we are interested in inverting $\To$. That is, we want to solve the following problem:
\begin{prob}\label{P:exact} Given $y \in \mR(\To)$, find the solution $x^* \in X$ of the equation $\To x = y$.
\end{prob}

Problem~\ref{P:exact} is said to be well-posed if the inverse $\To^{-1} \colon \mR(\To) \to X$ is bounded and ill-posed  otherwise. It can be seen that Problem~\ref{P:exact} is well-posed if and only if $\To$ is bounded from below, i.e., 
\begin{equation*}
	b  \coloneqq
	\inf\limits_{x \neq 0} \frac{\|\To x\|_{Y}}{\|x\|_{X}}
	= \sqrt{\inf\limits_{x \neq 0}
	\frac{\left< \To^* \To x, x \right>_{X}}{\|x\|^2_{X}} }
 \,	>0.
\end{equation*}

There are several methods to solve Problems~\ref{P:exact}. In this article, we will make use of three methods: Landweber's (see, e.g., \cite[Chapter 6]{engl1996regularization}), Nesterov's (see  \cite{nesterov1983method,nesterov2004introductory}), and the conjugate gradient (CG)  methods (see \cite{hestenes1952methods,hayes1952iterative,daniel1967conjugate,kammerer1972convergence,hanke1995conjugate,fortuna1979some,nevanlinna2012convergence,ernst2000minimal,axelsson2002rate,glowinski2004iterative,mardal2011preconditioning}). As we will see, the knowledge of the adjoint operator is essential for all of these iterative methods.

\begin{itemize}
\item {\bf Landweber's method.} The Landweber's method is simply  the gradient descent method for  minimizing the residual functional $\frac{1}{2} \|\To x- y\|^2$.
 It reads as follows:
\begin{equation*}  \label{E:iterate} x_{k+1}= x_k - \gamma \To^* ( \To x_k - y), \quad k \geq 0,\end{equation*}
where $0 < \gamma  < 2 / \norm{\To}^2$ is a fix relaxation. If the problem is well-posed, the Landweber's method converges linearly. Namely, $\|x_{k}- x^*\|_X \leq  \|\Io - \gamma \To^* \To \|^k \, \|x_0 -x^*\|_X.$

\item Nesterov's fast gradient method. Let $L \geq \|\To\|$ and $\mu \leq b$. The Nesterov's algorithm reads as follows (see \cite[page 80]{nesterov2004introductory}, and also the original paper \cite{nesterov1983method}) :
\begin{enumerate}
\item Initialization: $x_0=z_0$, $\alpha_0 \in [\sqrt{\mu/L},1)$, and $q= \frac{\mu}{L}$.
\item While (not stop) do
\begin{itemize}
\item $x_{k+1} = z_k - \frac{1}{L}  \To^* (\To z_k -y)$
\item compute $\alpha_{k+1} \in (0,1)$ from the equation $\alpha_{k+1}= (1- \alpha_{k+1}) \alpha_k^2 + q \alpha_{k+1}$ 
\item set $\beta_k = \frac{\alpha_k(1- \alpha_k)}{\alpha_k^2 + \alpha_{k+1}}$
\item $z_{k+1}= x_{k+1} + \beta_k (x_{k+1} -x_k)$
\end{itemize}
\end{enumerate}

\noindent Then \cite[Theorem 2.2.3]{nesterov2004introductory} gives  
$$\|\To x_k - y\|^2_Y \leq \min \left \{\left(1- \sqrt{\frac{\mu}{L}} \right)^k, \frac{4L}{(2 \sqrt{L} + k \sqrt{\gamma})^2} \right\} 
\bigl( 
\|\To x_0 - y\|_Y^2 + \gamma \|x_0 -x^*\|^2_X \bigr),$$ where $\gamma = \frac{\alpha_0(\alpha_0 L - \mu)}{1- \alpha_0}$. If the problem is well-posed and $\mu>0$, $x_k$ converges linearly to $x^*$ with the estimate 
$$\|x_k - x^* \|_X \leq \frac{1}{b} \left(1- \sqrt{\frac{\mu}{L}} \right)^k 
\bigl( \|\To x_0 - y\|_Y + \gamma \|x_0 -x^*\|_X \bigr) \,.$$
If $\mu \approx b$, this convergence rate is better than that of Landweber's method (see \cite{nesterov2004introductory}). When the problem is ill-posed ($b=0$) or $b$ is unknown, one may choose $\mu= 0$ (then the calculation of $\beta_k$ can be simplified, e.g., \cite{beck2009fast}). Although the resulting algorithm is not proven to have a linear convergence rate, it still converges faster than the Landweber's method in our numerical simulations (see Section~\ref{S:Num}).

\item{\bf Conjugate gradient (CG) method.} We will use the CG method to solve the normal equation $\To^* \To x = \To^* x$ (see, e.g., \cite[Algorithm 2.3]{hanke1995conjugate}):
\begin{enumerate}
\item Initialization: $k=0$, $r_0= y - \To x_0$, $d_0=\To^* r_0$
\item While (not stop) do
\begin{itemize}
\item $\alpha_k =\| \To^* r_k\|_X^2/ \|\To d_k\|_Y^2$
\item $x_{k+1} = x_k  + \alpha_k \, d_k$
\item $r_{k+1} =  r_k -  \alpha_k \, \To d_k$
\item $\beta_k = \|\To^* r_{k+1} \|_X^2/ \|\To^* r_k\|_X^2$
\item $d_{k+1} = \To^* r_{k+1} + \beta_k \, d_k$
\end{itemize}

\end{enumerate}
The iterates $x_k$ converges to the solution $x^*$ of Problem~\ref{P:exact}.  When the problem is well-posed, the CG method converges linearly with better convergence rate than Landweber's method. Namely, (see, e.g., \cite{daniel1967conjugate})
$$\|x_{k}- x\|_X \leq 2\,  \frac{\|\To\|}{b} \, \Big(\frac{\|\To\| -b}{\|\To\|+b}\Big)^k\,\|x_0 -x\|_X.$$
Moreover, assume that $\To^* \To = \alpha \Io + \Ko$ where $\Ko$ is a compact operator. Then, the CG method is known to converge superlinearly (see, e.g., \cite{herzog2015superlinear}).

\end{itemize}

\subsection{Iterative methods for PAT}  In this article, we will study the above three iterative methods for the inverse  source problem of PAT. To that end, we need to establish the proper form of the adjoint operator (or equivalently, the mapping spaces for $\Lo$), which is done in Section~\ref{S:Adjoint}. Our goal for the adjoint operator is two-fold: it should be relatively simple to implement and
to speed up the convergence. In particular, with our choice of mapping spaces for $\Lo$, the inverse problem of PAT is well-posed under the visibility condition. Therefore, the linear convergence for the Landweber's and CG method is guaranteed for the exact problem (the same for the Nesterov's method if $\mu>0$). This convergence rate, for partial data problem of PAT, has not been obtained before for any other methods. In Section~\ref{S:Num}, we will implement the iterative methods for PAT. We will also compare them with the iterative time reversal method proposed in \cite{US} (see also \cite{QiaSteUhlZha11}).

\section{The adjoint operator for PAT} \label{S:Adjoint}

Let us recall that $\Lo\colon f \mapsto g:=p|_{S \times (0,T)}$, where $p$ is defined by the acoustic wave equation (\ref{E:PAT})
and $S$ is a closed subset of $\pd \Om$ with nonempty interior. Our goal is to invert $\Lo$ using the iterative methods introduced in the previous section.
 It is crucial to analyze the adjoint operator $\Lo^*$ of $\Lo$. To that end, we first need to identify the correct mapping spaces
 for $\Lo$. We, indeed, will consider two realizations, $\Lo_0$ and $\Lo_1$, of $\Lo$ corresponding
 to two different choices of the mapping spaces. To make the presentation rigorous, we will work with several Sobolev spaces and their dual. In particular, we will need to deal with the spaces $L^2(O)$, $H^1(O)$, $H_0^1(O)$, and $H^{-1}(O)$ (the dual of $H_0^1(O)$), where $O$ is an open subset in $\R^d$. The reader is referred to \cite{lions2012non, Evb} for the definition and properties of these spaces. As usual, we will identify the space $L^2(O)$ with its dual $(L^2(O))'$ (see, e.g., \cite[page 31]{lions2012non} or \cite[page 299]{Evb}).

We first recall our assumption $\supp(f) \subset \Om_0$, where $\Om_0 \Subset \Om$. Let us denote
\begin{eqnarray*}
\Hs_0 &\coloneqq& \{ f \in L^2(\R^d) \colon \supp(f) \subset \overline{\Om}_0 \}, \\
\Hs_1 & \coloneqq& \{ f \in H^1(\R^d) \colon \supp(f) \subset \overline{\Om}_0 \}.
\end{eqnarray*}
Then, $\Hs_0$ and $\Hs_1$ are Hilbert spaces with the respective norms
\begin{align*}
	\|f\|_{\Hs_0}  &= \|c^{-1} f\|_{L^2(\Om_0)} \,,\\
	\|f\|_{\Hs_1} & =  \|\nabla f\|_{L^2(\Om_0)}\,.
\end{align*}
We note that $\Hs_0 \simeq L^2(\Om_0)$ and $\Hs_1 \simeq H_0^1(\Om_0)$ (this second equivalence comes from the Poincar\'e inequality, see, e.g., \cite[Corollary 9.19]{brezis2010functional}). The above chosen norms are convenient for our later purposes.

For the image space, we fix a nonnegative function $\chi \in L^\infty(\partial \Om \times [0,T])$ such that $\supp(\chi) = \Gamma:=S \times [0,T]$ and denote
\begin{eqnarray*}
    \Ds_0 &:=& \left\{g \colon \|g\|_{\Ds_0} \coloneqq \|\sqrt{\chi} \, g\|_{L^2(\Gamma)}< \infty \right\}, \\
    \Ds_1 &:=& \left\{g \colon  g(\edot,0) \equiv 0,~ \|g\|_{\Ds_1}  \coloneqq  \|g_t\|_{\Ds_0} < \infty \right \}.
 \end{eqnarray*}
Let $H^i(\Gamma)$ be the standard Sobolev space of order $i$ on $\Gamma$. Noticing that $\Lo$ is a bounded map from $\Hs_i \to H^i(\Gamma)$  (which follows from \cite[Proposition 3]{US}) and $H^i(\Gamma) \subset \Ds_i$, we obtain

\begin{theorem} For $i=0,1$, $\Lo_i := \Lo|_{\Hs_i}$ is a bounded map from $\Hs_i$ to $\Ds_i$.
\end{theorem}

We now consider $\chi \, g$ as a function on $\pd \Om \times [0,T]$ which vanishes on $(\pd \Om \setminus S) \times [0,T]$. The next theorem gives us an explicit formulation of $\Lo_i^*$.

\begin{theorem} \label{T:adjoint} Let $g \in C^\infty(\Gamma)$. 
\begin{enumerate}[leftmargin=4em,label=(\alph*)]
\item Consider the wave equation
\begin{eqnarray} \label{E:adjoint}
\left\{\begin{array}{l} c^{-2}(x) \, q_{tt}(x,t) - \, \Delta q(x,t) =0,
\quad (x,t) \in (\R^d \setminus \partial \Om) \times (0,T), \\[6 pt] q(x,T) =0, \quad q_t(x,T) =0, \quad x \in \R^d, \\[6 pt]
\big[ q \big](y,t) =0, \Big[\frac{\partial q}{\partial \nu} \Big](y,t) =\chi(y,t)  \, g(y,t),  \quad (y,t) \in \partial \Om \times [0,T]. \end{array} \right.
\end{eqnarray}
Then $$\Lo^*_0 g = q_t(\edot ,0)|_{\Om_0}.$$
Here and elsewhere, $[\edot]$ denotes the jump of a function across the boundary $\partial \Om$. That is $$[q] =q_+|_{\partial \Om} - q_-|_{\partial \Om} \quad  \mbox{ and } \quad \left[\frac{ \partial q}{\partial \nu}\right] =\frac{ \partial q}{\partial \nu}\Big|_{\partial \Om} - \frac{ \partial q}{\partial \nu} \Big |_{\partial \Om},$$
where $q_+ \coloneqq q|_{\R^d \setminus \overline \Om}$ and $q_- \coloneqq q|_{\Om}$ have well-defined traces on $\partial \Om$.
\item Assume further that $\chi$ is independent of $t$ (i.e., $\chi(y,t) =\chi(y)$). We define 
\[\bar g(x,t) := g(x,t) - g(x,T) \,, \]
and consider the wave equation
\begin{eqnarray} \label{E:adjoints}
\left\{\begin{array}{l} c^{-2}(x) \, \bar q_{tt}(x,t) - \, \Delta \bar q(x,t) =0,
\quad (x,t) \in (\R^d \setminus \partial \Om) \times (0,T), \\[6 pt] \bar q(x,T) =0, \quad \bar q_t(x,T) =0, \quad x \in \R^d, \\[6 pt]
\big[ \bar q \big](y,t) =0, \Big[\frac{\partial \bar q}{\partial \nu} \Big](y,t) =\chi(y,t)  \, \bar g(y,t),  \quad (y,t) \in \partial \Om \times [0,T]. \end{array} \right.
\end{eqnarray}
Then,
$$\Lo^*_1 g = \Pi [\bar q_t(\edot ,0)],$$
where $\Pi$ is the projection from $H^1(\Om_0)$ onto $\Hs_1 =H_0^1(\Om)$. \end{enumerate}
\end{theorem}
Let us mention that the projection operator $\Pi$ above is given by $\Pi(f) = f -\phi.$ Here, $\phi$ is the harmonic extension of $f|_{\partial \Om_0}$ to $\overline \Om_0$. That is, $\Delta \phi =0$ in $\Om_0$ and $\phi|_{\partial \Om_0} = f|_{\partial \Om_0}$. The proof of Theorem~\ref{T:adjoint} is similar to that \cite[Theorem 1.5]{belhachmi2016direct}. However, to make our presentation rigorous, we have to employ several  results from functional analysis and distribution theory.

\begin{proof}
Let us make use of the weak formulation for (\ref{E:adjoint}) (see Appendix~\ref{A:weak_soln}). Then, for any $v \in C^\infty(\R^n \times \overline \R)$ such that $v(\edot, t) \in C_0^\infty(\R^d)$ for all $t \in \overline \R$, we have
 \begin{multline*} \int_0^T \big(c^{-2}  \, q_{tt}(\edot,t)~ , ~v(\edot,t) \big) \, dt + \int_0^T \int_{\R^d} \nabla q(x,t) \, \nabla v(x,t) \,dx \, dt =
\\ - \int_0^T \int_{\partial \Om} \chi(y,t) \, g(y,t) \, v(y,t) \, dy \, dt.
\end{multline*}
Here, on the first term of the left hand, $(\edot, \edot)$ is the action of a distribution on a test function. Taking integration by parts with respect to $t$ for the first term and with respect to $x$ for the second term of the left hand side, we obtain
\begin{multline}\label{E:eq} - \big(c^{-2} \, q_t(\edot,0) , v(\edot,0) \big)+ \big(c^{-2} \, q(\edot,0) , v_t (\edot,0) \big)  \\ + \int_0^T \int_{\R^d}  q(x,t) \, \big[ c^{-2}(x) \, v_{tt} (x,t) -\Delta v(x,t) \big] \, dx \, dt  \\ = - \int_0^T \int_{\partial \Om} \chi(y,t) \, g(y,t) \, v(y,t) \, dy \, dt.\end{multline}

\noindent{\bf  (a)} Let $f \in C_0^\infty(\R^d)$ and $p$ be the solution of  (\ref{E:PAT}). Choosing $v = p$ in (\ref{E:eq}), we obtain
\begin{equation*}\big(c^{-2}  \, q_t(\edot,0) ,  f \big) = \int_0^T \int_{\partial \Om}  \chi(y,t) \,  g(y,t) \, \Lo(f)(y,t) \, dy \, dt.\end{equation*} 
For any $f \in C_0^\infty(\Om_0)$, the right hand side is bounded by $C\, \|f\|_{L^2(\Om_0)}$ (since $\Lo: L^2(\Om_0) \to L^2(\partial \Om \times [0,T])$ is bounded). Therefore, $c^{-2} \, q_t(\edot,0) \in (L^2(\Om_0))' = L^2(\Om_0)$. Moreover, by the definition of the product in $\Hs_0$ and $\Ds_0$,  we arrive to $\left<q_t(\edot,0) , f \right>_{\Hs_0} = \left< g , \Lo(f)\right>_{\Ds_0}.$
Since this is true for all $f \in C_0^\infty(\Om_0)$, we obtain
$\Lo_0^* g = q_t(\edot,0)|_{\Om_0}.$
This finishes the proof of (a). \\

\noindent{\bf (b)} Let $\bar q$ be the solution of (\ref{E:adjoints}). Then, $\bar q$ satisfies (\ref{E:eq}) with $g$ being replaced by $\bar g$. Assuming $f \in C_0^\infty(\Om)$ and picking $v=p_{tt}$, we obtain
$$\left(c^{-2} \, \bar q_t(\edot,0), \, p_{tt}(\edot,0)\right) = \int_0^T \int_{\partial \Om}  \chi(y,t) \, \bar g(y,t) \, \partial_t^2 \, \Lo(f)(y,t) \, dy \, dt.$$
Noting that $c^{-2} (x) \, p_{tt}(x,0) = \Delta f(x)$, we arrive to
$$\big(\bar q_t(\edot,0), \Delta f \big) =  \int_0^T  \int_{\partial \Om} \chi(y,t) \, \bar g(y,t)  \, \partial_t^2 \, \Lo(f)(y,t) \, dy \, dt.$$
Let us now consider $\chi(y,t) = \chi(y)$. Taking integration by parts for the right hand side with respect to $t$, we get \footnote{Note that $\bar g(\edot,T)  \equiv 0$ and $\partial_t \Lo(f)(\edot, 0) \equiv 0$.}
\begin{equation*} \big( \bar q_t(\edot,0), \Delta f \big)  = - \int_0^T \int_{\partial \Om} \chi(y) \, \bar g_t(y,t) \, \partial_t  \Lo(f)(y,t) \, dy \, dt. \end{equation*}
That is,
\begin{equation}\label{E:H1} \big( \bar q_t(\edot,0), \Delta f \big)  =  - \int_0^T \int_{\partial \Om} \chi(y) \, g_t(y,t) \, \partial_t  \Lo(f)(y,t) \, dy \, dt. \end{equation}
Let us prove that this equation implies $\bar q_t(\edot,0)|_{\Om_0} \in H^1(\Om_0)$. To that end, we fix $\Om'$ such that $\Om_0 \Subset \Om' \Subset \Om$. For any $f \in C_0^\infty(\Om')$, the left hand side equals $\big(\Delta \bar q_t(\edot,0), f \big)$ and the right hand side is bounded by $C \, \|f\|_{H_0^1(\Om')}$ (noting that $\Lo: H_0^1(\Om') \to H^1(\partial \Om \times [0,T])$ is bounded). We obtain $\Delta \bar q_t(\edot,0) \in H^{-1}(\Om')$. Moreover, by the same argument as in (a),  $\bar q_t(\edot,0) \in L^2(\Om')$. Let $\mu \in C^\infty_0(\Om')$ such that $\mu \equiv 1$ on $\Om_0$ and define $\Psi = \mu \, \bar q_t(\edot,0)$. Then, $- \Delta \Psi + \Psi \in H^{-1}(\Om')$ and $\Psi|_{\partial \Om'}=0$. Applying \cite[Theorem 9.1 (Chapter 2)]{lions2012non}, we obtain $\Psi \in H_0^1(\Om')$. Therefore $\bar q_t(\edot,0)|_{\Om_0} = \Psi|_{\Om_0} \in H^1(\Om_0)$.

Now assume $f \in C^\infty_0(\Om_0)$. Taking integration by parts of the left hand side with respect to $x$, we can write (\ref{E:H1}) in the form
\begin{eqnarray*} \int_{\Om_0} \nabla \big[ \bar q_t(x,0)\big]  \,\nabla f(x)  dx = \int_0^T \int_{\partial \Om} \chi(y) \, g_t(y,t) \, \partial_t  \Lo(f)(y,t) \, dy \, dt .\end{eqnarray*}
Let $\phi$ be the harmonic extension of $\bar q_t(\edot,0)|_{\partial \Om_0}$ to $\overline \Om_0$. Then, $\bar q_t(\edot,0) - \phi \in H^1_0(\Om_0)= \Hs_1$ (see, e.g., \cite[Section 2.9]{lions2012non}). Moreover, 
\begin{multline*} \int_{\Om_0} \nabla \big[ \bar q_t(x,0) - \phi(x) \big]  \,\nabla f(x)  dx = \int_{\Om_0} \nabla \big[ \bar q_t(x,0) \big]  \,\nabla f(x)  dx+ \int_{\Om_0} \Delta \phi(x) \,f(x)  dx \\ =  \int_{\Om_0} \nabla \big[ \bar q_t(x,0) \big]  \,\nabla f(x)  dx. \end{multline*}
We obtain
\begin{equation*}\int_{\Om_0} \nabla \big[\bar q_t(x,0) - \phi(x) \big]  \,\nabla f(x)  dx =\int_0^T  \int_{\partial \Om} \chi(y) \, g_t(y,t) \, \partial_t  \Lo(f)(y,t) \, dy \, dt \,.\end{equation*} 
That is, $\left< \Pi [\bar q_t(\edot,0)], f\right>_{\Hs_1} = \left< g, \Lo f\right>_{\Ds_1},$
which proves $\Lo_1^* g = \Pi[\bar q_t(\edot,0)].$ 
\end{proof}

\begin{remark} 
Let us make the following observations:
\begin{enumerate}[leftmargin=4em,label=(\alph*)]
\item Since $C^\infty(\Gamma)$ is dense in both $\Ds_0$ and $\Ds_1$, the adjoint operators $\Lo^*_0$ and $\Lo_1^*$ are uniquely determined by the formulas in Theorem~\ref{T:adjoint}.

\item Equation (\ref{E:adjoint}) can be reformulated as (see \ref{E:reformg}):
\begin{eqnarray} \label{E:reform}
\left\{\begin{array}{l} c^{-2}(x) \, q_{tt}(x,t) - \, \Delta q(x,t) = -  \delta_{\pd \Om}(x)\, \chi(x,t) \,g(x,t) ,~ (x,t) \in \R^d \times(0,T), \\[6 pt] q(x,T) =0, \quad q_t(x,T) =0, \quad x \in \R^d. \end{array} \right.
\end{eqnarray}
This formulation will be used to implement the adjoint operator 
in Section~\ref{S:Num}. 

\item Compared to $\Lo_0^*$, the expression for $\Lo_1^*$ involves an extra projection operator. In our numerical experiments, we will only use $\Lo_0^*$ since it is simpler to implement. However, the knowledge of $\Lo_1^*$ is helpful in designing iterative algorithms that converge in the $\Hs_1$ norm (which is equivalent to the $H^1$-norm). 

\item If, instead of $\Lo_0 \colon \Hs_0 \to \Ds_0$, we consider $\Lo \colon L^2(\Om_0) \to \Ds_0$, then $\Lo^* g = \frac{1}{c^2} q_t(\edot ,0)$. Our formulations of $\Lo^*$ are different from \cite{belhachmi2016direct},  where $\Lo^* g= -\Delta^{-1} (\frac{1}{c^2}
q_t(\edot ,0) )$. Our formulations make the inverse problem of PAT well-posed under the visibility condition (see Theorem~\ref{T:well-posed} below).
\end{enumerate}
\end{remark}

Let us recall the visibility condition described in the introduction:

\begin{vis}
There is a closed subset $S_0 \subset \partial \Om$ such that $S_0 \subset \interior(S)$ and $T_0 \leq T$ such that the following condition holds: for any element $(x,\xi) \in \cT^*\Om_0 \setminus 0$, at least one of the unit speed geodesic rays originating from $x$ at time $t=0$ along the directions $\pm \xi$ intersects with $S_0$ at a time $t < T_0$.
\end{vis}

Let us prove that with our choices of mapping spaces, the inverse problem of PAT is well-posed.
\begin{theorem} \label{T:well-posed} Assume that the visibility condition holds and $\chi >0$ on $S_0 \times [0,T_0]$.
For $i=0,1$, there is a constant $C>0$ such that for any $g= \Lo f$, we have
$$\|f\|_{\Hs_i} \leq C \|g\|_{\Ds_i}.$$
\end{theorem}

\begin{proof} Let us pick a closed subset $S_1$ of $S$ such that $\chi>0$ on $S_1 \times [0,T_0]$, $S_0 \subset \interior(S_1)$, and $S_1  \subset \interior(S)$. Following the lines of  \cite[Theorem 3]{US}, we obtain for $i=0,1$ \footnote{The result for $i=1$ is obtained in that reference. The result for $i=0$ is obtained similarly, one only needs to invoke \cite[Theorem~2.3]{lasiecka1986nonhomogeneous} instead of \cite[Theorem~2.1]{lasiecka1986nonhomogeneous}.}
\begin{equation} \label{E:USestimate} \|f\|_{\Hs_i} \leq C \|g\|_{H^i(S_1 \times [0,T])}.\end{equation}
Here and elsewhere, $H^i(S_1 \times [0,T])$ is the standard Sobolev space of order $i$ on $S_1 \times [0,T]$ and $C$ is a generic constant which may be different in one place from another.

Let us consider $i=0$. Noticing that $\|g\|_{H^0(S_1 \times [0,T])} \leq \|g\|_{\Ds_0}$, we obtain
$$\|f\|_{\Hs_0} \leq C \|g\|_{\Ds_0}.$$

Let us now consider $i=1$. From (\ref{E:USestimate}), we have
\begin{equation}\label{E:H11} \|f\|_{\Hs_1} \leq C \|g\|_{H^1(S_1 \times [0,T])}.\end{equation}
Let us now prove 
\begin{equation} \label{E:est1} \|f\|_{\Hs_1}  \leq C(\|\partial_t g\|_{L^2(S_1 \times [0,T])} + \|g\|_{L^2(S_1 \times [0,T])}).\end{equation}
By using a local chart for $\partial \Om$ if necessary, we can assume without loss of generality that $S_1 =  \R^{d-1}$ and $g=g(x',t)$ is a compactly supported function on $\R^d = \R^{d-1} \times \R$.
Let us denote by $\Hr$ the hyperbolic zone
$$\Hr = \{(x',t,\eta,\tau) \in \cT^* (\R^{d-1} \times \R): c(x') \, |\eta| <\tau\}.$$
Then $\WF(g) \subset \Hr$  (see, e.g., \cite[Proposition 3]{US}). Therefore, the Fourier transform $\hat{g}(\eta,\tau)= \mathcal{F}(g)(\eta,\tau)$ of $g$ decays faster than any powers of $|(\eta,\tau)|$ outside of the region $A:=\{(\eta,\tau): |\tau| \geq c_0 |\eta|\}$, where $c_0=\min_x c(x)$. We denote by $\chi_A$ the characteristic function of $A$ and define
$$\To(f) = \mathcal{F}^{-1}\big(\hat{g}(\eta,\tau) \chi_A(\eta,\tau)\big), \quad \Ko(f) = \mathcal{F}^{-1}\big(\hat{g}(\eta,\tau) \big(1- \chi_A(\eta,\tau)\big).$$
Then $\To$ and $\Ko$ are bounded operators from $\Hs_1$ to $H^1(\R^{d})$ and $\To + \Ko = \Lo$. Since $\mathcal{F}(\Ko f)(\eta,\tau)= \hat g(\eta,\tau) (1- \chi_A(\eta,\tau))$ decays faster than any powers of $|(\eta,\tau)|$, we obtain $\mR(\Ko) \subset H^s(\R^d)$ for any $s>0$. Therefore, $\Ko$ is a compact operator. Moreover, from (\ref{E:H11}), we obtain
$$\|f\|_{\Hs_1} \leq C \Big(\|\To f\|_{H^1(\R^d)} + \|\Ko f\|_{H^1(\R^d)} \Big) \leq  C \Big(\|(\To f, g)\|_{H^1(\R^d) \times L^2(\R^d)} + \|\Ko f\|_{H^1(\R^d)} \Big).$$
Since $f \to (\To f, g)$ is injective, applying \cite[Theorem V.3.1]{Taylor}, we obtain
$$\|f\|_{\Hs_1} \leq C  \, \|(\To f, g)\|_{H^1(\R^d) \times L^2(\R^d)}.$$
We note that
\begin{multline*} \|\To f\|^2_{H^1(\R^d)}= \int_{A} |\hat{g}(\eta,\tau)|^2 (1+ |\eta|^2 + |\tau|^2) d\eta d \tau \leq C \, \int_{\R^d} \, |\hat{g}(\eta,\tau)|^2 \, (1+ |\tau|^2) d\eta \, d \tau  \\ = C \big(\|g\|^2_{L^2(\R^d)} + \|g_t\|^2_{L^2(\R^d)} \big). \end{multline*} 
Therefore,
$$\|f\|_{\Hs_1}  \leq C(\|\partial_t g\|_{L^2(\R^d)} + \|g\|_{L^2(\R^d)}).$$
This finishes the proof of (\ref{E:est1}).

Keeping in mind that $g(\edot,0) \equiv 0$, we obtain from (\ref{E:est1})
$$\|f\|_{\Hs_1}  \leq C \, \|\partial_t g\|_{L^2(S_1 \times [0,T])}.$$
That is, 
$$ \|f\|_{\Hs_1} \leq C \, \|g\|_{\Ds_1},$$
which concludes our proof. \end{proof}

\begin{remark}
Let us recall that (see Section~\ref{S:IterativeH}) when the linear inverse problem is well-posed, Landweber's and the CG methods have a linear rate of convergence. Theorem~\ref{T:well-posed} shows that with our choices of mapping spaces, the inverse problem of PAT is well-posed under the visibility condition. Therefore, Landweber's and the CG methods converge linearly in either $L^2$-norm (i.e., $\Hs_0$-norm) or $H^1$-norm (i.e., $\Hs_1$-norm), depending on our choice of the adjoint operator in Theorem~\ref{T:adjoint}, if the visibility condition holds \footnote{The same conclusion holds for the Nesterov's method if $\mu>0$.}. This convergence rate has not been obtained before by any method. 
\end{remark}

\subsection{Microlocal analysis for the normal operator $\Lo^* \Lo$}

To better understand the nature of $\Lo^*  \Lo$, we will analyze it from the microlocal analysis point of view.
Let us recall that $r_\pm(x,\xi)$ is the (unit speed) geodesic rays originating from $x$ along direction of $\pm \xi$.
 We assume that $r_\pm(x,\xi)$ intersects the boundary $\partial \Om$ at a unique point $x_\pm = x_\pm(x,\xi)$.
 We denote by $\theta_\pm$ the angle between $r_\pm(x,\xi)$ and the normal vector of $\partial \Om$ at $x_\pm$.
 Our main result is the following theorem.

\begin{theorem}\label{T:Micro} Assume that $\chi \in C_0^\infty(\pd \Om \times [0,T])$ and $\Lo = \Lo_0$.
Then, the normal operator
$\No = \Lo^* \, \Lo$ is a pseudo-differential operator of order zero, whose principal symbol is
\begin{equation} \label{E:Symbol} \sigma_0(x,\xi) = \frac{1}{4} \left(\frac{c(x_+) \, \chi(x_+,t_+)}{\cos(\theta_+)} + \frac{c(x_-) \, \chi(x_-,t_-)}{\cos(\theta_-)} \right).\end{equation}
Here, $t_\pm$ is the geodesics distance between $x$ and $x_\pm$.
\end{theorem}
Let us note that it may happen that one (or both) of the geodesic
rays $r_\pm(x,\xi)$ does not intersect $\pd \Om$ (that is, speed is trapping). In that case, Theorem~\ref{T:Micro}
still holds if we replace $\chi(x_\pm, t_\pm)$ by $0$. 
\begin{proof}

We first intuitively describe the effect of $\Lo^* \Lo$ on the wave front set of a function $f$ supported inside $\overline \Om_0$. For simplicity, we assume that $f$ is microlocally supported near an element $(x,\xi) \in \cT^*\Om \setminus 0$. Let us analyze the effect of $\Lo$ to $f$ by considering the wave equation (\ref{E:PAT}). At time $t=0$, the singularity of $f$ at $(x,\xi)$ breaks into two equal parts  (see \cite[page 7]{US}). They induce the singularities of $p$ on the bicharacteristic rays $\mathcal{C}_\pm(x,\xi)$ originating at $(x,0,\xi,\tau = c(x) |\xi|)$ and $(x,0,-\xi,\tau = c(x) |\xi|)$ (see \cite{Ho1}). The projection of each bicharacteristic ray $\mathcal{C}_\pm(x,\xi)$ on the spatial domain $\R^d$ is the geodesic ray $r_\pm(x,\xi)$ on $\R^d$ (recalling that $\R^d$ is equipped with the metric $c^{-2}(x) \, dx^2$). Each of the geodesic ray hits the boundary $\pd \Om$ at a unique point $x_\pm$  and time $t_\pm$. The corresponding singularity of $p$ at $(x_\pm,t_\pm)$ is denoted by $(x_\pm,t_\pm,\xi_\pm,\tau_\pm)$. Its projection on $\cT^*_{(x_\pm,t_\pm)} (\partial \Om \times [0,T])$ induces a singularity of $g$ at $(x_\pm,t_\pm,\eta_\pm,\tau_\pm)$. Now, consider the adjoint equation (\ref{E:adjoint}) which defines $\Lo^*$. The singularity of $g$ at $(x_\pm,t_\pm,\eta_\pm,\tau_\pm)$ then induces two singularities of $q$ at $(x_\pm,t_\pm,\xi',\tau_\pm)$. Here, $\xi' = \eta_\pm  \pm \sqrt{c^{-2}(x_\pm) \tau^2 - |\eta_\pm|^2} \, \nu$ where $\nu$ is the normal vector of $\pd \Om$ at $x_\pm$ (note that one of the such $\xi'$ equals $\xi_\pm$). These two singularities propagate along two opposite directions when going backward in time, one into the domain $\Om$ (along the ray $\mathcal{C}_\pm(x,\xi)$ but in the negative direction) and one away from $\Om$. At $t=0$ the first one lands back to $(x,\xi)$ and the other one lands outside of $\Om$. This shows the pseudo-locality of $\Lo^* \Lo\colon f \to q_t(\edot,0)|_{\Om_0}$ and heuristically explains that $\Lo^* \Lo$ is a pseudo-differential operator. Our rigorous argument follows below.

Let us recall that up to a smooth term (e.g., \cite{TrFour}):
$$p(x,t) = \frac{1}{(2\pi)^d} \sum_{\sg =\pm} \int e^{i \phi_\sg(x,t,\xi)} a_\sg(x,t,\xi) \hat f(\xi) d \xi =: p_+(x,t) +p_-(x,t). $$
The phase function $\phi_\sg$, $\sg=\pm$, satisfies the eikonal equation
$$\partial_t \phi_\sg(x,t,\xi) + \sg |\nabla_x \phi_\sg(x,t,\xi)| =0, \quad \phi_\sg(x,0,\xi) = x \cdot \xi.$$
The amplitude function $a_\sg$ satisfies $$a(x,t,\xi) \sim \sum_{m=0}^\infty a_{-m}(x,t,\xi),$$ where $a_{-m}$ is homogeneous of order $-m$ in $\xi$. The leading term $a_0=a(x,t,\xi)$ satisfies the transport equation\footnote{In several references, the equation contains a zero order term. However, that term turns out to be zero.}
$$\big(\partial_t \phi_\sg  \, \partial_t - c^2(x) \, \nabla_x \phi_\sg \cdot \nabla_x) \, a_0(x,t,\xi) =0,$$ with the initial condition $a_0(x,\xi,0) = 1/2$.

Then, up to a smooth term, we obtain $g= (p_++p_-)|_{\partial \Om} =: g_+ + g_-$. Solving the adjoint problem (\ref{E:adjoint}), we obtain, up to a smooth term,  $\Lo^* \Lo f = \partial_t q_+{(\edot,0)} + \partial_t q_-{(\edot,0)}$. Here, $q_\sg$ (for $\sg= \pm$) is defined by
\begin{eqnarray} 
\left\{\begin{array}{l} c^{-2}(x)\,  q_{\sg, tt}(x,t) - \, \Delta q_\sg(x,t) =0, \quad (x,t) \in (\R^d \setminus \partial \Om) \times (0,T), \\[6 pt] q_\sg(x,T) =0, \quad q_{\sg,t}(x,T) =0, \quad x \in \R^d, \\[6 pt]  \big[q_{\sg} \big](y,t) =0, \Big[\frac{\partial q_\sg}{\partial \nu} \Big](y,t) =\chi(y,t) \, g_\sg(y,t), \quad (y,t) \in  \partial \Om \times (0,T). \end{array} \right.
\end{eqnarray}
Let us show that $f \to f_+ := \partial_t q_+{(\edot,0)}$ is a pseudo-differential operator with the principal symbol
$$\sg_+(x,\xi) = \frac{1}{4} \frac{\chi(x_+,t_+)}{\cos(\theta_+)}.$$
We recall that $x_+$ is the intersection of the positive geodesic ray $r_+(x,\xi)$ and $\partial \Om$, and $t_+$ is the time to travel along the geodesic from $x$ to $x_+$. Let $(x_+,t_+, \xi_+, \tau_+)$ be the corresponding element on the bicharacteristic and $(x_+,t_+,\eta_+,\tau_+)$ its projection on $\cT^*_{(x_+,t_+)} (\pd \Om \times [0,T])$. Then, $(x_+,t_+,\eta_+,\tau_+)$ is in the hyperbolic zone, that is $\tau_+ > c(x_+) \, |\eta_+|$. Let us show that the mapping $q_+|_{\partial \Om \times [0,T]} \to [\pd_\nu q_+]|_{\partial \Om \times [0,T]}$ is an elliptic pseudo-differential operator near $(x_+,t_+,\eta_+, \tau_+)$.

Indeed, for simplicity, we assume that locally near $x_+$, $\partial \Om$ is flat, and $h= q_+|_{\partial \Om \times [0,T]}$ is supported near $(x_+,t_+)$. We then can write $y=(y',0)$ for all $y \in \partial \Om$ and assume that $\Om \subset \{ x \in \R^d \colon x_n <0\}$. The parametrix $q_{\rm in}$ ({respectively} $q_{\rm out}$) for $q_+$ in $\Om$ ({respectively} $\Om^c$) near $(x_+,t_+)$ is of the form
\begin{equation} \label{E:para} q_{in/out} (x,t) = \frac{1}{(2 \pi)^d}  \sum_{s=F,B} \int_\R \int_{\R^{d-1}}  e^{i \, \psi_{s}(x,t,\eta,\tau)} d_s(x,t,\eta,\tau) \hat{h}(\eta,\tau) \, d\eta \, d\tau.\end{equation}
Here,
\begin{eqnarray*} d_F (y,t,\eta,\tau) + d_B (y,t,\eta,\tau)= 1,\quad \psi_s(y,t,\eta,\tau) = y' \cdot \eta - t \tau, \quad y \in \partial \Om,
\end{eqnarray*}
and
$$\hat{h}(\eta,\tau) = \int_{\R} \int_{\R^{d-1}} h(y',0,t) e^{i(- \eta \cdot y + t \tau)} \, dy' \, dt.$$
Similarly to $\phi_+$, the phase function $\psi = \psi_{F,B}$ satisfies the eikonal equation
$$|\partial_t \psi(x,t,\eta,\tau)| = c(x) \, |\nabla_x \psi(x,t,\eta,\tau)|.$$
In particular, we obtain
\[\partial_{x_n} \psi_F(y,t,\eta,\tau) = \sqrt{c^{-2}(y) \, \tau^2 -\eta^2}, \quad \partial_{x_n} \psi_B(y,t,\eta,\tau) = - \sqrt{c^{-2}(y) \, \tau^2 -\eta^2}, \quad y \in \pd \Om.\]
Roughly speaking, the phase function $\psi_F$ transmits (forward) wave from left to right (along the $x_n$ direction) and $\psi_B$ transmits (backward) wave to the opposite direction. Since $q(x,T)=q_t(x,T) =0$ for all $x \in \R^d$, we obtain that there is no backward wave inside $\Om$ and no forward wave outside $\Om$. That is,
\[q_{\rm in}(x,t) = \frac{1}{(2 \pi)^d}  \int_{\R} \int_{\R^{d-1}} e^{i \psi_F(x,t,\xi)} d_{\rm in} (x,t,\eta,\tau) \hat h(\eta,\tau) d \eta \, d \tau,\] and
\[q_{\rm out}(x,t) = \frac{1}{(2 \pi)^d}  \int_{\R} \int_{\R^{d-1}}   e^{i \psi_B(x,t,\xi)} d_{\rm out} (x,t,\eta,\tau) \hat h(\eta,\tau) d \eta \, d\tau. \]
Moreover, \[d_{\rm in} (y,t,\eta,\tau)=d_{\rm out} (y,t,\eta,\tau)=1,\quad y \in \pd \Om.\] Up to lower order terms, we obtain micirolocally near the hyperbolic element $(x_+, t_+,\eta_+,\tau_+)$
\[ [\partial_{x_n} q_+](y,t) = \frac{1}{(2 \pi)^d}  \int_{\R} \int_{\R^{d-1}}   e^{i (y' \cdot \eta - t \tau)} (-2 i) \sqrt{c^{-2}(y) \, \tau^2 -\eta^2} \, \hat h(\eta,\tau) d \eta \, d\tau.\]
That is, the mapping $q_+|_{\pd \Om} \to [\partial_{\nu} q_+]$ is an elliptic pseudo-differential operator at the element $(x_+, t_+, \eta_+, \tau_+)$  with principal symbol $ (-2 i) \sqrt{c^{-2}(x_+) \, \tau_+^2 -\eta_+^2} $. Therefore, the mapping $ [\pd_\nu q_+]_{\partial \Om \times [0,T]} \to q_+|_{\partial \Om \times [0,T]}$ is also a pseudo-differential operator near $(x_+,t_+,\eta_+, \tau_+)$ with the principal symbol $\frac{1}{(-2 i) \sqrt{c^{-2}(x_+) \, \tau_+^2 -|\eta_+|^2}}$.

Noting that $f \to \chi g_+$ and $q_+|_{\pd \Om \times [0,T]} \to q_+( \edot,t)|_{\Om}$ are FIOs, we obtain $f \to q_+(\edot,t)|_{\Om}$ is also an FIO. We, hence, can write the parametrix for $q_+$ in $\Om$ in the form
\[q_+(x,t) = \frac{1}{(2 \pi)^d}\int_{\R^d} e^{i \phi_+(x,t,\xi)} b(x,t,\xi) \hat f(\xi) d \xi. \]
In particular,
\begin{equation} \label{E:qplus} q_+(y,t) = \frac{1}{(2 \pi)^d}\int_{\R^d} e^{i \phi_+(y,t,\xi)} b(y,t,\xi) \hat f(\xi) d \xi, \quad y \in \pd \Om.\end{equation}
On the other hand,
\[\chi(y,t) \, g_+(y,t) = \frac{1}{(2 \pi)^d} \int_{\R^d} e^{i \phi_+(y,t,\xi)} \chi(y,t) \, a(y,t,\xi) \hat f(\xi) d \xi, \quad y \in \pd \Om.\]
Since $\chi \, g_+=[\pd_\nu \, q_+] \to q_+$ is a pseudo-differential with principal symbol $\frac{1}{(-2 i ) \, \sqrt{c^{-2}(x_+) \, \tau_+^2 -|\eta_+|^2}}$ at $(x_+,t_+,\eta_+,\tau_+)$, the principal part $b_p$ of $b$ satisfies
\[b_p(x_+,t_+,\xi_+) =- \frac{\chi(x_+,t_+) }{2 \, i \, \sqrt{c^{-2}(x_+) \tau_+^2 - |\eta_+|^2}} \, a_0(x_+,\xi_+,t_+).\]
Since $b_{\rm in}$ and $a_0$ satisfy the same transport equation on the geodesic ray $r_+(x,\xi)$, the above equation implies
\[b_{p}(x,0,\xi) = - \frac{\chi(x_+,t_+) }{2 \, i \, \sqrt{c^{-2} \tau_+^2 - |\eta_+|^2}} \, a_0(x,0,\xi)= - \frac{\chi(x_+,t_+)}{4 \, i \, \sqrt{c^{-2} \tau_+^2 - |\xi_+|^2}}.\]

Noting that $\partial_t \phi_+(x,0,\xi) = - c(x) |\nabla_x \phi_+(x,0,\xi)| = - c(x) \, |\xi|$,
we obtain,  from (\ref{E:qplus}), up to lower order terms,
$$f_+(x) = \frac{1}{(2 \pi)^d}\int_{\R^d} e^{i x \cdot \xi} \frac{ c(x) |\xi| \, \chi(x_+,t_+)}{4 \, \sqrt{c^{-2}(x_+) \tau_+^2 - |\xi_+'|^2}}  \hat f(\xi) \, d \xi.$$
Noting that $c(x) |\xi| =c(x_+) |\xi_+|$,\footnote{This comes from the fact that $c(x) |\xi| =\tau$ and $\tau$ is constant on the bicharacteristic rays (see, e.g., \cite{nguyen2011singularities}).} we obtain the mapping $f \to f_+$ is a pseudo-differential operator with principal symbol
$$\frac{c(x) |\xi| \, \chi(x_+,t_+)}{4  \sqrt{c^{-2}(x_+) \tau_+^2 - |\eta_+|}} = \frac{  c(x_+) \, |\xi_+|  \, \chi(x_+,t_+)}{4 \sqrt{|\xi_+|^2- |\eta_+|}}  = \frac{c(x_+) \, \chi(x_+,t_+)}{4 \cos (\theta_+)}. $$

Repeating the above argument for $f \to f_-$,  we finish the proof.
\end{proof}

\begin{remark} Let us make the following observations:
\begin{enumerate}[leftmargin=4em,label=(\alph*)]
\item
The calculus of symbols can be explained more intuitively by considering the current set up as the limit of the open set measurement. This will be discussed in Section~\ref{S:Open}.

\item We notice that the function $\chi$ plays the role of preconditioning for the inverse problem of PAT. Formula (\ref{E:Symbol}) may give us some hint on how to make a good choice of $\chi$. Indeed, let us consider the case $c=1$, $S$ is the sphere of radius $R$, and $\chi(y,t)=t$. Since the geodesics are straight lines, we observe that $\theta_+=\theta_-=\theta$ and hence
$$\sg_0(x,\xi) = \frac{t_+ + t_-}{4 \, \cos \theta}.$$
We notice that $t_++ t_-$ is the length of the line segment connecting $x_+$ and $x_-$. A simple geometric observation then gives:
$$\sg_0(x,\xi) =\frac{R}{2}.$$
We obtain
\begin{equation} \label{E:Finch} \Lo^* \Lo = \frac{R}{2} \, \Io + \Ko, \end{equation} where $\Ko$ is a compact operator. Therefore, in such a situation, the CG method for the inverse problem of PAT converges superlinearly (see the discussion of the CG method in Section~\ref{S:IterativeH}).
We note that (\ref{E:Finch}) can be derived from the results in \cite{FPR,FHR}. Indeed, for odd $d$, \cite{FPR} even gives $\Lo^* \Lo = \frac{R}{2} \Io$.

The above discussion also suggests the choice of $\chi(y,t)=t$ when the speed is almost constant. The in-depth discussion on the preconditioning, however, is beyond the scope of this article. 

\end{enumerate}
\end{remark}

Let us recall the time-reversal technique for PAT (see, e.g., \cite{FPR,HKN,US}). Consider the time reversal wave equation
\begin{eqnarray*}
\left\{\begin{array}{l}c^{-2}(x) \,  q_{tt}(x,t) -  \Delta q(x,t) =0, \quad (x,t) \in \Om \times [0,T],  \\[6 pt] q(x,T) =\phi(x), \quad q_t(x,T) =0, \quad x \in \Om, \\[6 pt]  q(x,t)= \chi(x,t) \, g(x,t), \quad (x,t) \in  \partial \Om \times [0,T]. \end{array} \right.
\end{eqnarray*}
Here, $\phi$ is the harmonic extension of $\chi(x,T) \, g(x,T)|_{x \in \partial \Om}$ to $\overline \Om$. The time-reversal operator is defined by $\Lambda g = q{(\edot,0)}$. It is proved in  \cite[Theorem 1]{US} that, if $\chi \equiv 1$,  $$\|\Io - \Lambda \Lo\|_{} <1.$$
This suggest that $\Lambda \Lo$ can be used as the first step for a iterative method  (see  \cite[Theorem 1]{US}); it is called iterative time reversal method or Neumann series solution (see also \cite{QiaSteUhlZha11} for the thorough numerical discussion and  \cite{stefanov2011thermo} for nonsmooth sound speed).
It is shown in \cite{US} that $\Lambda \Lo$ is a pseudodifferential operator of order zero with the principal symbol
$$\sigma_0(x,\xi) = \frac{1}{2}\big(\chi(x_+,t_+)  +\chi(x_+,t_-)\big).$$
This is different from the symbol of $\Lo^* \Lo$ shown in Theorem~\ref{T:Micro}. We, in particular, conclude that the adjoint operator $\Lo^*$ is fundamentally different from the time reversal operator $\Lambda$. We also note that no proof for the convergence of iterative time reversal method is available for limited data problem, even under the visibility condition.
However, numerically it works reasonably well in this situation (as demonstrated in \cite{QiaSteUhlZha11} and also Section~\ref{S:Num}).

\subsection{Open domain observations revisited} \label{S:Open}

Let us consider the setup used in \cite{arridge2016adjoint}.
Namely, let us consider the operator $\Lo_\omega$ defined by
$$(\Lo_\omega f)(x,t) = \omega(x,t) \, p(x,t). $$

Here, $p$ is the solution of \eqref{E:PAT} with initial pressure $f$, and
$0 \leq \omega \in C^\infty(\R^d \times [0,T])$ is the window function, whose support determines the accessible region for the data.
We assume that $\supp  (\omega) = \mB \times [0,T]$, where $\mB \subset \R^d$ is an open band.
That is, $\R^d \setminus \overline \mB= \Om \cup \Om'$, where $\Om \cap \Om' =\emptyset$, $\Om$ is bounded and $\Om'$ is unbounded.
We again, assume that $f$ is supported in $\overline \Om_0$ where $\Om_0 \Subset \Om$, and are interested in the problem of finding $f$ given $\Lo_\omega f$.
It can be solved by the iterative methods described in Section~\ref{S:IterativeH},
which we do not elaborate further in this article.  We, instead, focus on analyzing the adjoint operator in this setup.

We propose the following method to compute the adjoint of $\Lo_\omega$.  Consider the time-reversed problem
\begin{eqnarray} \label{E:Adjoint}
\left\{\begin{array}{l}  c^{-2}(x) \, q_{tt}(x,t) -  \Delta q(x,t) = -  \omega(x,t) \, h(x,t), \quad (x,t) \in \R^d \times (0,T), \\[6 pt] q(x,T) =0, \quad q_t(x,T) =0, \quad x \in \R^d. \end{array} \right.
\end{eqnarray}
We define
$$\Lo^*_\omega(h) = \partial_t q(\edot,0).$$
Let $\bar h(\edot, t) = h(\edot,t) - h(\edot,T)$ and $\bar q$ be the solution of (\ref{E:Adjoint}) with $h$ being replaced by $\bar h$. We define (recalling that $\Pi$ denotes the projection from $\Hs_1$ 
onto $H^1_0(\Om_0)$)
$$\overline \Lo^*_\omega(h) = \Pi[ \bar q_t(\edot,0)].$$

Let us recall the space $\Hs_i$ defined at the beginning of Section~\ref{S:Adjoint}. Similarly to the spaces $\Ds_i$, 
we define
\begin{eqnarray*}
\widetilde \Ds_0 &&:= \left\{ h \colon  \|h\|_{\widetilde \Ds_0} := \|\sqrt{\omega} \, h \|_{L^2(\mB \times [0,T])} < \infty\right\}, \\
\widetilde \Ds_1 &&:= \left\{h \colon h(\edot,0) \equiv 0 \mbox{ in } \mB,~ \|h\|_{\widetilde \Ds_1} := \|h_t\|_{\widetilde \Ds_0}< \infty \right\}.\end{eqnarray*}
The following lemma shows that $\Lo^*_\omega$ and $\overline \Lo_\omega^*$ are the adjoints of $\Lo_\omega$, given the correct mapping spaces. 
\begin{theorem}\label{T:adjointopen}
We have
\begin{enumerate}[leftmargin=4em, label=(\alph*)]
\item  For all $f \in H_0(\Om)$ and $h \in \widetilde \Ds_0$,
$$\left<\Lo_\omega f, h \right>_{\widetilde \Ds_0} = \left<f,\Lo_\omega^* 
h\right>_{\Hs_0}.$$
That is, $\Lo_\omega^*$ is the adjoint of $\Lo_\omega \colon \Hs_0 \to \widetilde \Ds_0$.
\item  Assume that $\omega$ is independent of $t$, then for all $f \in H_1(\Om)$ and $ h \in \widetilde \Ds_1$,
$$\left<\Lo f, h \right>_{\widetilde \Ds_1} = \left<f,  \overline \Lo_\omega^*  h \right>_{\Hs_1}.$$
That is, $ \overline \Lo_\omega^*$ is the adjoint of $\Lo_\omega: \Hs_1 \to \widetilde \Ds_1$. 
\end{enumerate}
\end{theorem}
The proof of Theorem~\ref{T:adjointopen} is similar to that of Theorem~\ref{T:adjoint}. We skip it for the sake of brevity. Let us notice that our definition of $\Lo_\omega^*$ is slightly different from \cite{arridge2016adjoint}. It is motivated by case of the observation on a surface discussed the previous section. Our definition matches with that in \cite{arridge2016adjoint} if $ \omega(\edot,T) \,  h (\edot,T) \equiv 0$.

The following theorem gives us a microlocal characterization of the normal operator $\Lo_\omega^* \Lo_\omega$.

\begin{theorem}
The operator $\Lo_\omega^* \Lo_\omega$ is a  pseudo-differential operator of order zero whose principal symbol is
$$\sg_0(x,\xi) = \frac{1}{4} \Big(\int_0^T c^2(x_+(t)) \, \omega(x_+(t),t) dt +  \int_0^T c^2(x_-(t))  \, \omega(x_-(t),t) dt \Big).$$
Here, $x_\pm(t)=r_\pm(x,\xi)(t)$ is unit speed geodesic ray originated from $x$ at time $t=0$ along the direction of $\pm \xi$.
\end{theorem}

\begin{proof}

Let us consider the solution $q$ of the time reversed problem (\ref{E:Adjoint}) with $h=p|_{\mB \times [0,T]}$. Applying the Duhamel's principle, we can write
$$q(\edot,t) = \int_t^T q(\edot,t; s) ds,$$
where $q(x, t;s)$ satisfies
\begin{eqnarray*} 
\left\{\begin{array}{l} c^{-2}(x) \, q_{tt}(x, t;s) -  \Delta q(x,  t;s) = 0, \quad (x,t) \in \R^d \times (0,s), \\[6 pt] q(x, s;s) = 0, \quad q_t(x, s;s) = c^2(x) \, \omega (x,s) \, p (x,s), \quad x \in \R^d. \end{array} \right.
\end{eqnarray*}
Therefore, denoting $q'(\edot, t;s) = q_t(\edot, t;s)$,
$$\Lo_\omega^* (p) =  q_t(  \edot,0) = \int_0^T q'(  \edot,0;  s) ds - q( \edot,0;0) =\int_0^T q'(  \edot,0;  s) ds = \int_0^T  \Lo(s)(p) \, ds.$$
Let us show that $f \to \Lo(s)(p) \coloneqq q'(\edot,0;s)$ is a pseudo-differential operator with the principal symbol
$$\sigma_0(x,\xi;s) = \frac{1}{4}\big(c^2( x_+(s)) \, \omega( x_+(s),s) + 
c^2( x_+(s)) \, \omega(x_-  (s),s)\big).$$
Indeed, we note that $q'(x, t;s)$ satisfies
\begin{equation} \label{E:tau'} \left\{\begin{array}{l} c^{-2}(x) \, q'_{tt}(x, t;s) -  \Delta q'(x, t;s) = 0, \quad (x,t) \in \R^d \times (0,s), \\[6 pt] q'(x, s;s) = c^2(x) \, \omega  (x,s) \, p (x,s), \quad q'_t(x, s;s) = 0, \quad x \in \R^d. \end{array} \right. \end{equation}
Assume that $(x,\xi) \in (\cT^* \Om \setminus 0) \cap \WF(f)$ and consider the propagation of  $p$, governed by the wave equation (\ref{E:PAT}). The singularity of $f$ at $(x,\xi)$ is broken into two equal parts propagating along the geodesic rays $r_\pm(x,\xi)$. Let us consider the propagation along $r_+(x,\xi)$. The projection of the propagated singularity at $ t=s$ to $\cT^*\R^d$ produces a corresponding singularity of $p(.,s)$. Let us consider the propagation of that singularity due to the equation (\ref{E:tau'}). Firstly, due to the end time condition at $ t=s$, it is multiplied by $c^2(x) \, \omega( x(s),s)$. Then, it is broken into two equal parts propagating along two opposite directions (in reversed time). One of them hits back to $(x,\xi)$ at $t=0$ (this travel along $r_+(x,\xi)$ but in negative direction) and the other one lands outside of $\Om$. Therefore, the strength of the recovered singularity at $(x,\xi)$, as just described, is $\frac{1}{2} c^2( x_+(s)) \, \omega( x_+(s),s)$ times that of the part of original singularity at $(x,\xi)$ propagating along the ray $r_+(x,\xi)$. Similar argument for the negative ray $r_-(x,\xi)$ gives us the second recovered singularity with the magnitude $\frac{1}{2} c^2( x_+(s)) \, \omega( x_+(s), s)$ of the part of original singularity at $(x,\xi)$ propagating along the negative ray $r_-(x,\xi)$. Since each part (propagating on each direction) is half of the original singularity, we obtain the recovered singularity is  $\frac{1}{4} \big(c^2( x_+(s)) \,  \omega( x_+(s),s) + c^2( x_-(s)) \, \omega(x_-  (s),s)\big)$ of the original singularity. This intuitively, shows that $\Lo(s) \,  \Lo_\omega$ is a pseudo-differential operator of order zero with the principal symbol $$\frac{1}{4} \Big(c^2( x_+(s)) \, \omega( x_+(s),s) + c^2( x_-(s)) \, \omega(x_- (s),s)\Big).$$
A more rigorous proof can be done by writing the corresponding form of the parametrix for the wave equations (\ref{E:PAT}) and (\ref{E:tau'}). However, we skip it for the sake of simplicity.

Now, since $\Lo^*_\omega \, \Lo_\omega = \int_0^T \Lo(s) \, \Lo_\omega \, ds$,  we obtain that $\Lo^*_\omega \, \Lo_\omega$ is a pseudo-differential operator of order zero with the symbol
 $$\frac{1}{4} \int_0^T \Big(c^2(x_+(s)) \,  \omega(x_+(s),s) + c^2( x_-(s)) \,  \omega(x_- (s),s)\Big) ds.$$
This finishes the proof of the theorem. \end{proof}

\begin{remark} Let us consider $\omega= \omega_\eps$ to be a family of smooth function that approximate the $\chi(x,t)\, \delta_{\partial \Om}(x)$. Then, the setup for the observation on the surface $\partial \Om$ is just the limit as $\eps \to 0$. We note that
$$\lim_{\eps \to 0} \int_0^T c^2( x_\pm(s)) \,  \omega( x_\pm(s), s) ds = \frac{c(x_\pm) \, \chi(x_\pm,t_\pm)}{\cos \theta_\pm}. $$
Therefore,
$$\lim_{\eps \to 0}\sg_0(x,\xi) = \frac{1}{4} \left(\frac{c(x_+) \, \chi(x_+,t_+)}{\cos(\theta_+)} + \frac{c(x_-) \, \chi(x_-,t_-)}{\cos(\theta_-)} \right),$$
which is the symbol $\Lo^* \Lo$ in Theorem~\ref{T:Micro}.
\end{remark}
\section{Numerical experiments}
\label{S:Num}

In this section we  implement the iterative methods presented in Section~\ref{S:IterativeH} for PAT for the observation on a surface $\partial \Om$. We will employ the explicit formulation of $\Lo^* = \Lo_0^* $ presented in Section~\ref{S:Adjoint}. We will chose the weight function $\chi$ to be independent of the time variable $t$.\footnote{Other choices of $\chi$ may result in better conditioning of the problem. However, studying optimal preconditioning is beyond the scope of this article.}  We note that under the visibility condition,  Landweber's and the CG  methods have linear convergence
in $H_0(\Om_0) = L^2(\Om_0)$, since by Theorem~\ref{T:well-posed} the inversion of $\Lo f =g$ is well-posed in this situation.  If using $\Lo^* =\Lo_1^*$, we would obtain the linear rate of convergence in  $\Hs_1   \simeq H^1(\Om_0)$. However, we will refrain from that choice.

We only consider two-dimensional simulations and assume $\Om$ to be a disc centered at the origin:
$\Om = B_R(0) = \set{   x \in \R^2 \colon  \lvert x \rvert  <  R} $. All presented results assume  non-constant sound speed.
We consider the following test cases
\begin{enumerate}[label=(T\arabic*)]
\item\label{t1}  Complete data;
\item\label{t2} Partial data, visible phantom (i.e., the visibility condition holds);
\item\label{t3} Partial data, invisible phantom (i.e., the invisibility condition holds);
\end{enumerate}
We will compare the results for the Landweber's method, Nesterov's method,  the CG method
(as proposed in the present paper) as well as the  iterative time reversal  algorithm proposed
in~\cite{QiaSteUhlZha11}.  Thereby we investigate the numerical speed of  convergence as well as stability and accuracy of all these algorithms. For Landweber's  and Nesterov's method we have taken the step size equal to $\gamma =1$, which worked well in all our numerical simulations.

As described in Subsection~\ref{sec:realization} the proposed iterative schemes are implemented by numerical
realizations of all involved operators. Thereby the most crucial steps  are accurate discrete solvers
for the forward and backward wave equation. For that purpose we implemented the $k$-space method
(described in Appendix~\ref{sec:kspace})
that is  an efficient FFT based numerical solution method  that does not suffer from numerical dispersion
that arises when solving the wave equation with standard finite difference or finite element methods. We note that to compute the adjoint operator $\Lo^*$ numerically, we make use of the formulation (\ref{E:reform}).

 \begin{figure}[tbh!]\centering
\includegraphics[width =0.7\textwidth]{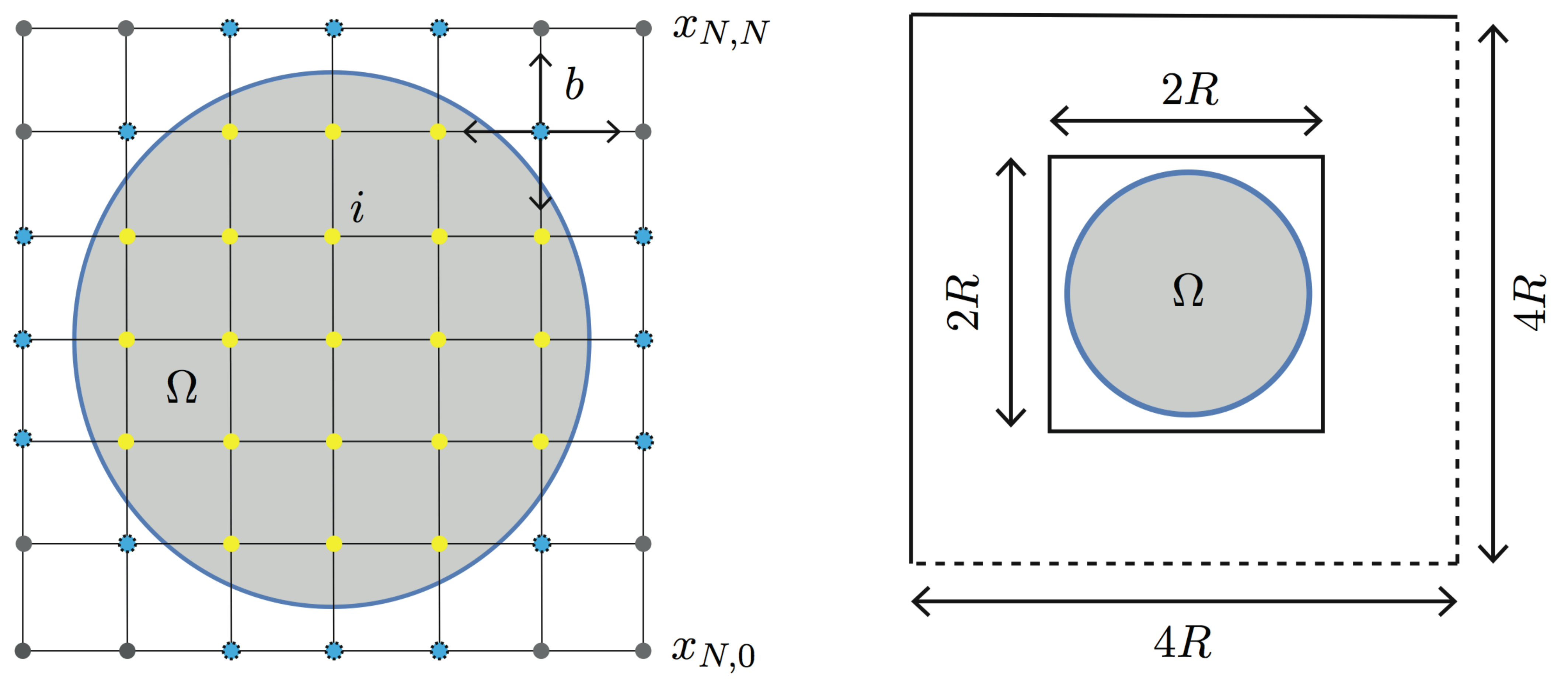} \caption{Left: The discrete domain $\Om_N$ is defined as the set of all indices $i \in \set{0, \dots, N}^2$ with $x_i \in \Omega$.
The index $b$ is contained in the discrete boundary $\partial \Om_N$, because one of
its  neighbors is contained in $\Om_N$. Right: The domain $[-R,R]$ is embedded in a larger domain  $[-2R,2R]$  to avoid
effects due to periodization.\label{fig:discrete}}
\end{figure}

\subsection{Numerical realization}
\label{sec:realization}

The iterative approaches for solving the equation $\Lo f=g$ are implemented with
discrete counterparts  of all operators introduced above. Thereby the function $f \colon \R^2 \to \R$
is represented by a discrete vector
\begin{equation*}
 \fnum =  (f(x_{i}))_{i_1, i_2 = 0}^{N} \in \R^{(N+1) \times (N+1)} \,,
\end{equation*}
where $x_i = (-R,-R) + 2i R/N$  for $i  = (i_1, i_2)\in \set{0, \dots, N}^2$ are equidistant grid points in the square
$[-R,R]^2$. We define the discrete domain $\Om_N \subset   \set{0, \dots, N}^2$ as the set of
all indices $i$ with $x_i \in \Omega$. Further, the discrete measurements are made on parts of the
discrete  boundary $\partial \Om_N$, that is  defined as the set of all elements
$b = (b_1,b_2)  \in \set{0, \dots, N}^2 \setminus \Om_N$   for which at least one of the discrete neighbors $(b_1+1,b_2), (b_1-1,b_2), (b_1,b_2+1), (b_1,b_2-1)$
 is  contained  in $\Om_N$, see the right image in Figure~\ref{fig:discrete}. All phantoms in our numerical simulations
 are chosen to have support  in a compact subset $\Om_0$ of $\Om$. We will choose the discrete version of $\Om_0$ to be the set $\{x_i \colon  i \in \Om_N\}$.

The discrete forward operator can be written in the form
 \begin{equation} \label{eq:fwdN}
	 \Lnum_{N,M}
	 \colon \R^{(N+1) \times (N+1) }  \to \R^{ |\partial \Om_N| \times (M+1)}
	 \colon \fnum \mapsto (\Num_{N,M} \circ \Wnum_{N,M}) \fnum \,.
 \end{equation}
Here $M+1$ is the number of equidistant temporal sampling points  in  $[0,T]$,
$\Wnum_{N,M} $ is a discretization of the  solution operator for the wave equation and
$\Num_{N,M} $ the linear operator that restricts the discrete pressure to spatial grid points
restricted to $\partial \Om_N  \subset \set{0, \dots, N}^2$. The adjoint operator is then
given by $\Lnum_{N,M}  = \Dnum_0 \circ  \Wnum_{N,M}^* \circ \Num_{N,M} ^*$, where
$\Num_{N,M}^*$ is the embedding operator from $\R^{|\partial \Om_N|}$ to
$\R^{(N+1) \times (N+1)}$, $\Wnum_{N,M}^*$  is the solution operator the adjoint wave equation \eqref{eq:wave2ad}, and $\Dnum_0$
a discretization of the time derivative evaluated at $t=0$  and restricted to  $\Om_0$.

For computing the  solution operator  $\Wnum_{N,M}$  we use the $k$-space method
described in Appendix~\ref{sec:kspace}. In the actual implementation of Algorithm~\ref{alg:kpace},
the Fourier transform of the function  $f$ is replaced by the FFT algorithm applied to $\fnum$
(and likewise for the  inverse Fourier transform). When  applied directly to the given
function values, the FFT algorithm causes the numerical solution to be $2R$ periodic.
For the numerical solution of the wave equation  we therefore  embed the data vector  $\fnum \in \R^{(N+1)\times (N+1)} $ in a larger vector in $\R^{(2N+1)\times (2N+1)}$, whose entries
correspond to sampled values on an equidistant grid in $[-2R, 2R]^2$ (see the right image in
Figure~\ref{fig:discrete}).
As the sound speed is assumed to be equal to one outside of $\Om$, the numerical  solution
for times $t \leq 2R$  (which will  always be the case in our simulations)  is free from
periodization artifacts in the domain $\Om$ and on the measurement  surface.

\begin{remark}[Numerical complexity of $k$-space based iterative algorithms]
Using the FFT algorithm any time step in the $k$-space method 
(summarized in   Algorithm~\ref{alg:kpace}) can be implemented
using  $\mathcal O(N^2 \log N)$ floating point operations (FLOPS).  Performing $M \sim N$
time steps therefore yields to $\mathcal O(N^3 \log N)$ algorithms for implementing the forward
operator $\Lnum_{N,M}$ and its adjoint  $\Lnum_{N,M}^*$.  Consequently, performing one iterative step (for example using the CG or the Landweber iteration) is almost as fast as applying the filtered backprojection type algorithm   (which requires $\mathcal O(N^3)$
 FLOPS) for evaluating the
adjoint or the inverse of $\Lo$.
In three spatial dimensions the complexity of the $k$-space method scales to $\mathcal O(N^4 \log N)$.
In this case one iterative step is already faster than filtered backprojection type algorithms
(which in this case requires $\mathcal O(N^5)$ FLOPS).
As we will see in the  numerical results presented below, around 10 iterations with the CG method already gives very accurate reconstruction  results. This shows that our iterative algorithms are a
good option for PAT image reconstruction even in situations,
where an explicit filtered backprojection type formula is available.
\end{remark}

\begin{figure}[tbh!]\centering
\includegraphics[width =0.3\textwidth]{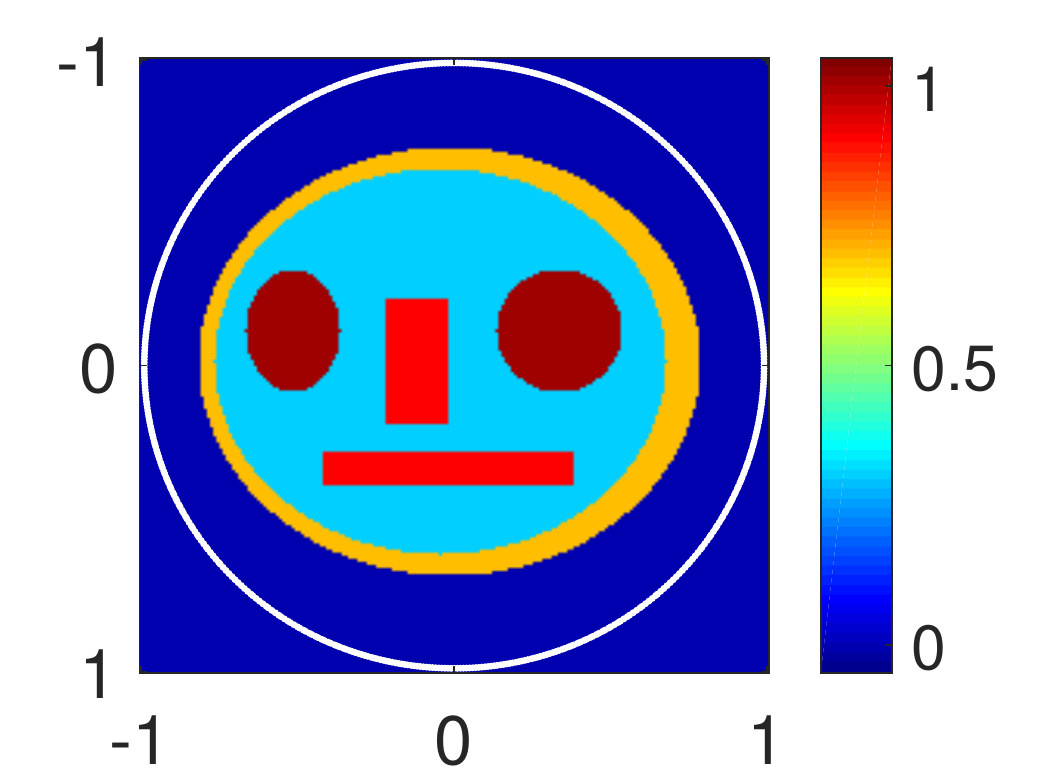}
\includegraphics[width =0.3\textwidth]{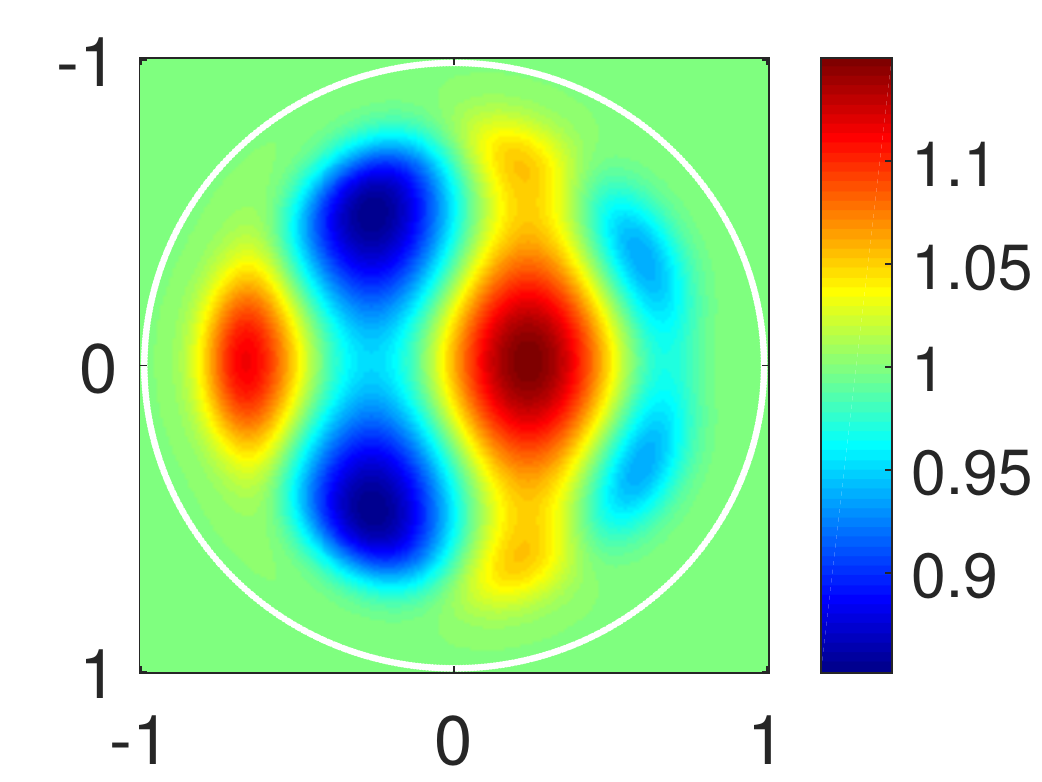}
\includegraphics[width =0.3\textwidth]{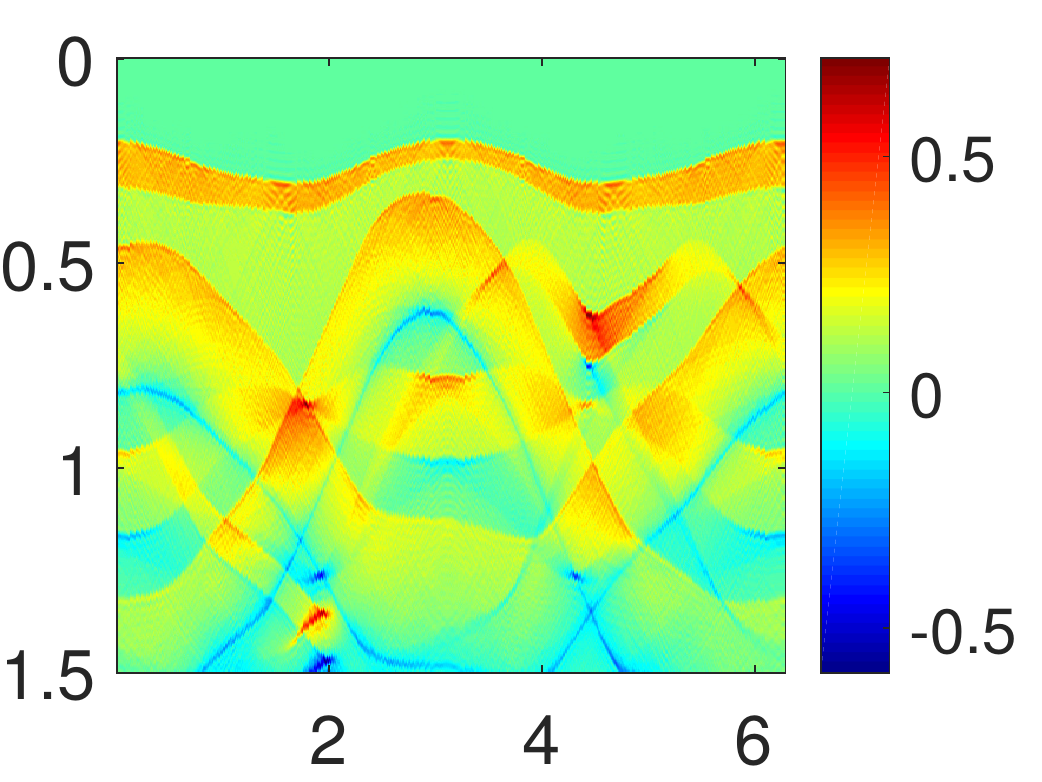}\\
\includegraphics[width =0.22\textwidth]{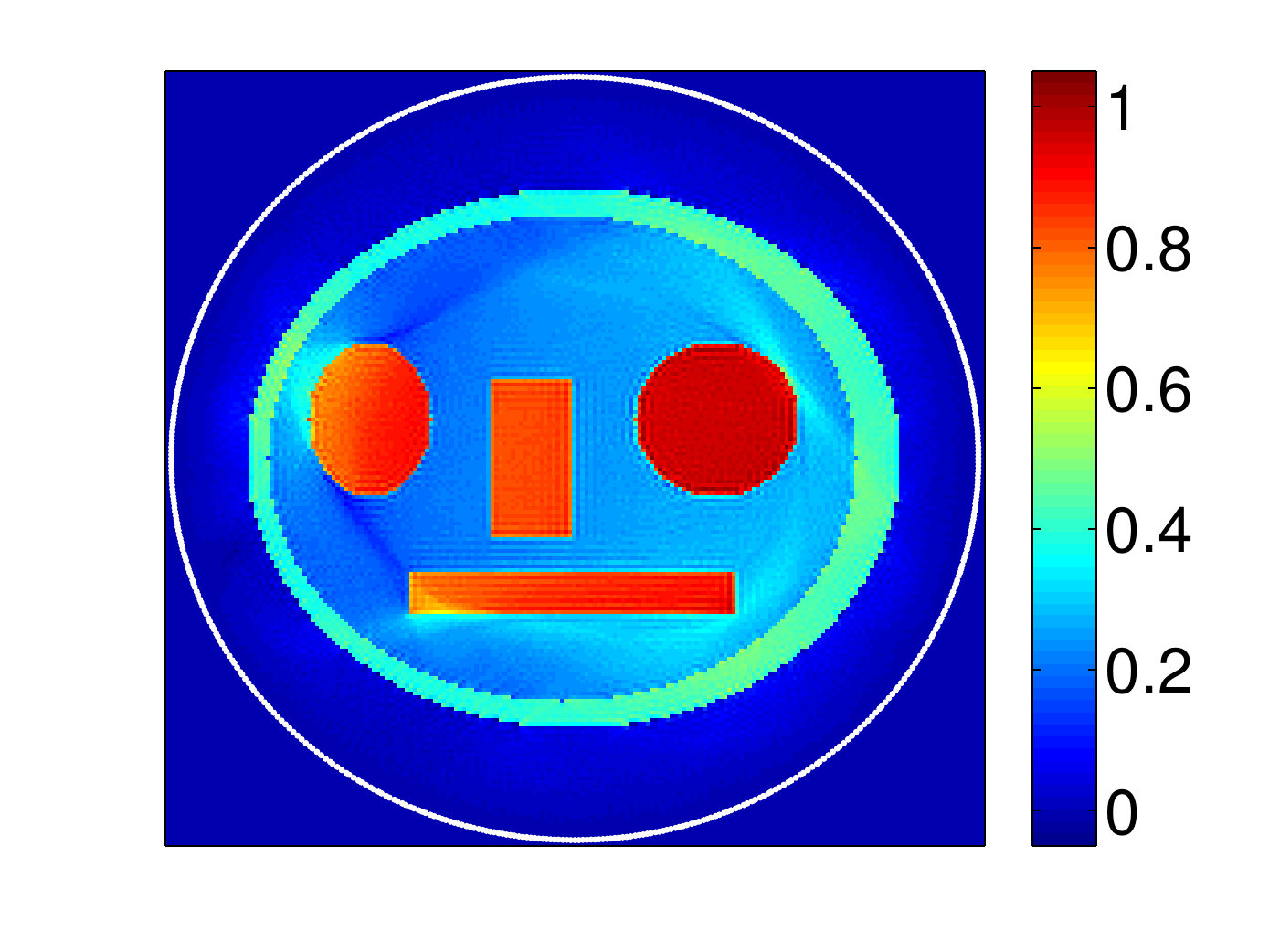}
\includegraphics[width =0.22\textwidth]{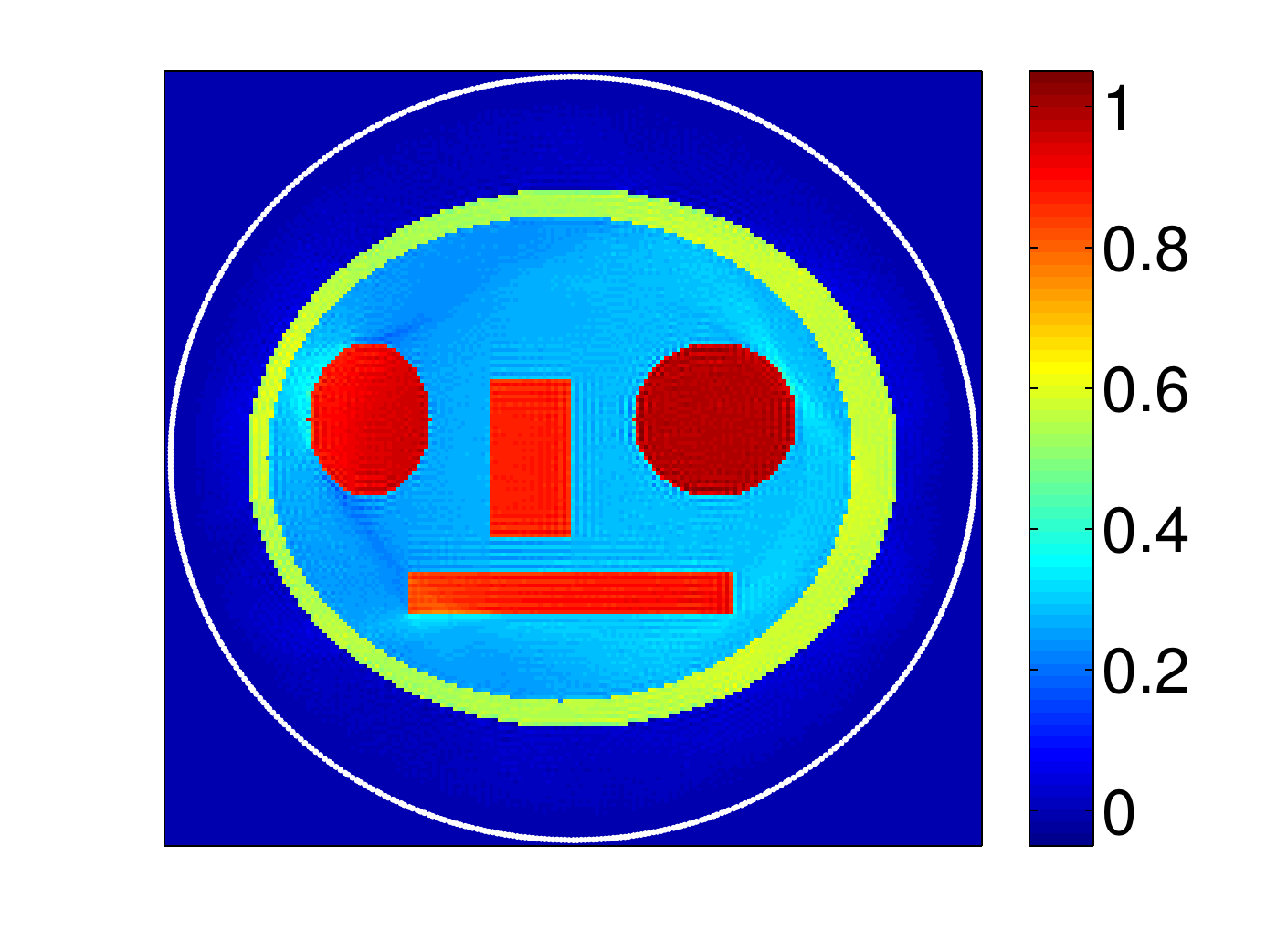}
\includegraphics[width =0.22\textwidth]{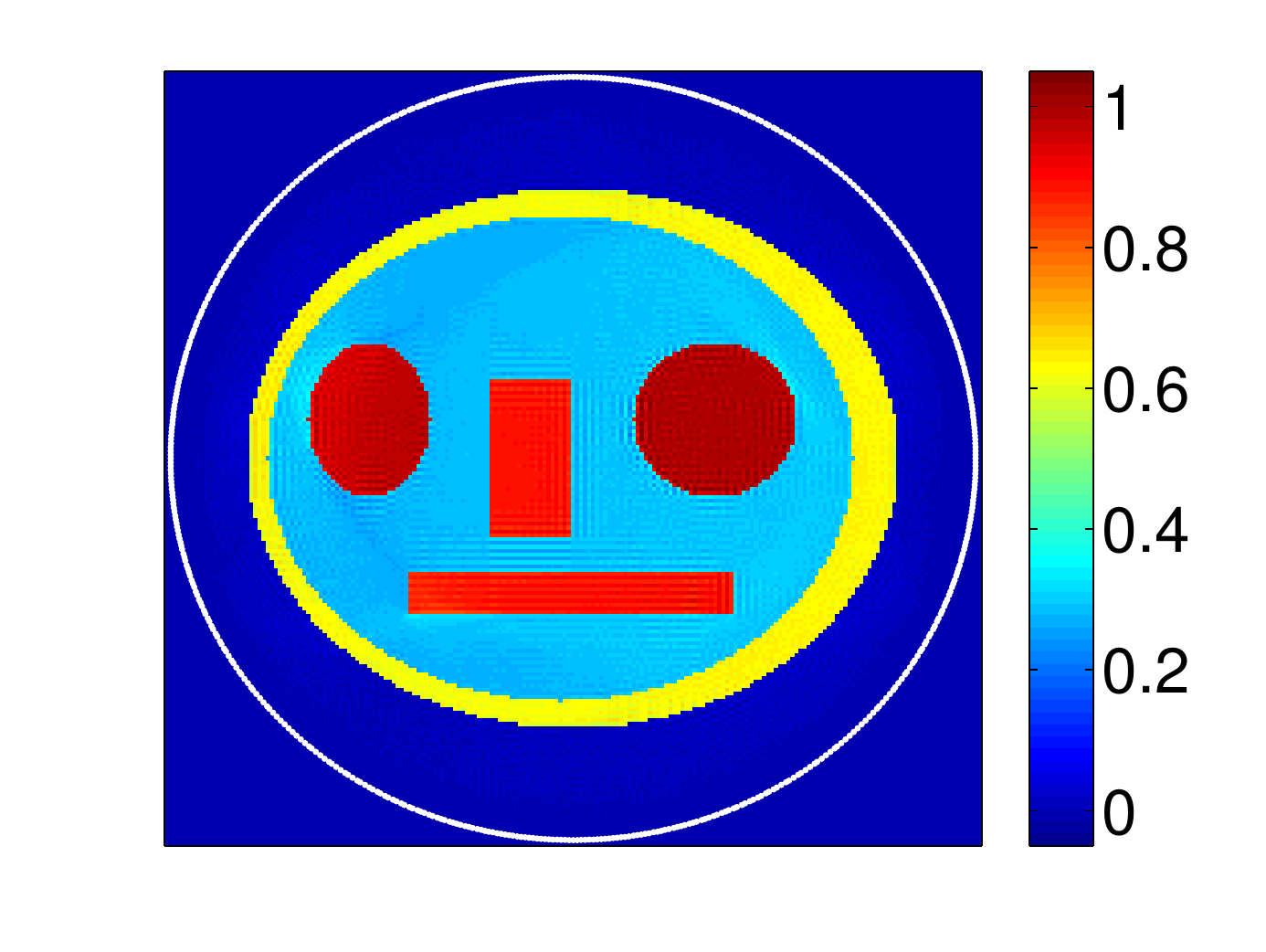} 
\includegraphics[width =0.22\textwidth]{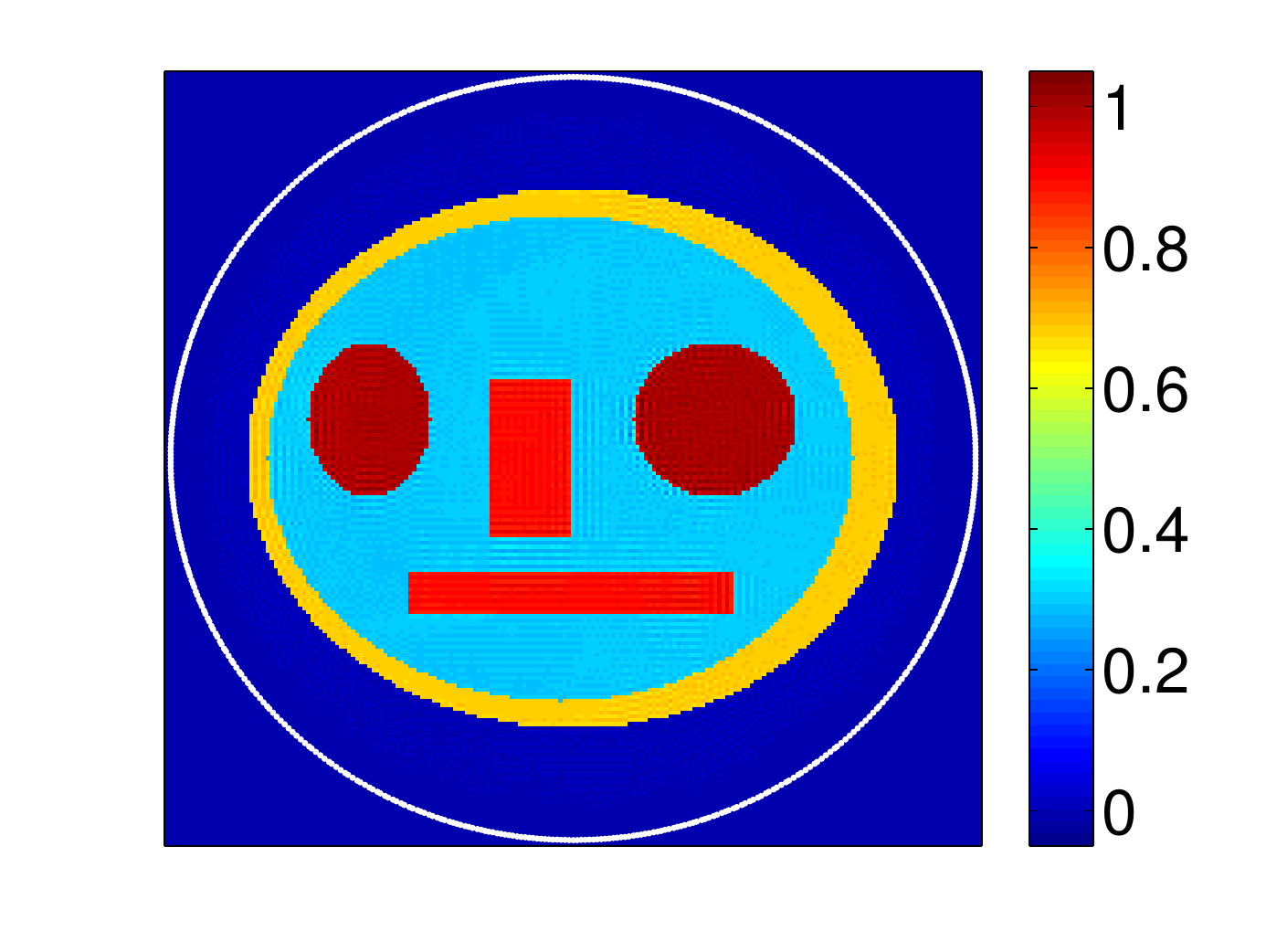} 
\\
\includegraphics[width =0.22\textwidth]{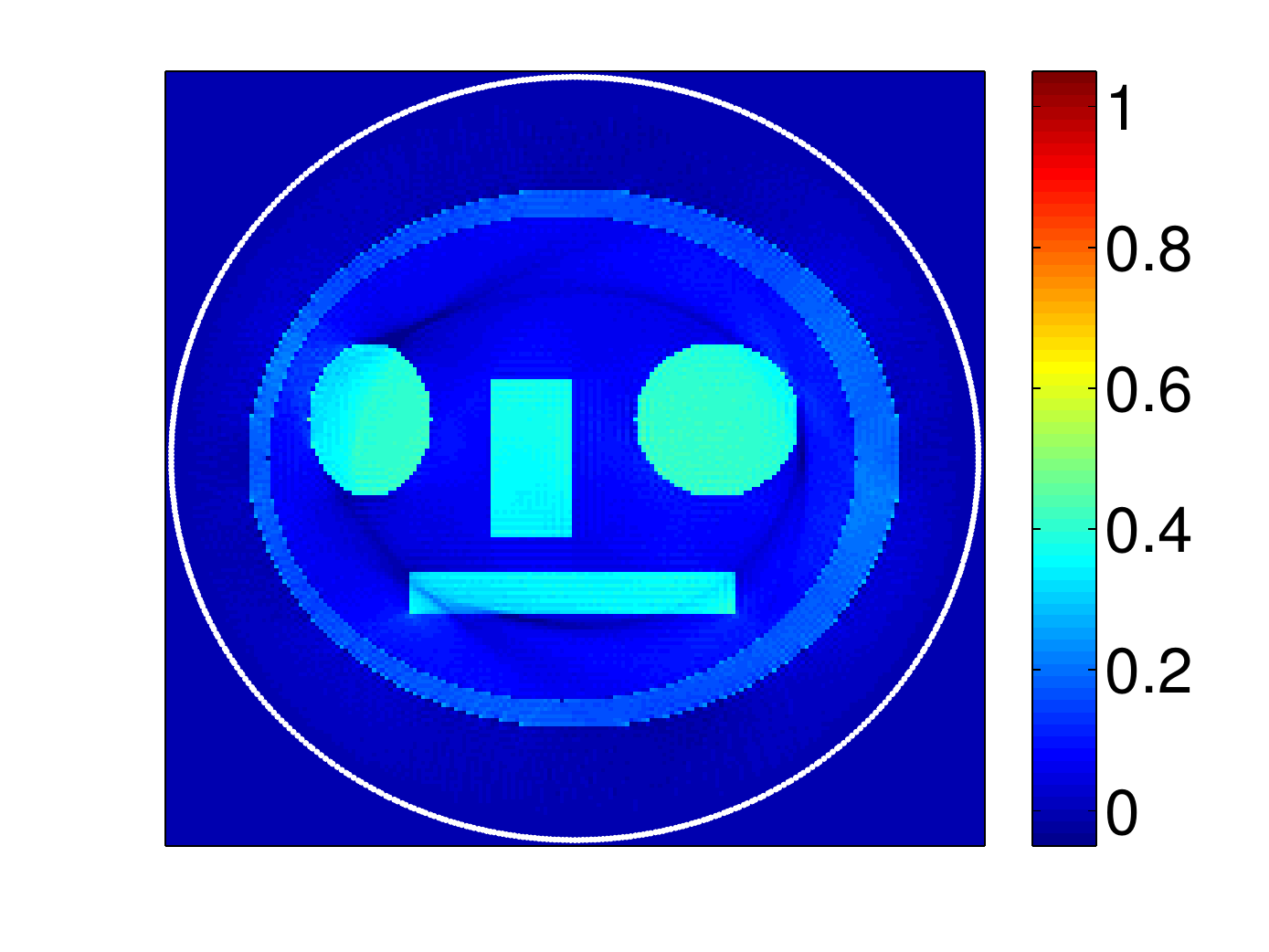}
\includegraphics[width =0.22\textwidth]{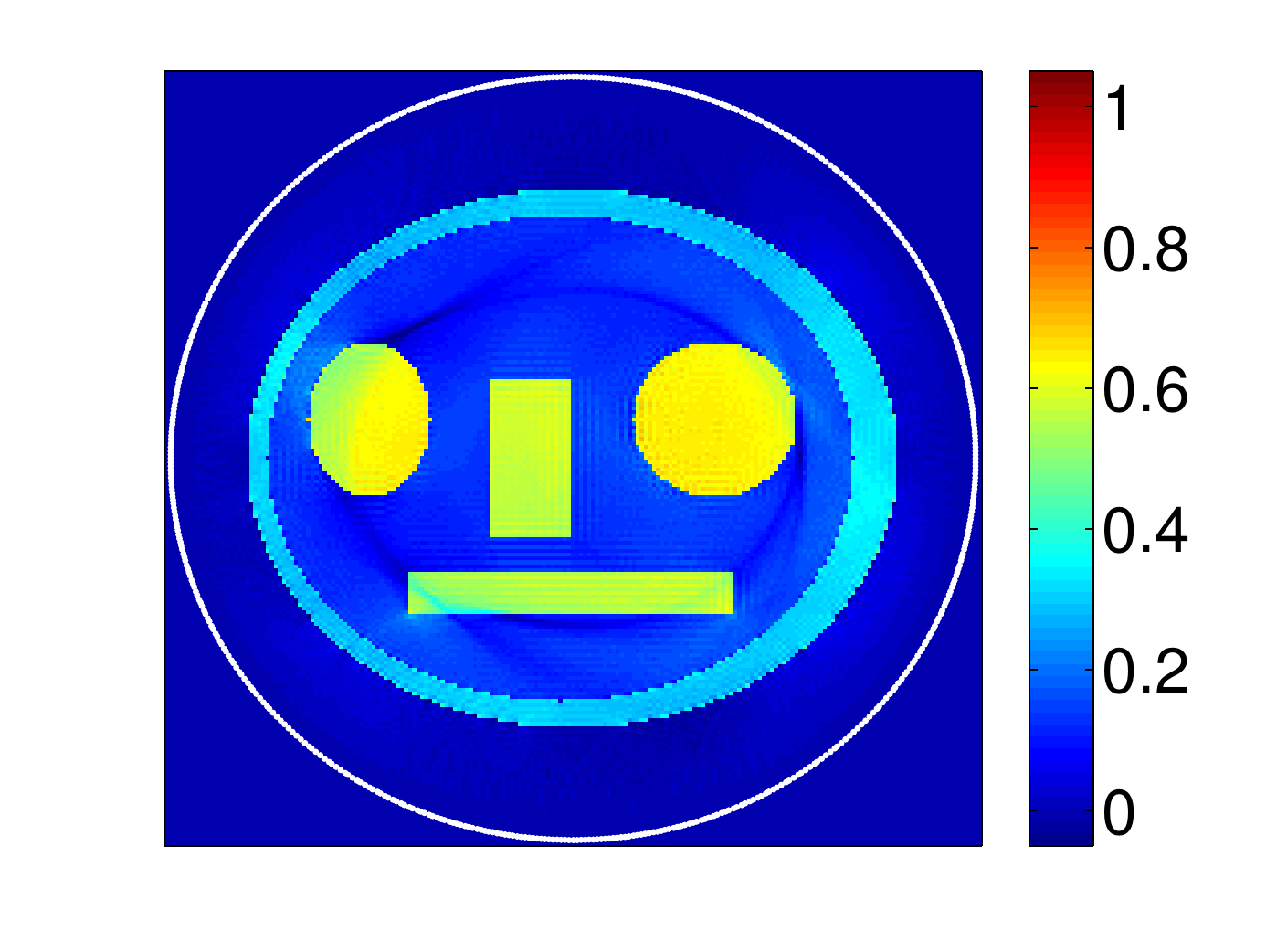}
\includegraphics[width =0.22\textwidth]{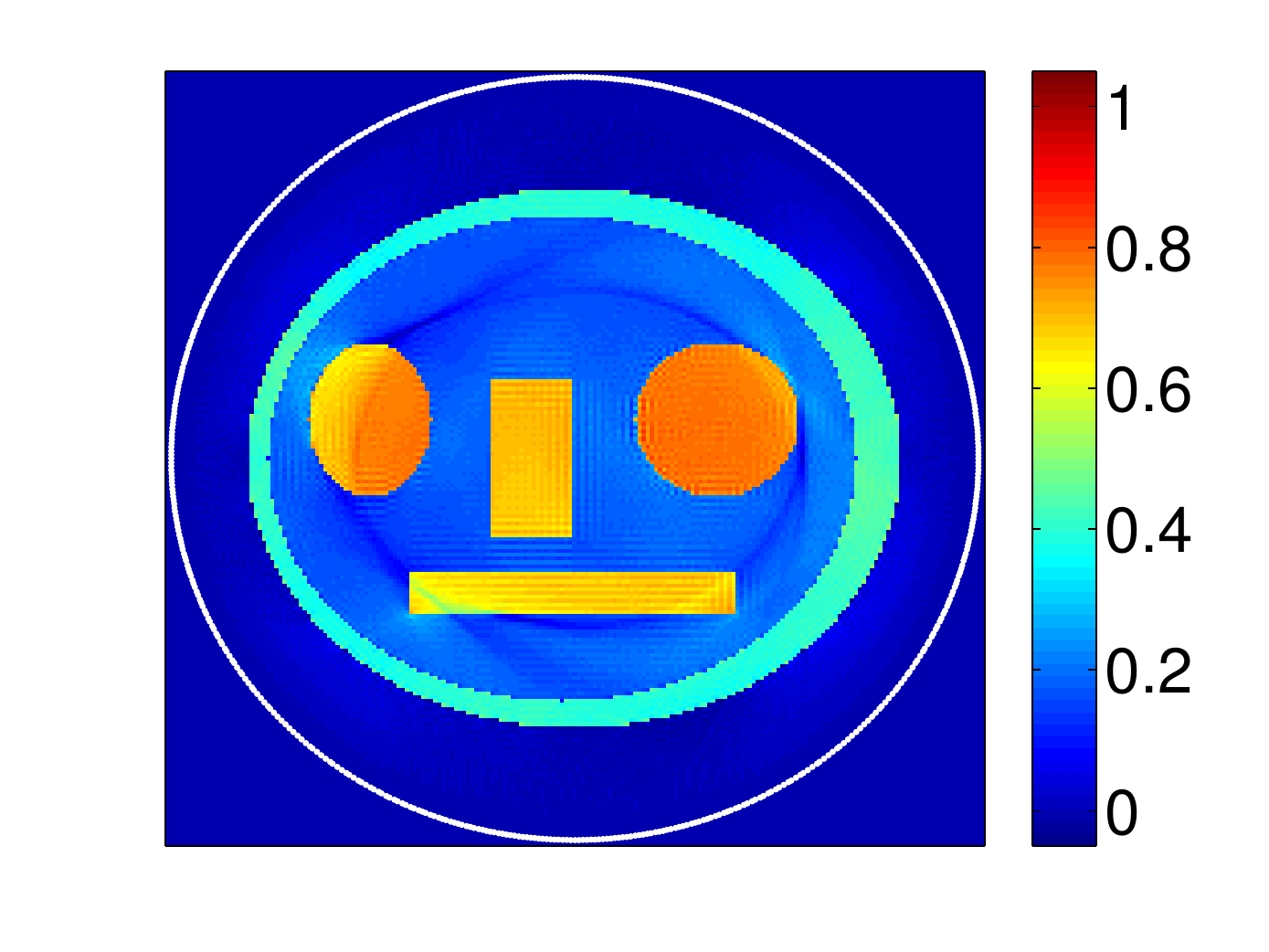} 
\includegraphics[width =0.22\textwidth]{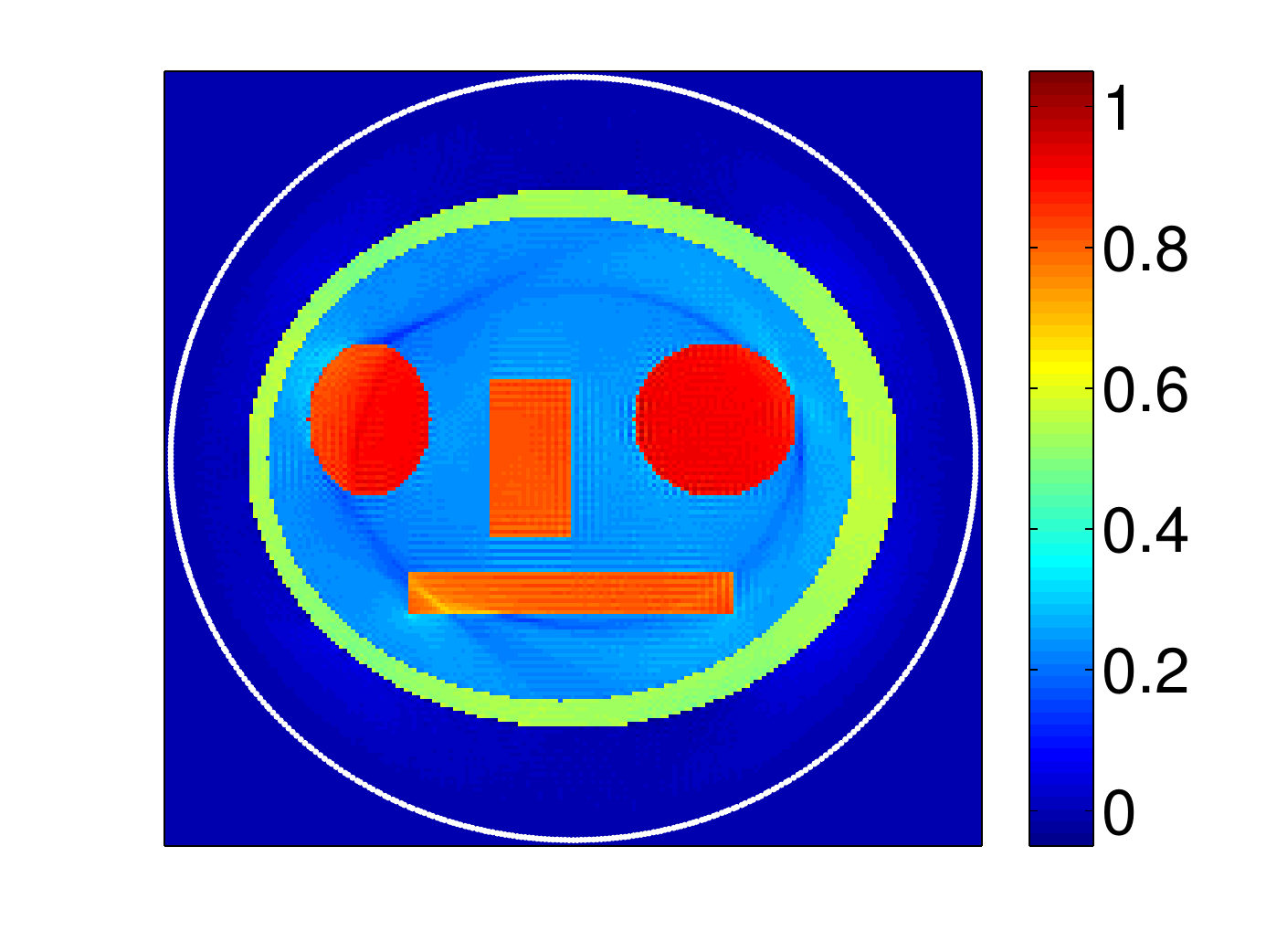} 
\\
\includegraphics[width =0.22\textwidth]{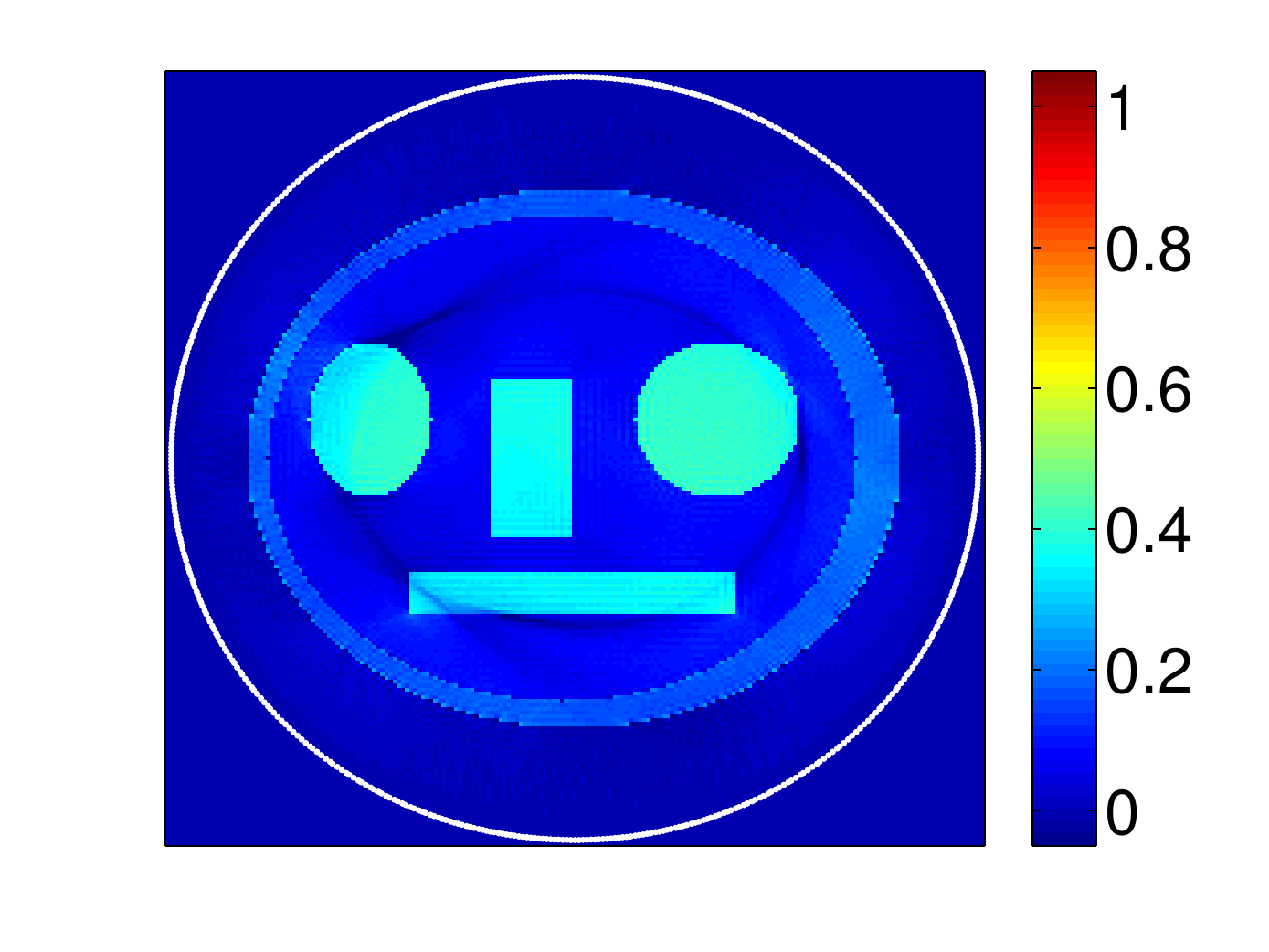}
\includegraphics[width =0.22\textwidth]{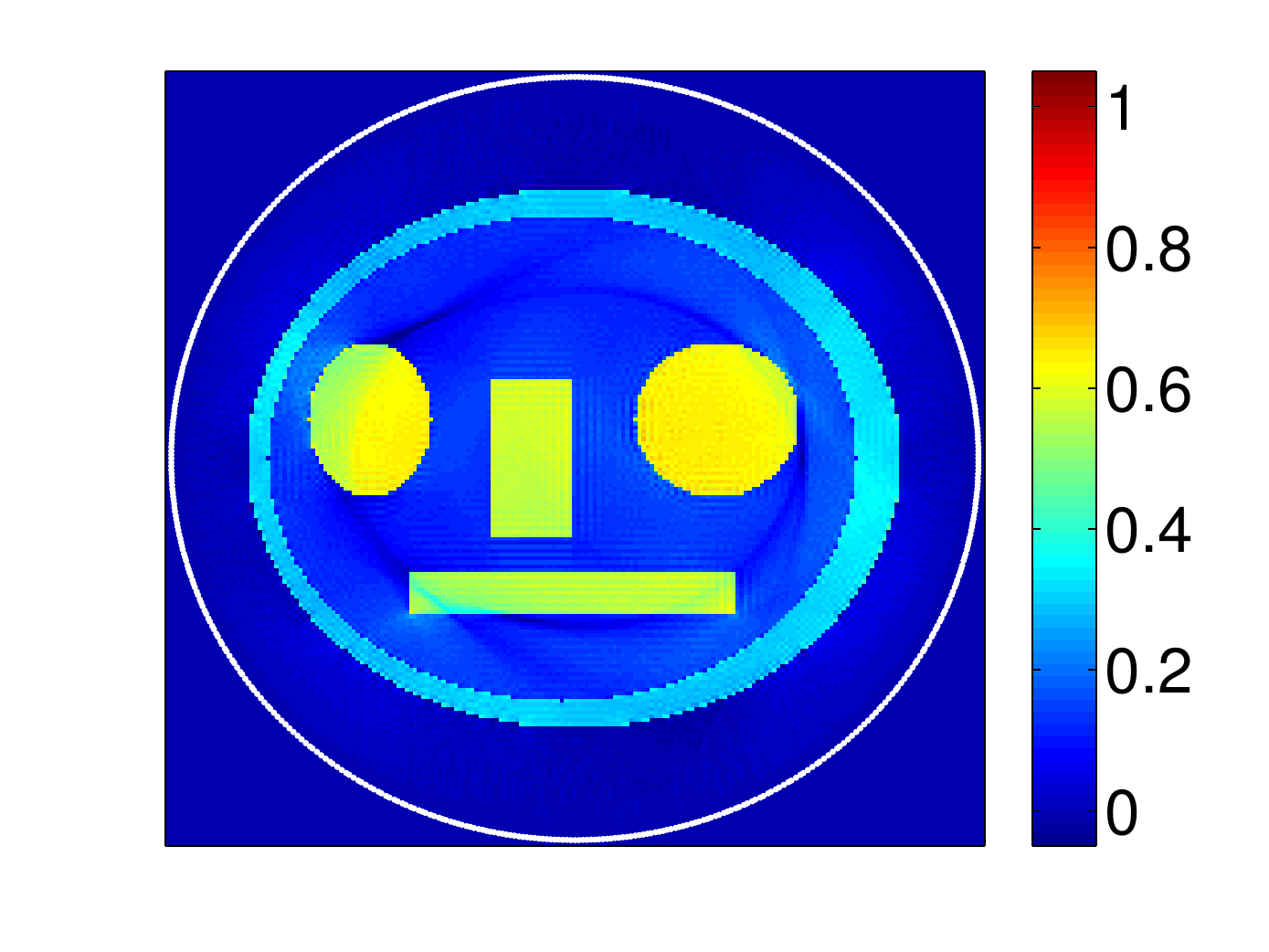}
\includegraphics[width =0.22\textwidth]{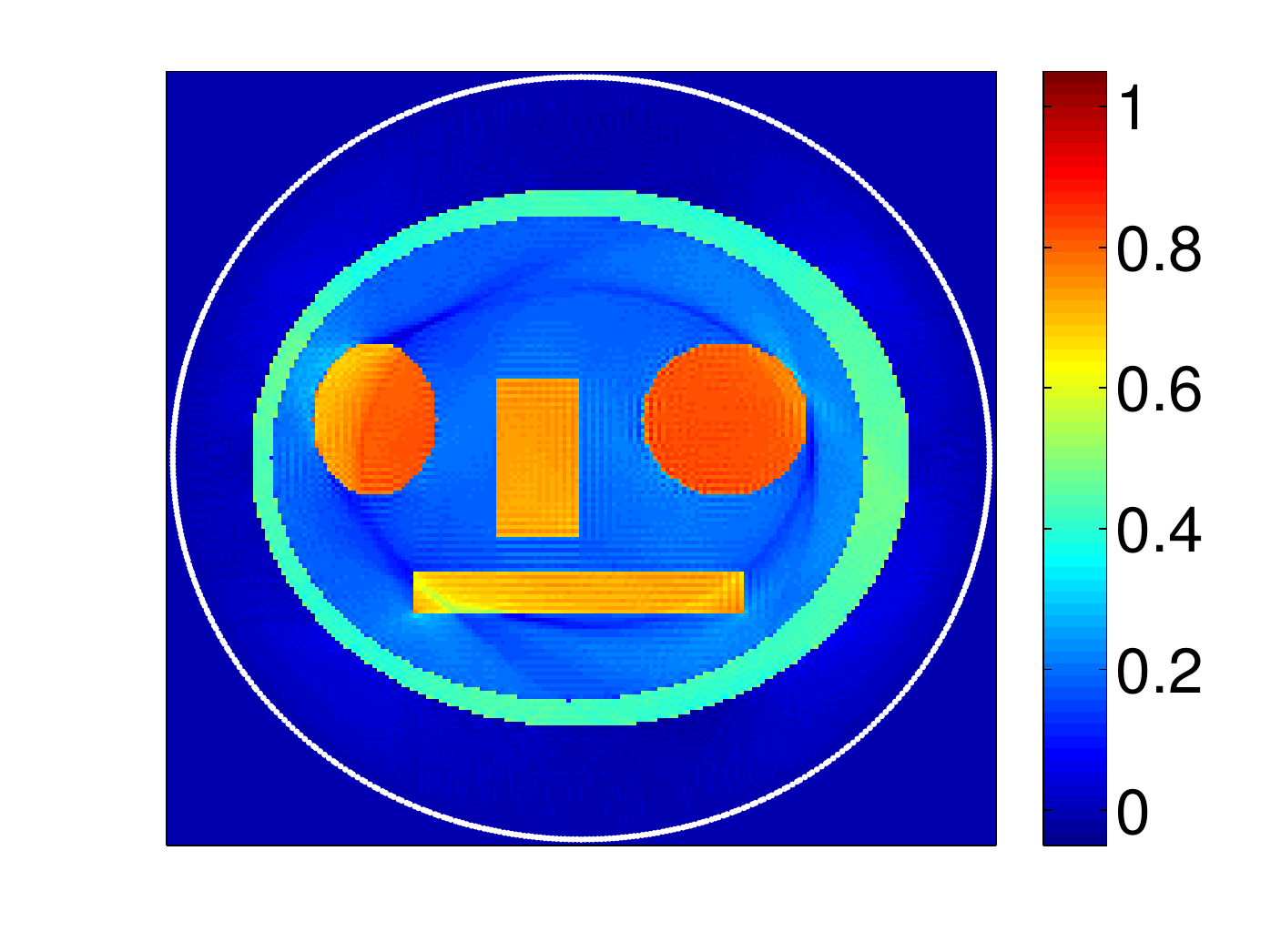} 
\includegraphics[width =0.22\textwidth]{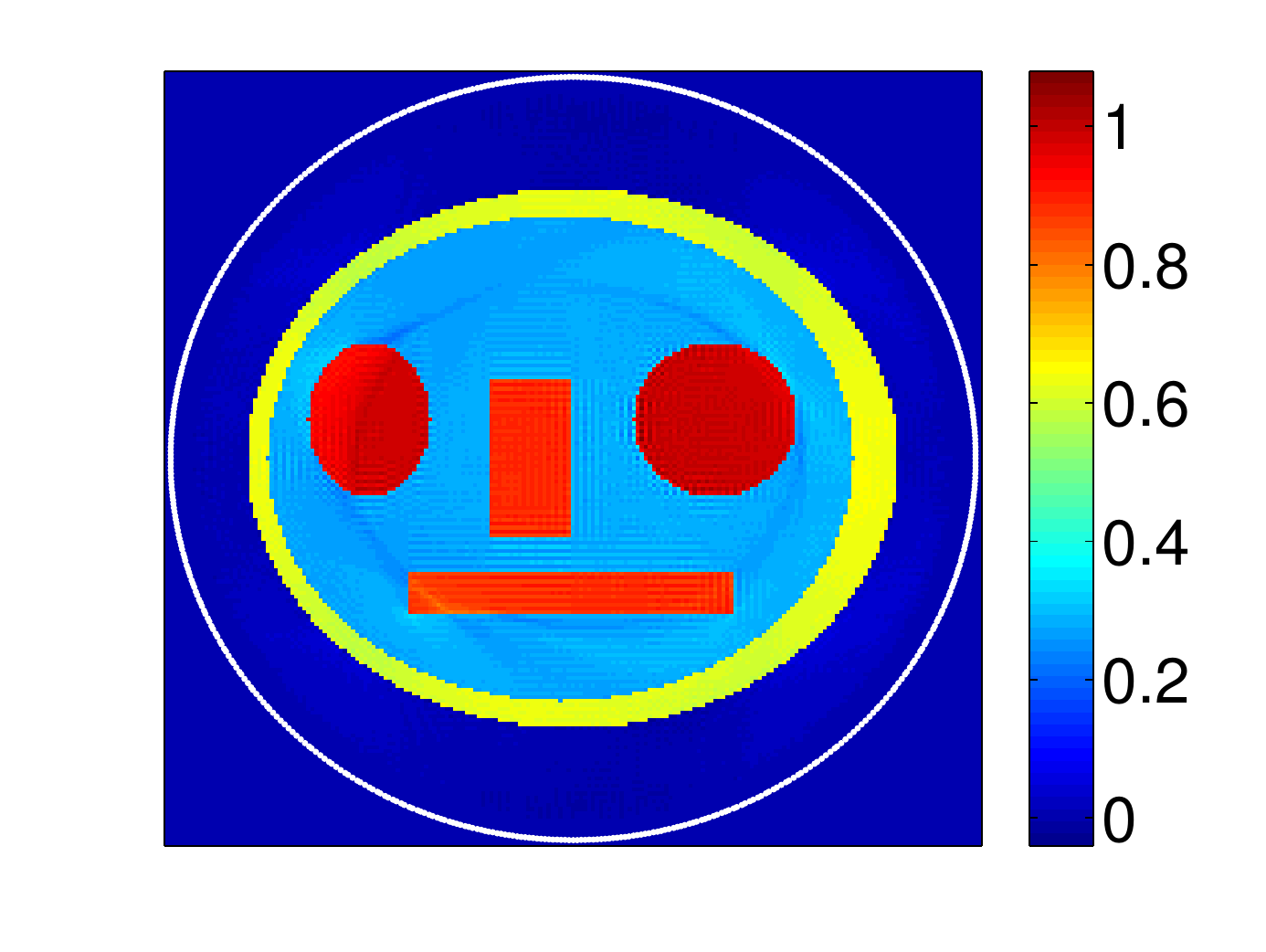} 
\\
\includegraphics[width =0.22\textwidth]{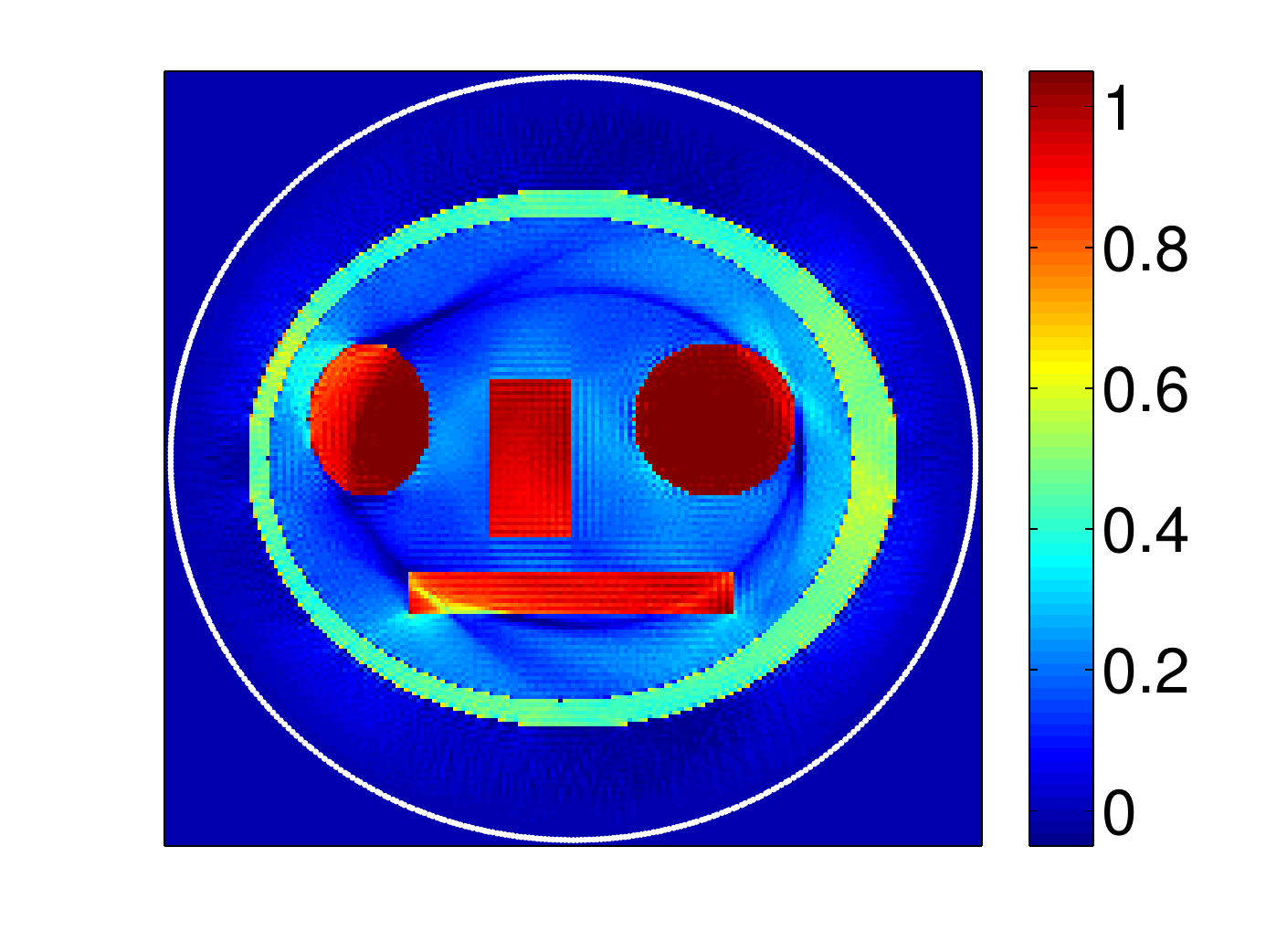}
\includegraphics[width =0.22\textwidth]{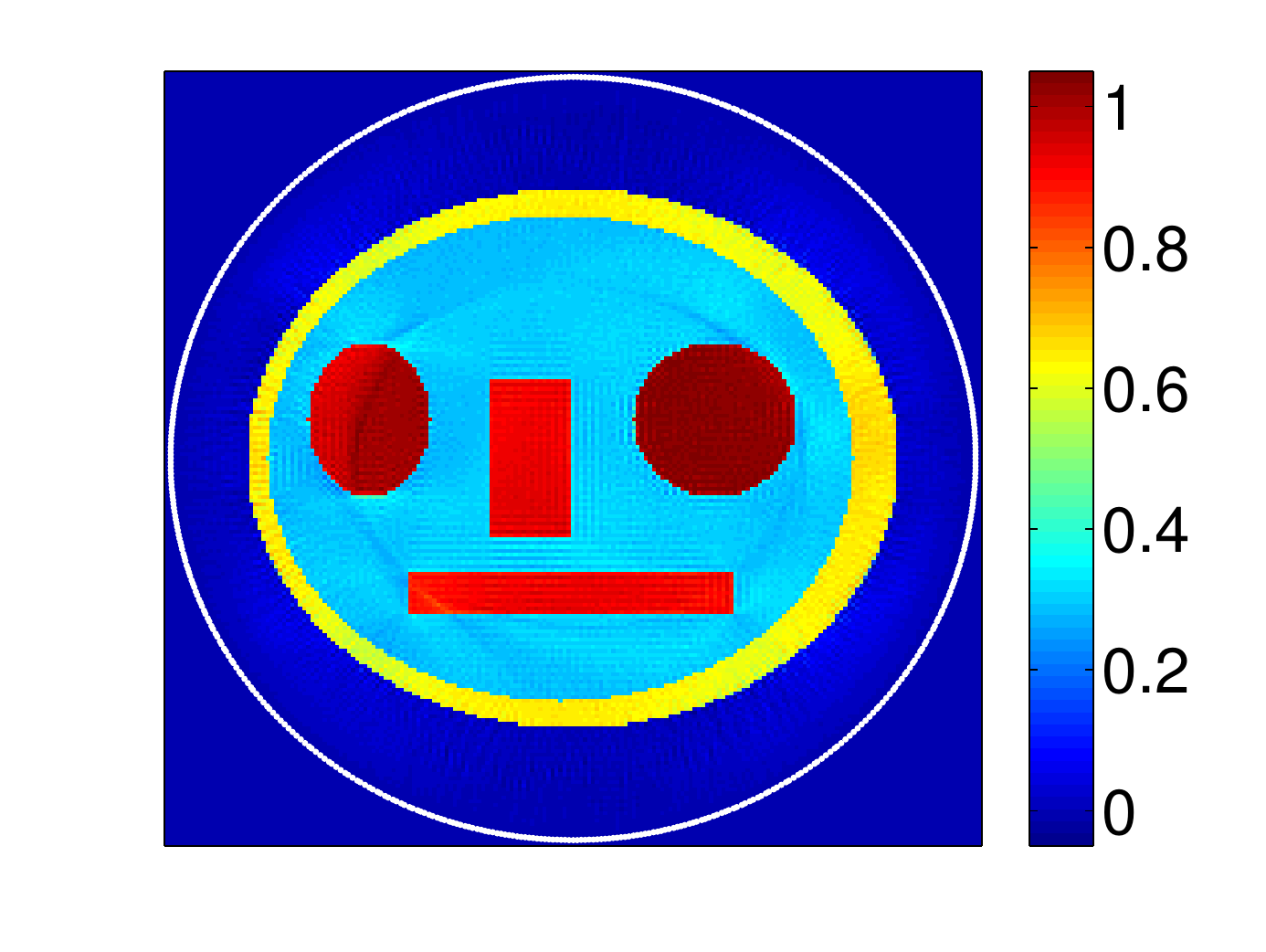}
\includegraphics[width =0.22\textwidth]{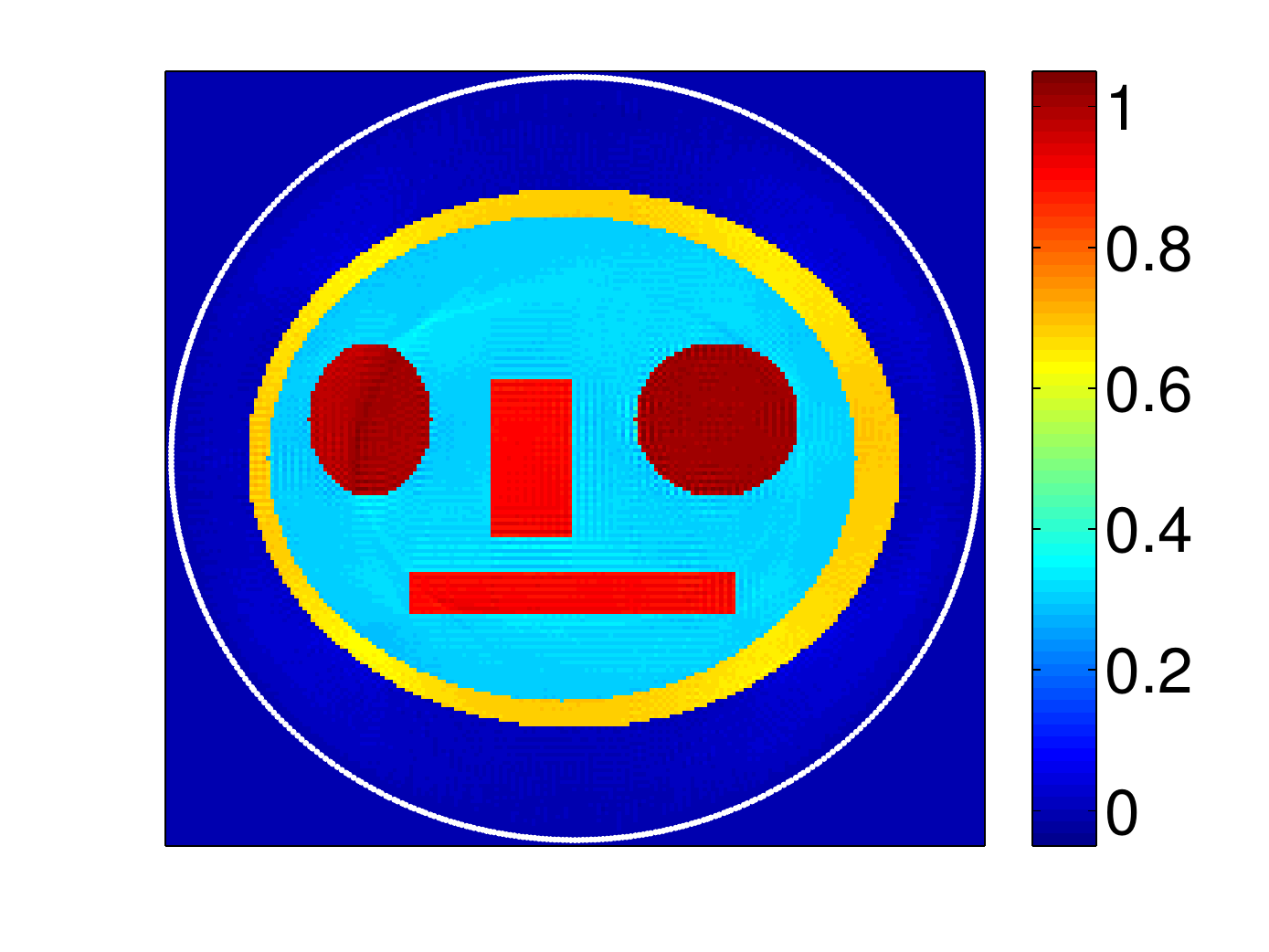} 
\includegraphics[width =0.22\textwidth]{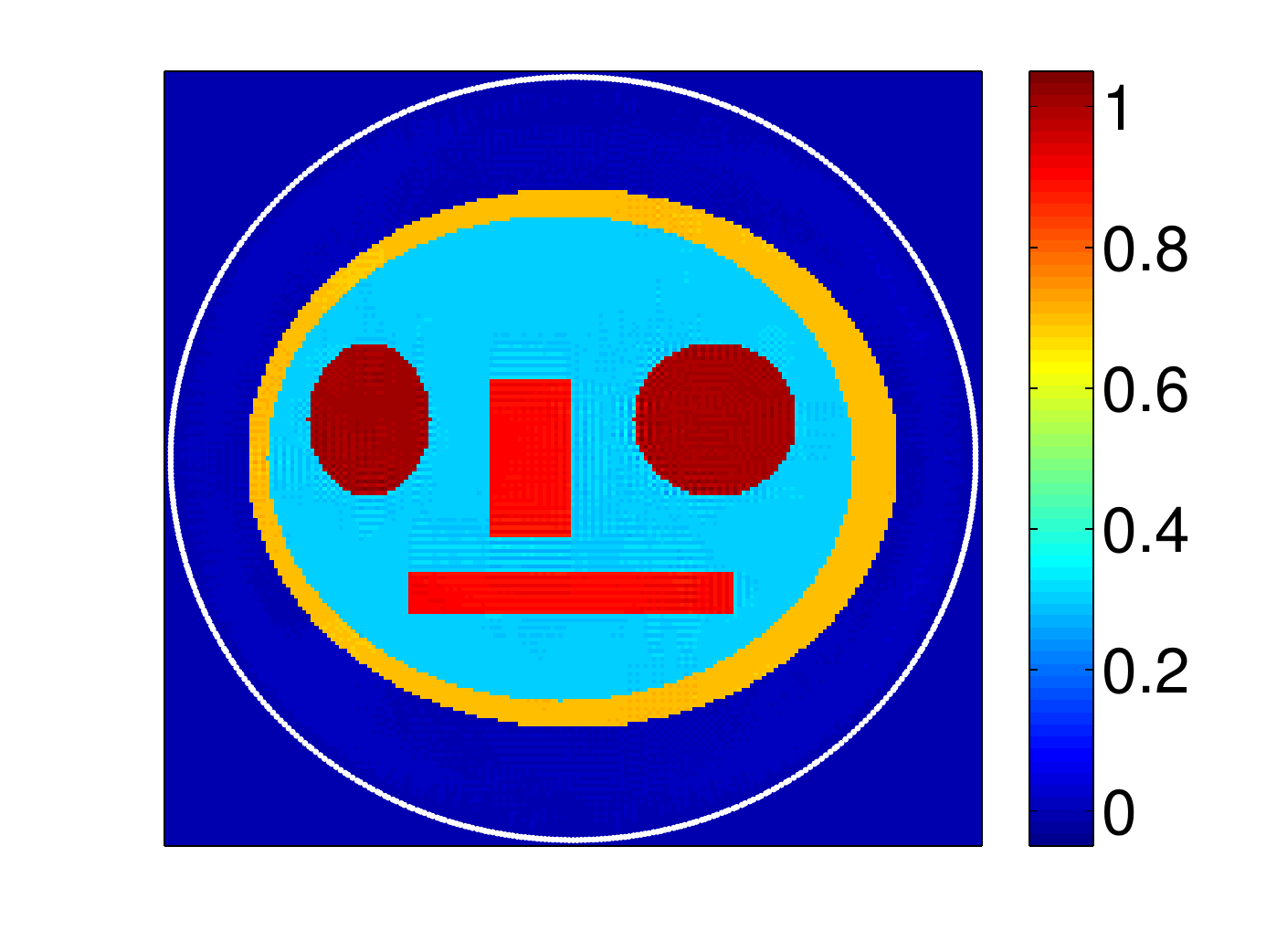} 
\\
\includegraphics[width =0.22\textwidth]{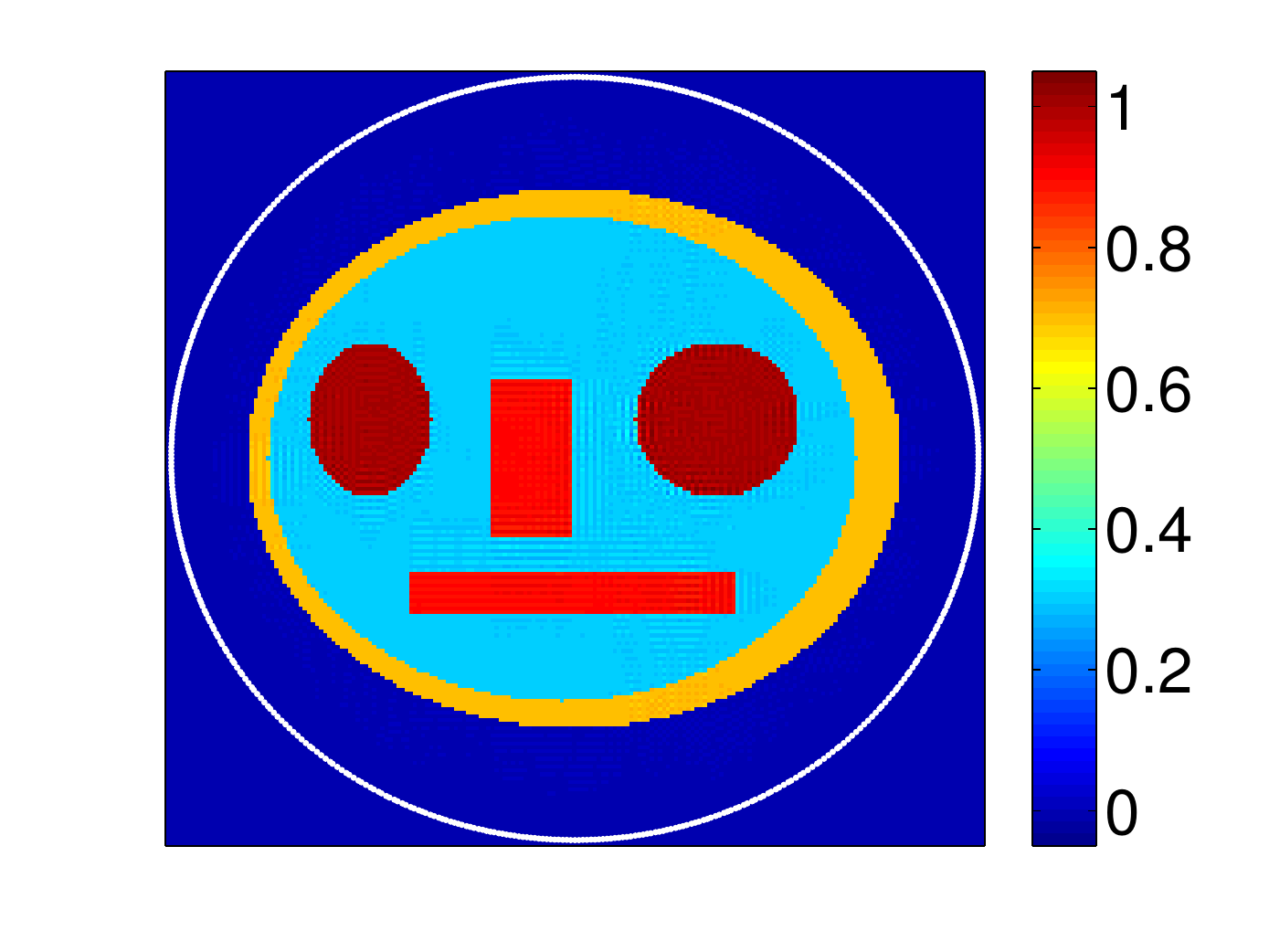}
\includegraphics[width =0.22\textwidth]{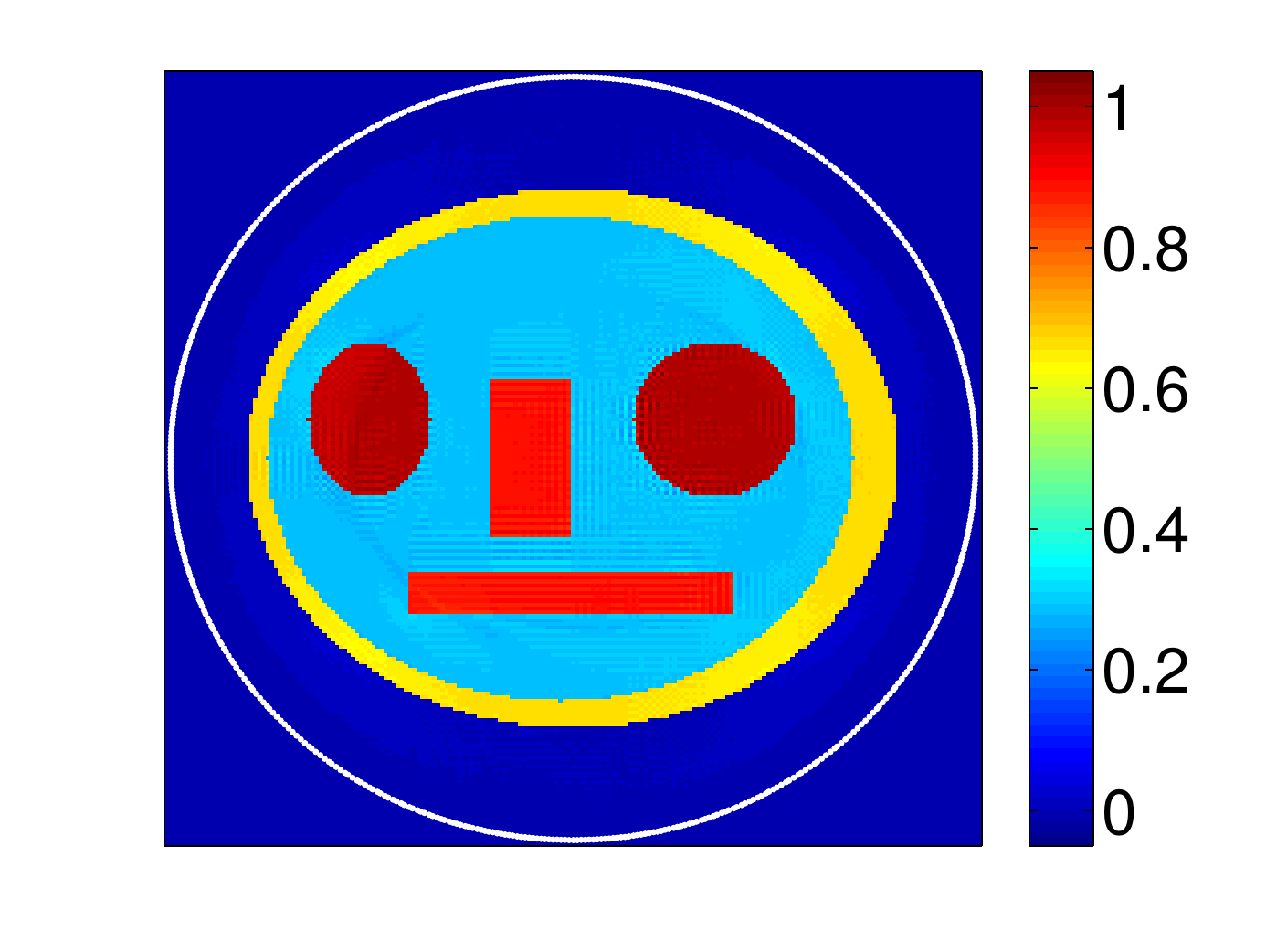}
\includegraphics[width =0.22\textwidth]{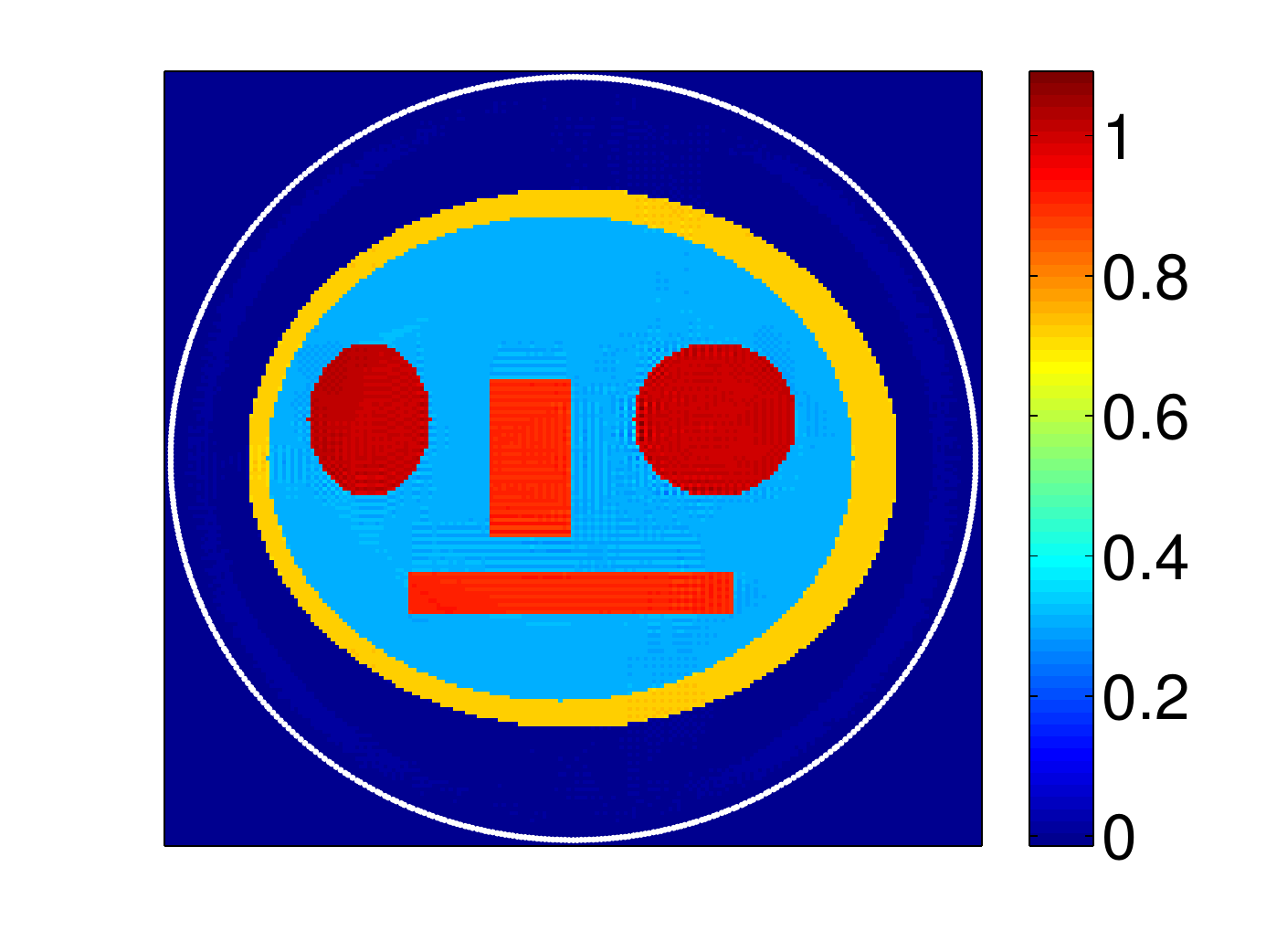} 
\includegraphics[width =0.22\textwidth]{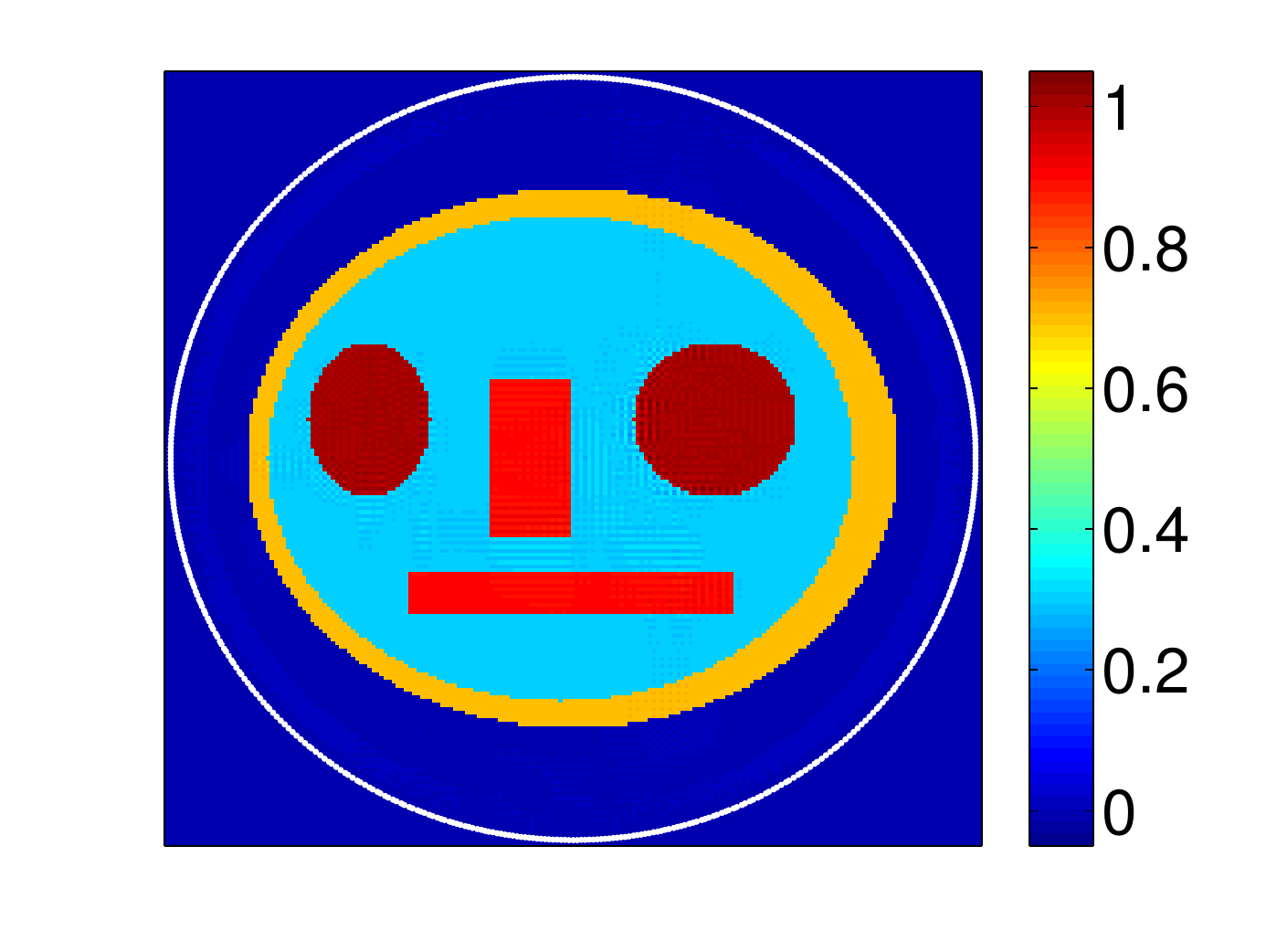} 
\caption{{\scshape Test case~\ref{t1}, Reconstructions from exact data.}
Row 1: Initial pressure data $f$ (left), non-trapping sound speed $c$ (middle), and  pressure data $\Lo f$ (right). The white dots indicate 
the measurement curve.
Row 2: Iterative time reversal (after 1, 2, 3 and 5 iterations).
Row 3: Landweber's  method (after 1, 2, 3 and 5 iterations).
Row 4: Nesterov's method (after 1, 2, 3 and 5 iterations).
Row 5: CG method  (after 1, 2, 3 and 5 iterations).
Row 6: Results after 10 iterations using iterative time reversal,  
Landweber's method, Nesterov's method, CG method (from left to right).
\label{fig:test1}}
\end{figure}

\subsection{Test case \ref{t1}:  Complete data}
\label{sec:test1}

We  consider the  nontrapping sound speed (taken from
 \cite{QiaSteUhlZha11})
\begin{equation}\label{eq:ss-nontrapping}
 	c(x) = 1 +  w(x) \kl{ 0.1 \cos (2\pi x_1) +  0.05 \sin (2\pi x_2)} \,,
\end{equation}
where $w \colon \R^2 \to [0,1] $ is a smooth function
that vanishes outside  $B_1(0)$ and  is equal to one on $B_{1/2}(0)$.
The  sound speed $c$, the phantom $f$ and the corresponding
full data $\Lo f$ are illustrated in  the top row in Figure~\ref{fig:test1}.
For the results presented in this section we use $R=1$ and $N=200$,  which yields a
spatial step size of $h_x = 2R/N = 1/100$.  We further use a final time $T = 1.5$ and take
$M=800$  for the temporal discretization.

\begin{figure}[tbh!]\centering
\includegraphics[width =0.48\textwidth]{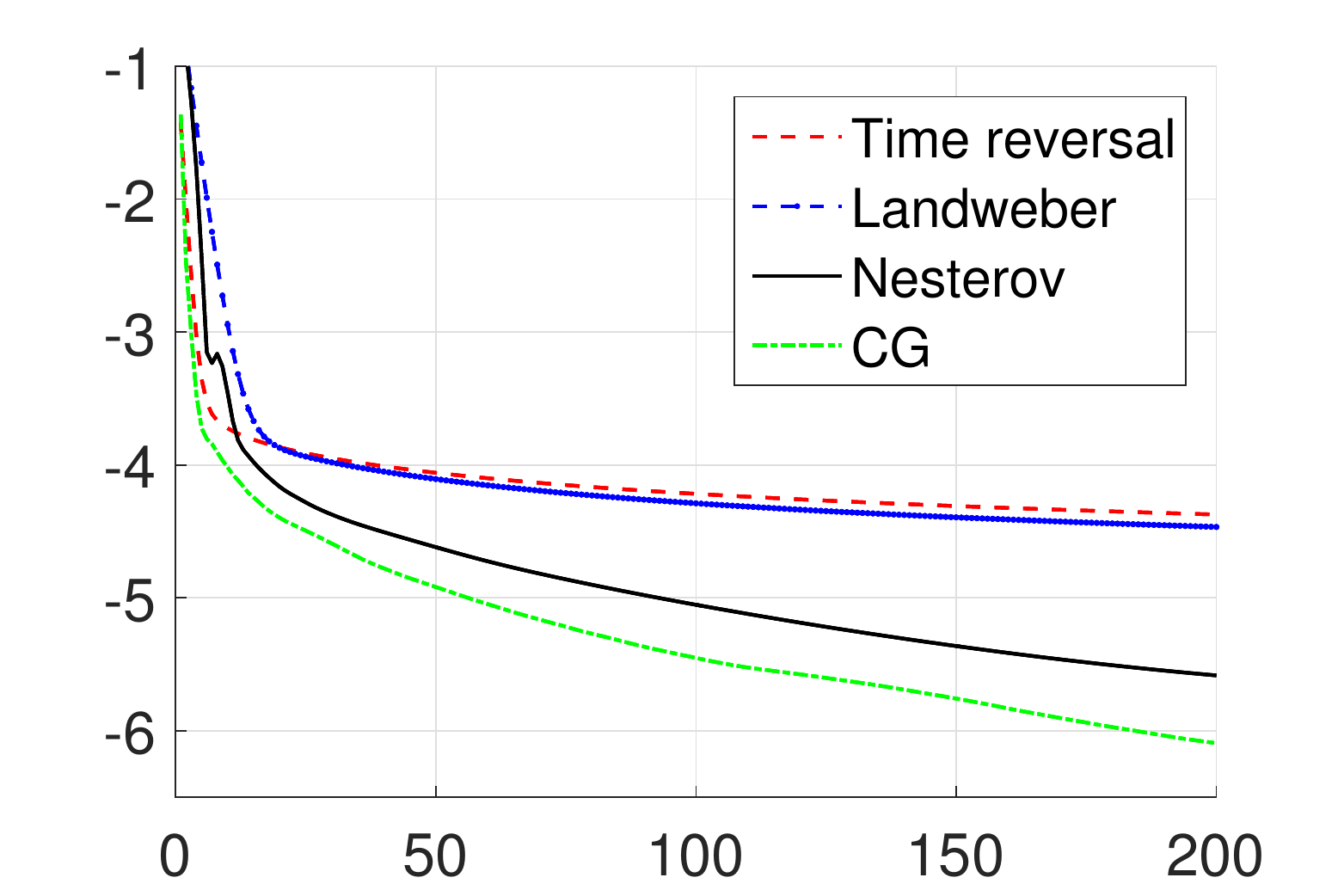}
\includegraphics[width =0.48\textwidth]{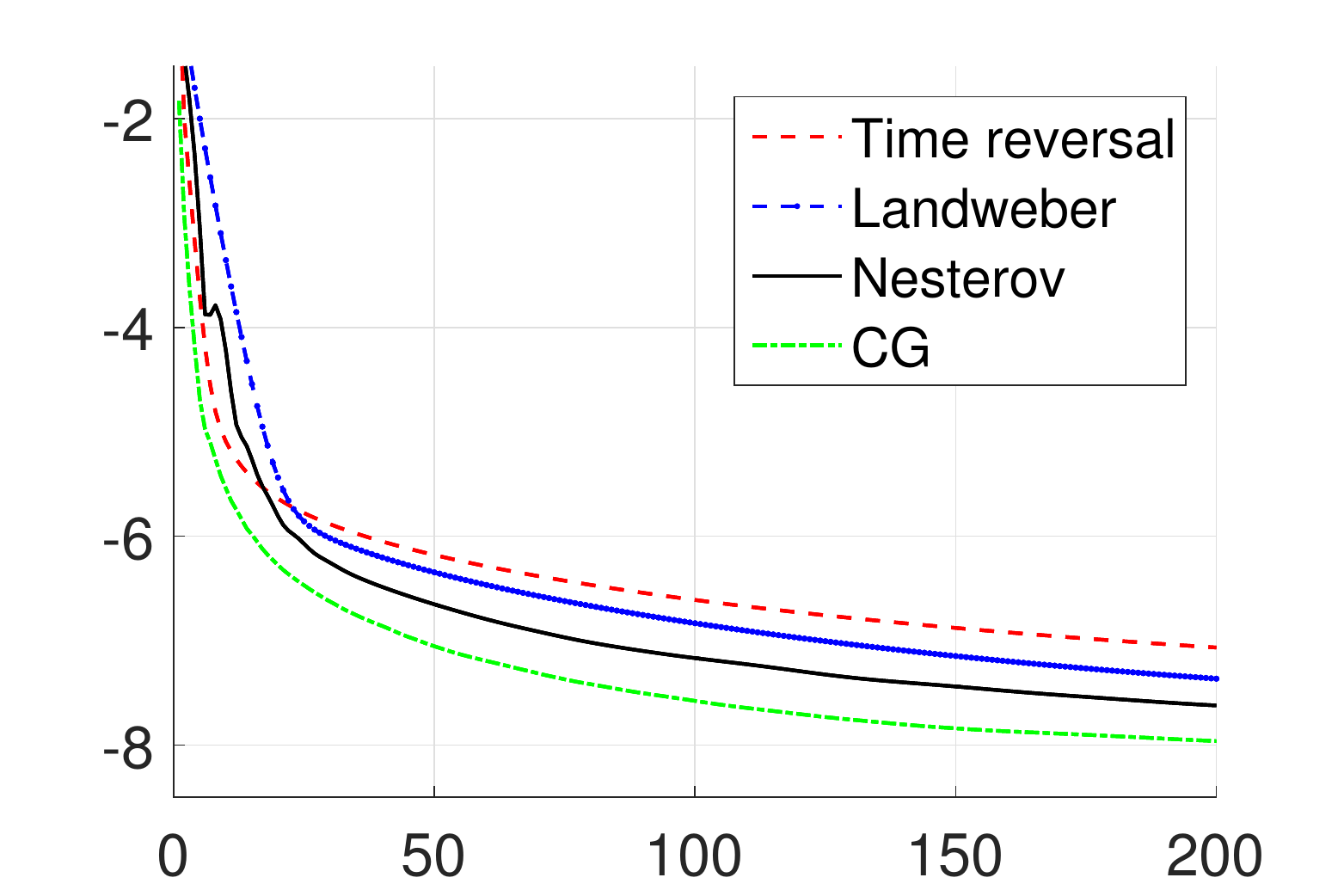}
\caption{{\scshape Test case \ref{t1}, convergence behavior for exact data}
Left:  Logarithm of squared reconstruction error $\norm{\fnum_n - \fnum}^2_2$
in dependence of the iteration  number.  Right:  Logarithm of residual $\norm{\Lnum_{N,M}\fnum_n - \gnum}^2_2$
in dependence of the iteration  number.\label{fig:test1error}}
\end{figure}

We performed iterative reconstructions  using  the following methods:
\begin{enumerate}[label=(\alph*)]
\item Iterative time reversal method
\item Landweber's method
\item Nesterov's method
\item CG method.
\end{enumerate}
Figure~\ref{fig:test1}  shows reconstruction  results using these methods after 1, 2, 3, 5, 
and 10 iterations. One notices that all iterations converge quite rapidly  to the phantom to be recovered. In the initial iterations the time reversal and the CG method produce the best results.
After 10 iterations all reconstructions look very similar to the initial phantom.      
To investigate the convergence behavior further, in  Figure \ref{fig:test1error}
we plot the logarithm of the squared discrete $L^2$-reconstruction error and squared residual
\begin{align*}
     \norm{\fnum_n -\fnum}_2^2
     &\coloneqq
     \sum_{i} \abs{ \fnum_n[i] - \fnum[i] }^2  h_x^2
     \simeq \norm{f-f_n}_{L^2}^2 \,,
     \\
     \norm{\Lnum_{N,M} \fnum_n -\gnum}^2
     &\coloneqq \sum_{b,j} \abs{ \Lnum_{N,M} \fnum_n[b,j]- \gnum[b,j] }^2  h_x h_t  \simeq  \norm{\Lo f_n -g}_{L^2}^2 \,,
\end{align*}
respectively.
One concludes from Figure~\ref{fig:test1error}, that all iterative schemes converge 
quite rapidly. In particular the CG method is the fastest.

In order to further investigate the behavior of the algorithms we repeated the
computations  with inexact data.  To that end, we generated the data on a different grid, where we use $N=350$ and $M=1300$ (recall that the iterative algorithm uses $N=200$ and $M=800$). Further, we added Gaussian white noise to the data with a standard deviation equal to $5\%$  of the $L^2$-norm of  $\Lnum_{N,M} \fnum$. The total $L^2$-error in the data is $0.049$  and the $L^2$-norm of the exact data is $\norm{\gnum}_2 = 0.44$.

Figure~\ref{fig:test1noisy} shows the reconstruction results from inexact data
using iterative time reversal, Landweber's, Nesterov's, and the CG methods.
 The errors and the residuals again decrease quite rapidly in the first iterative  steps.
However after about $10$  iterations the error as well as the residuals do not further decrease.   Consequently, the iterations can be stopped at a certain iteration index $n_\star$.
This is due to the noise in the data which causes the data to be outside the range of
$\Lnum_{N,M}$.  However, these results also reveal that we are in a stable situation,
because the error does not significantly increase after reaching the stopping  index $n_\star$.

\begin{figure}[tbh!]\centering
\includegraphics[width =0.3\textwidth]{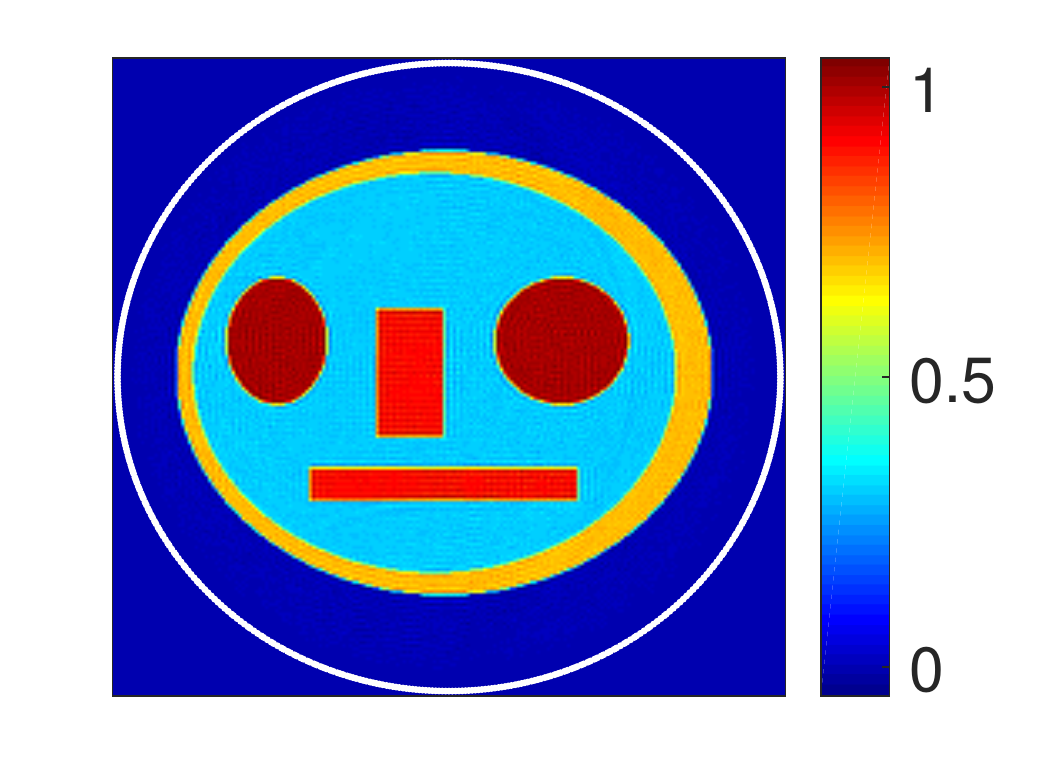}
\includegraphics[width =0.3\textwidth]{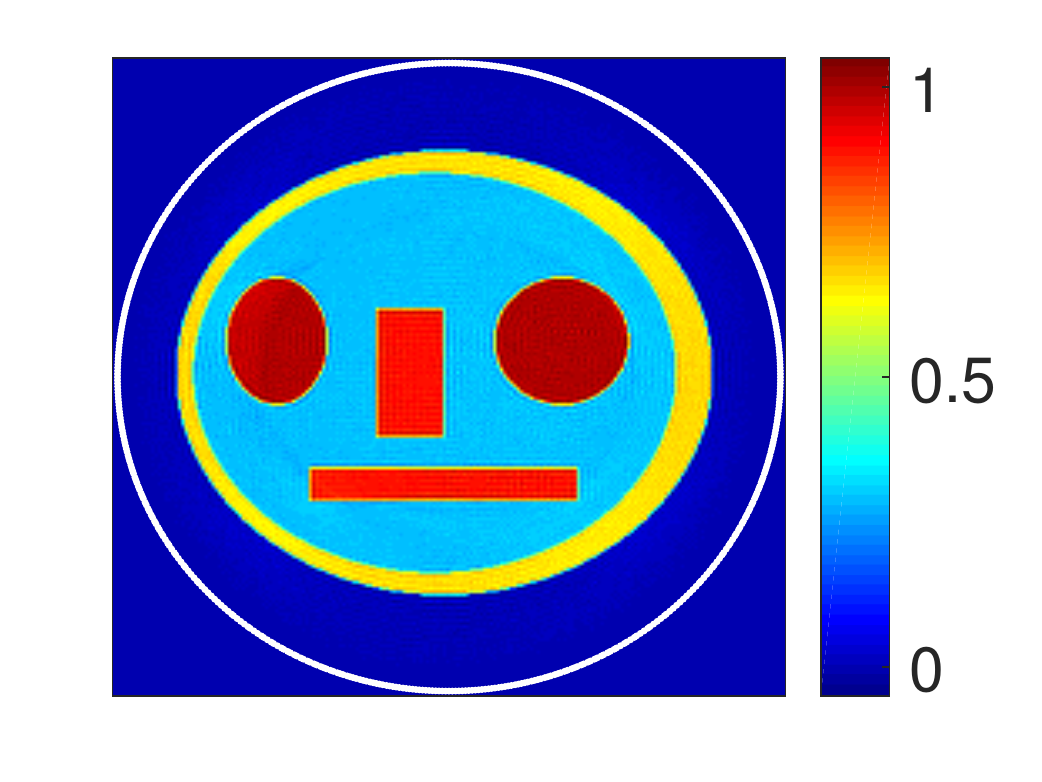}
\includegraphics[width =0.3\textwidth]{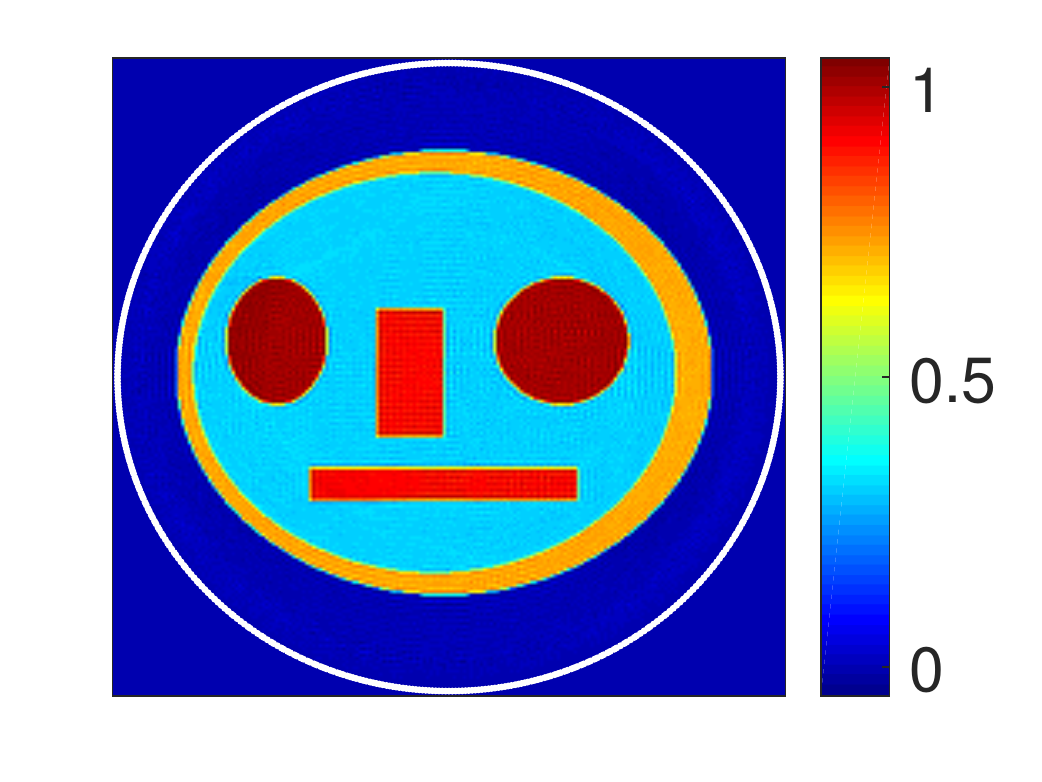}\\
\includegraphics[width =0.3\textwidth]{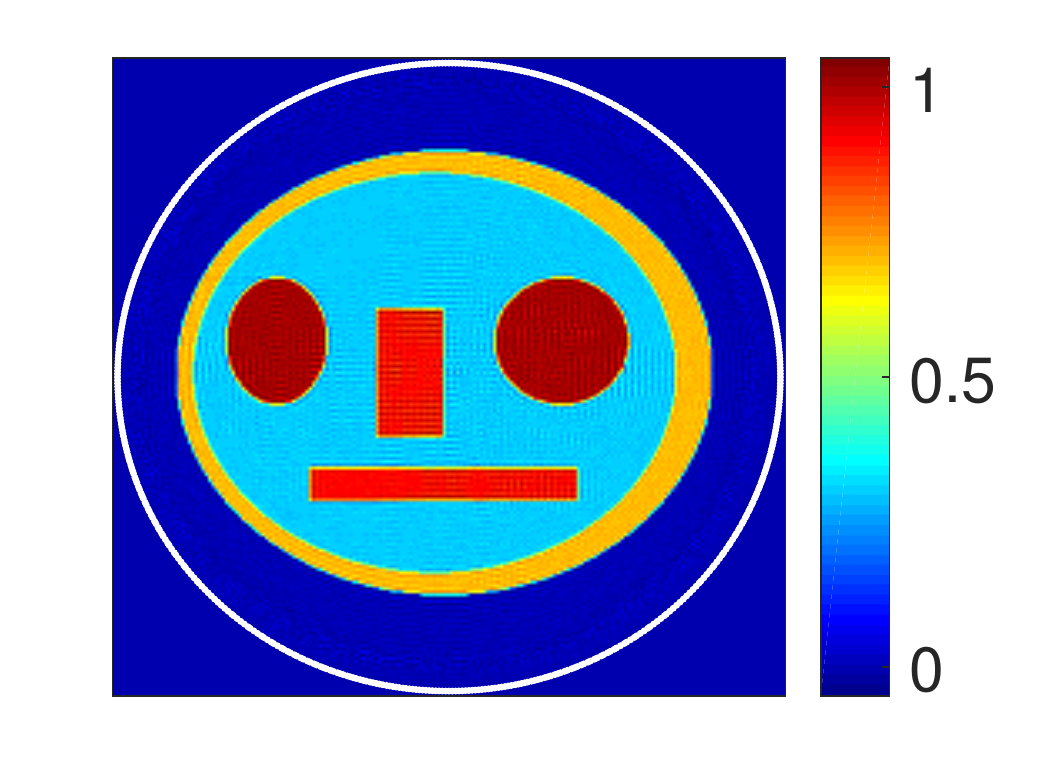}
\includegraphics[width =0.3\textwidth]{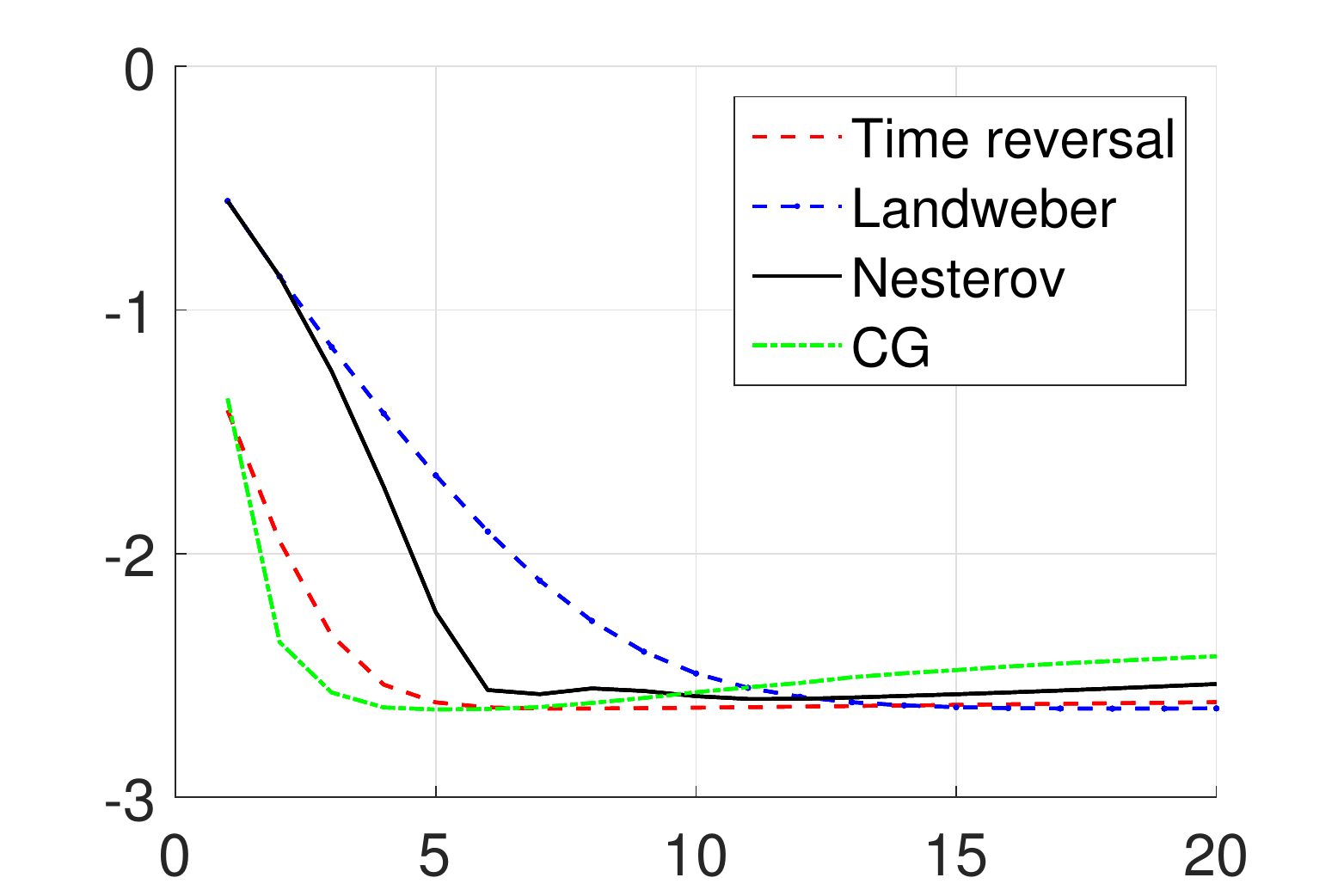}
\includegraphics[width =0.3\textwidth]{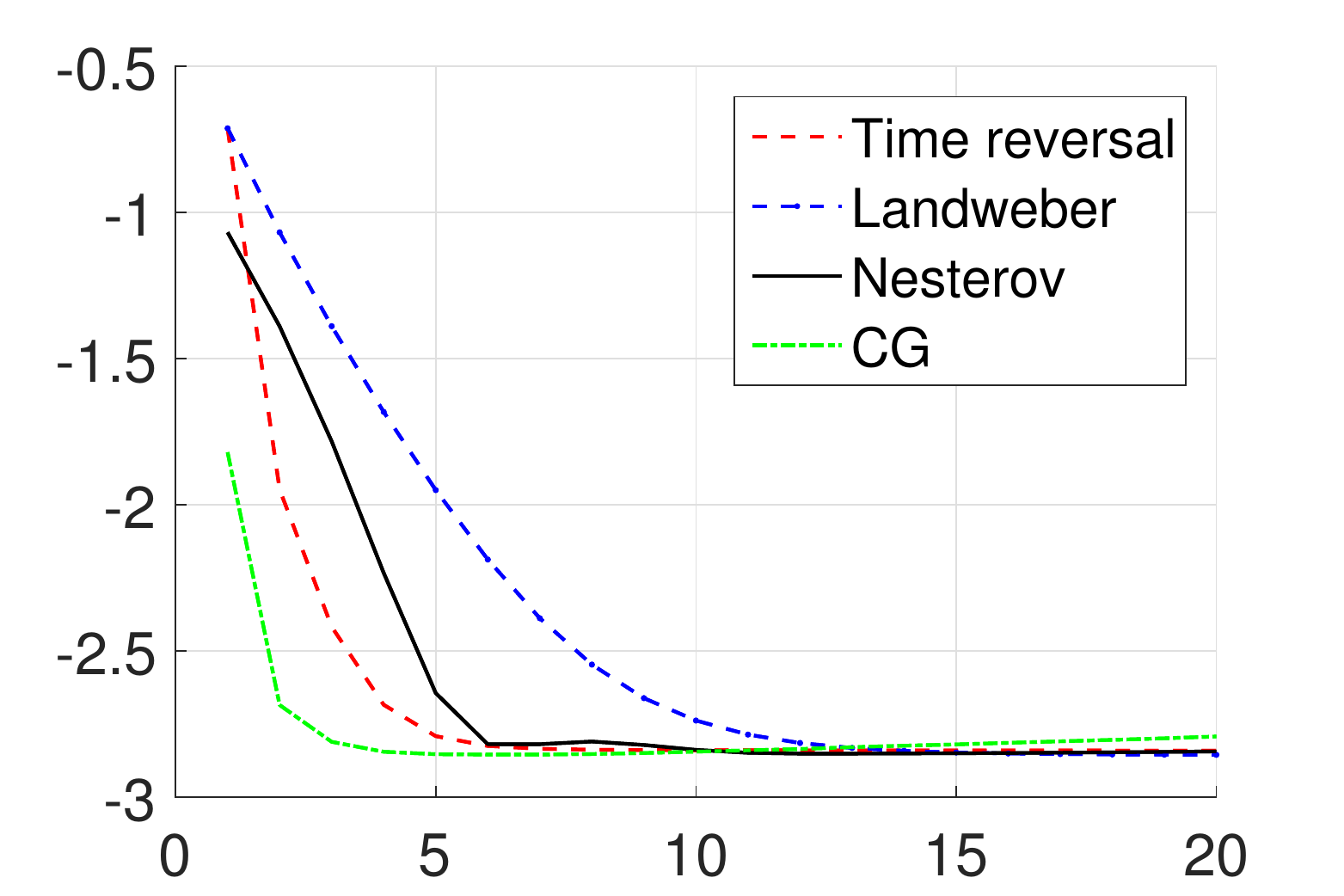}
\caption{{\scshape Test case \ref{t1}, inexact data.} Top left: Iterative time reversal,  Top center: Landweber's method,
top right: Nesterov's method. Bottom left: the CG method (all after 10 iterations). Bottom center: squared error.
Bottom right: squared residuals.\label{fig:test1noisy}}
\end{figure}

\begin{figure}[tbh!]\centering
\includegraphics[width =0.3\textwidth]{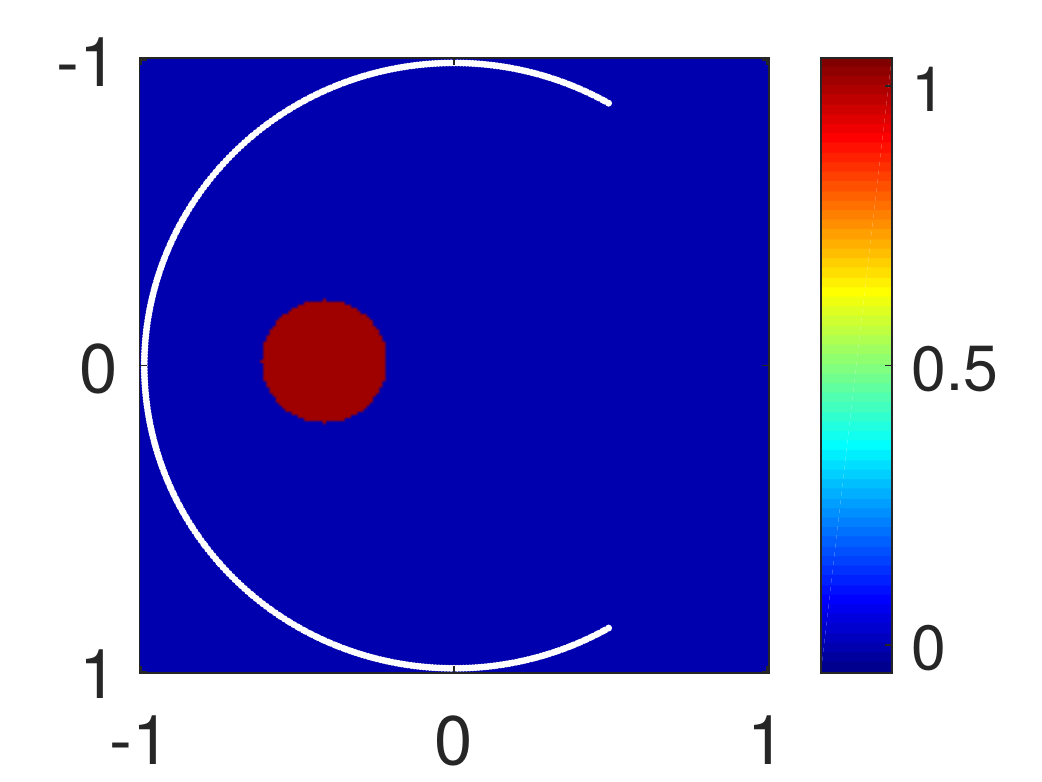}
\includegraphics[width =0.3\textwidth]{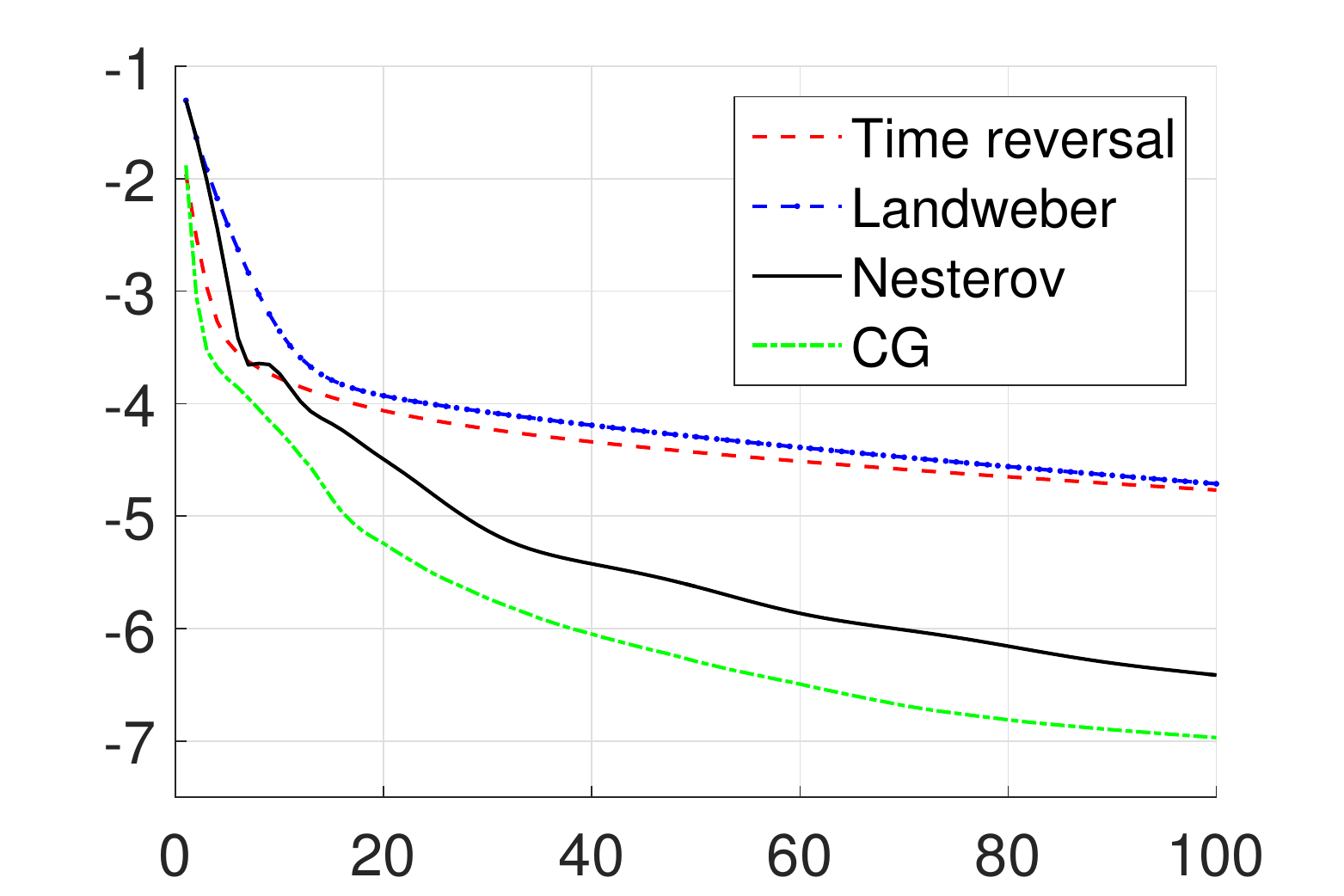}
\includegraphics[width =0.3\textwidth]{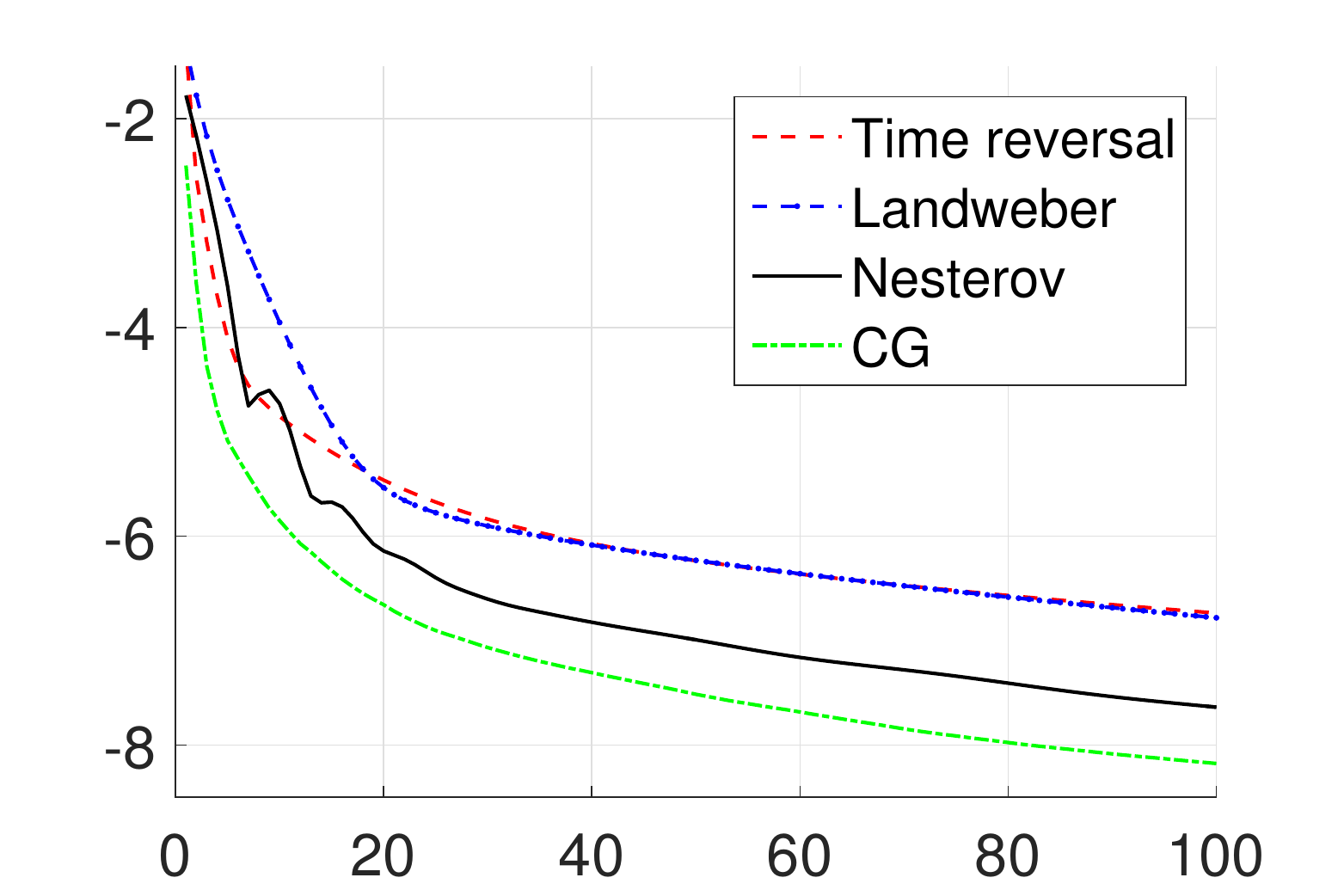}
\caption{{\scshape Test case \ref{t2},  exact data.}
Left: Initial pressure data $f$. Middle:  
Logarithm of squared reconstruction error $\norm{\fnum_n - \fnum}^2_2$
in dependence of the iteration  number.  Right:  Logarithm of residual $\norm{\Lnum_{N,M}\fnum_n - \gnum}^2_2$
in dependence of the iteration  number.\label{fig:test2}}
\end{figure}

\begin{figure}[tbh!]\centering
\includegraphics[width =0.3\textwidth]{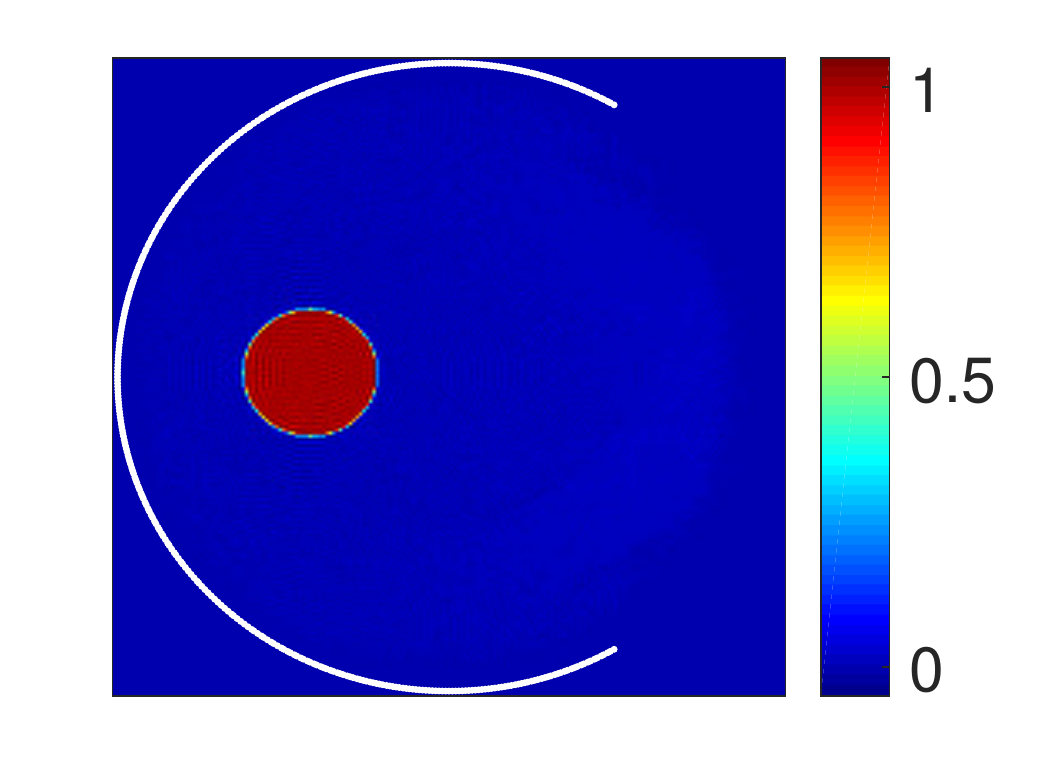}
\includegraphics[width =0.3\textwidth]{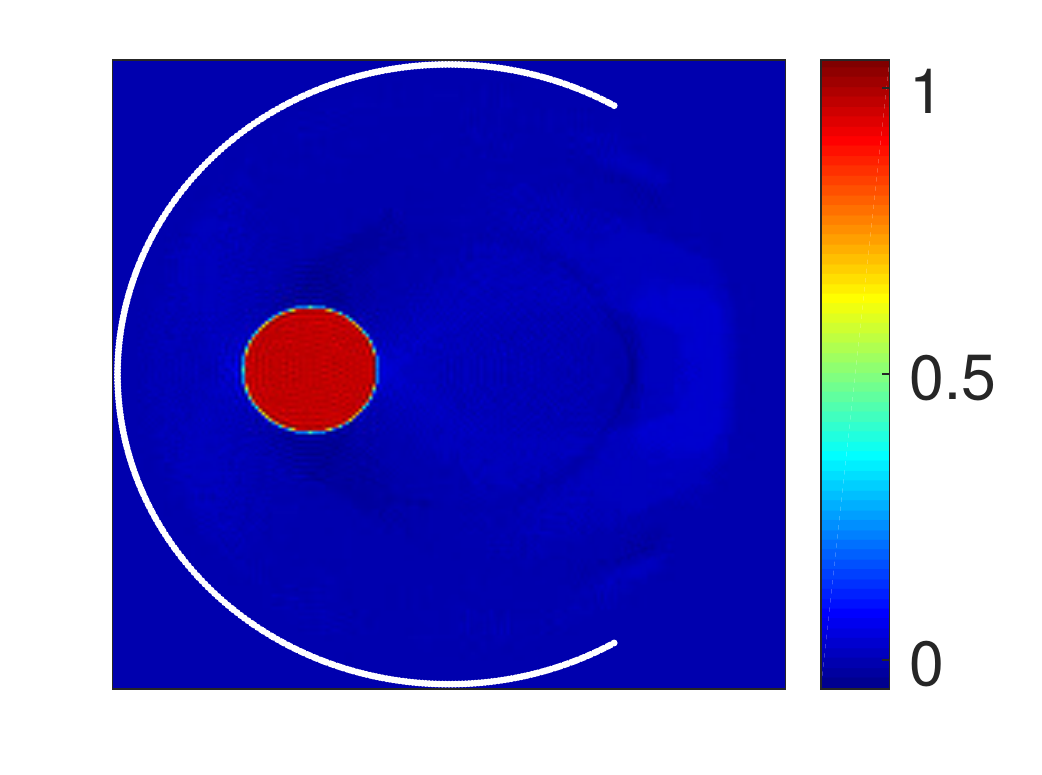}
\includegraphics[width =0.3\textwidth]{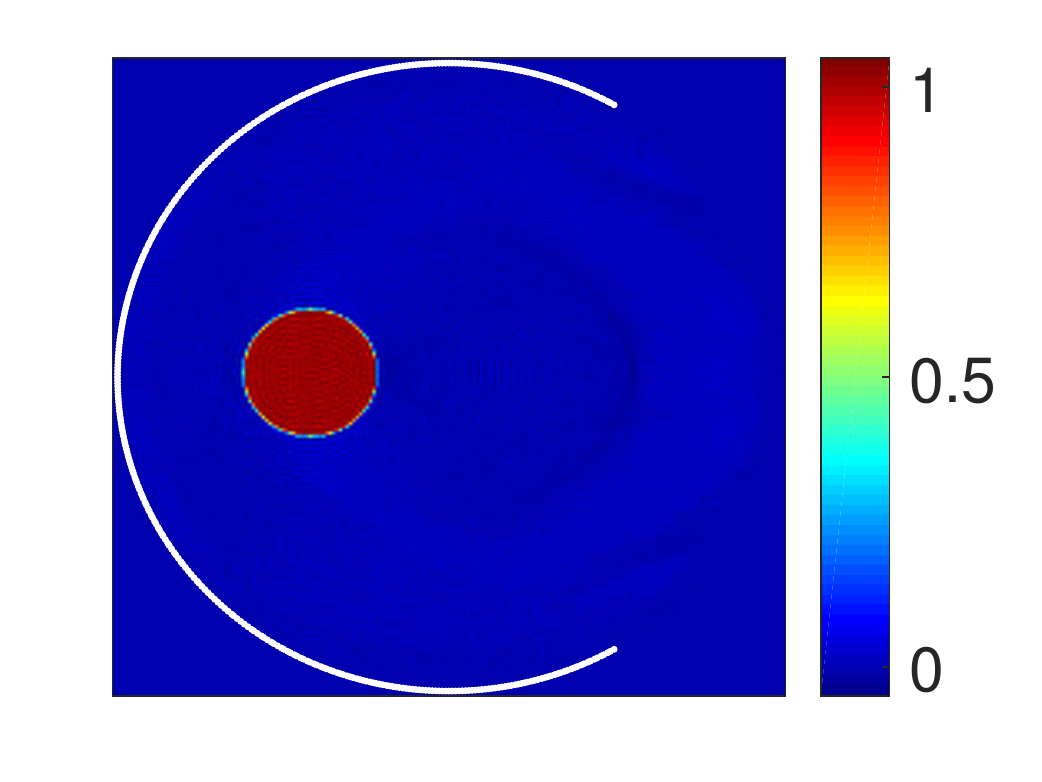}\\
\includegraphics[width =0.3\textwidth]{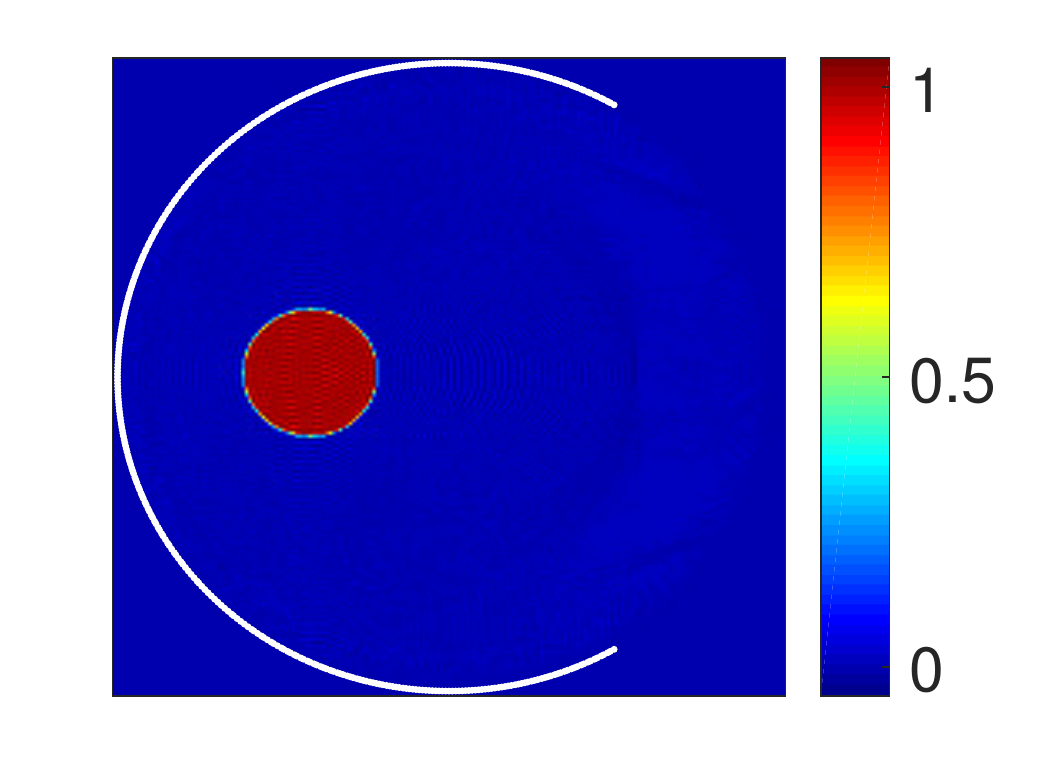}
\includegraphics[width =0.3\textwidth]{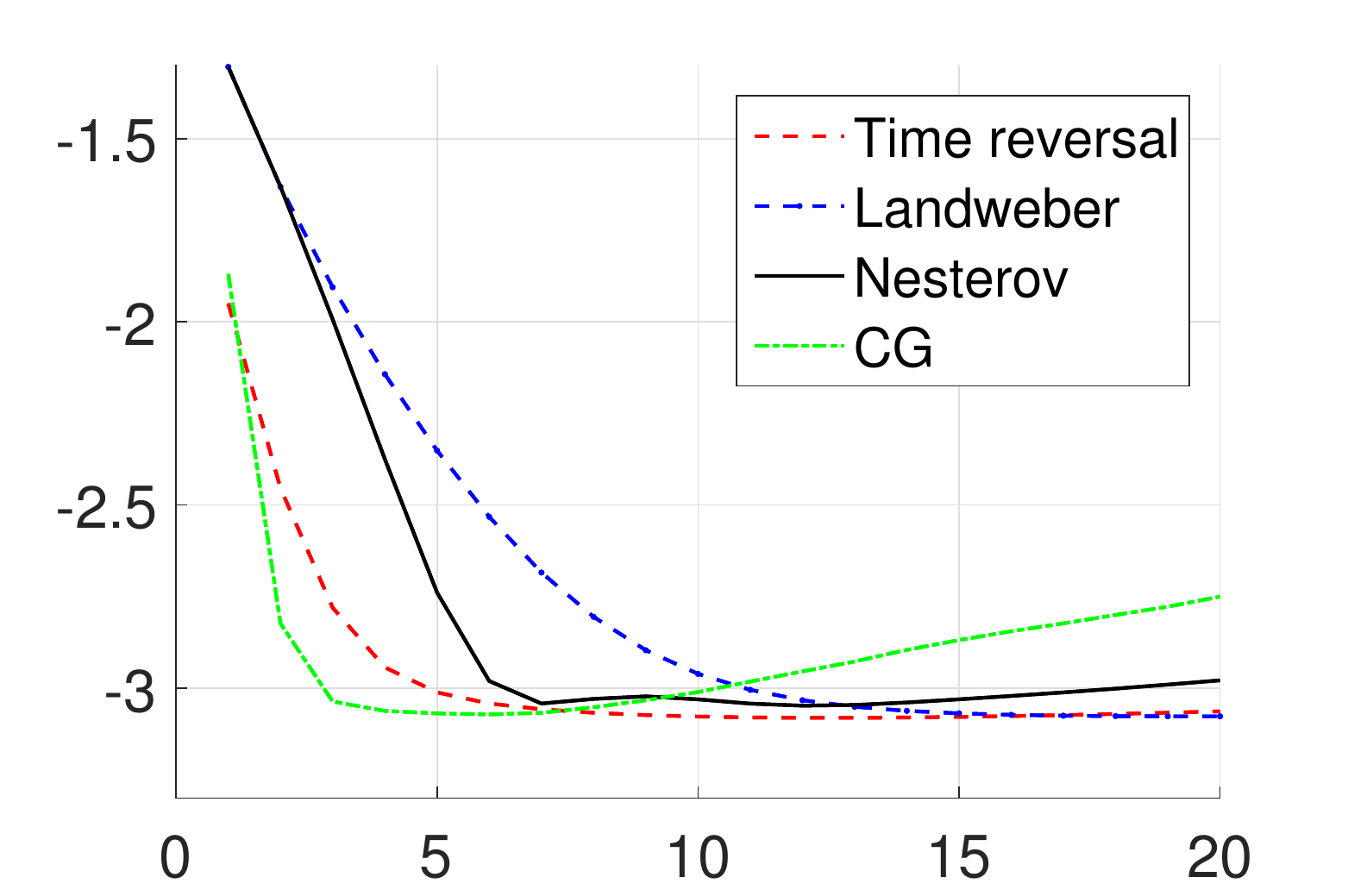}
\includegraphics[width =0.3\textwidth]{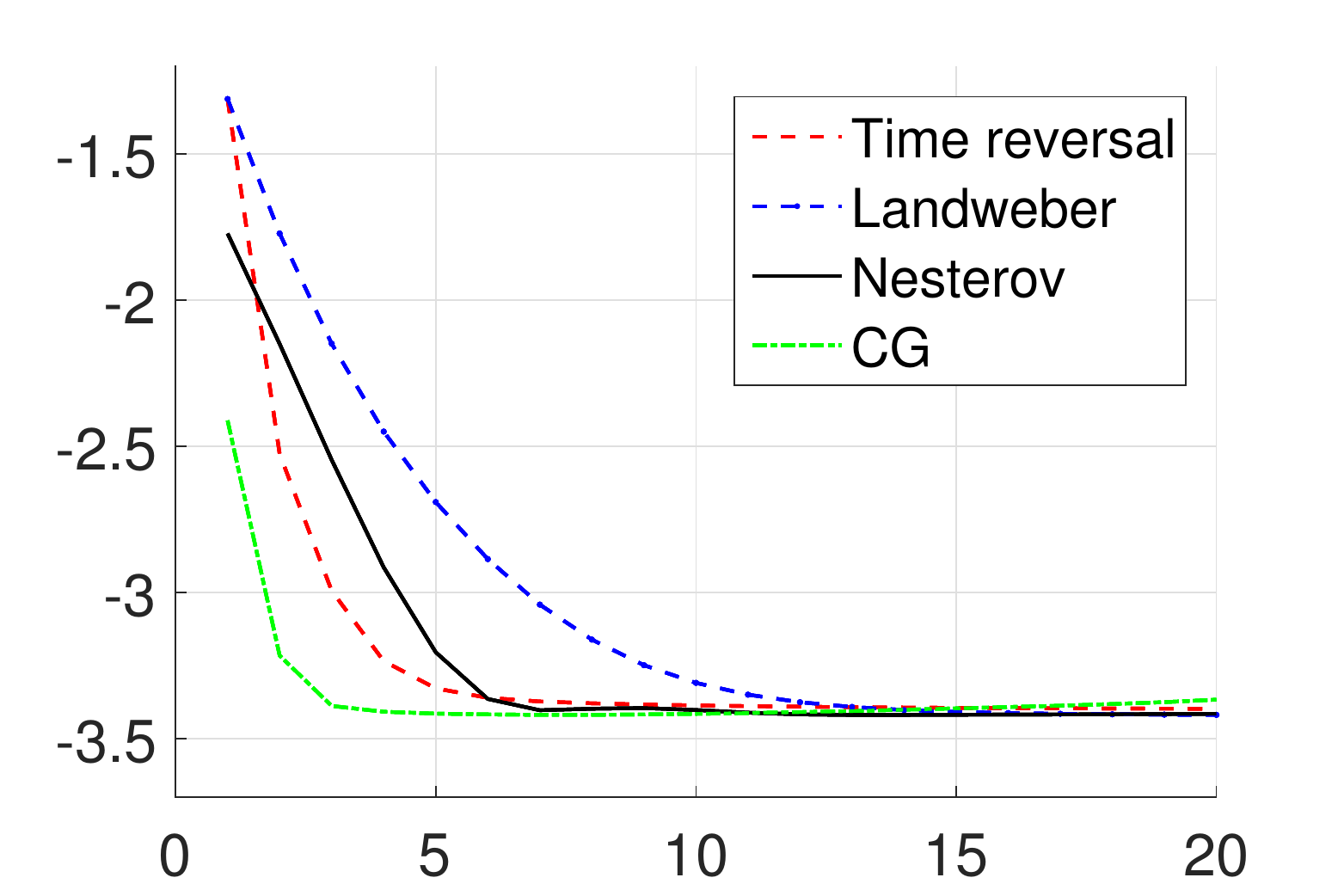}
\caption{{\scshape Test case \ref{t2}, inexact data.} Top left: Iterative time reversal,  Top center: Landweber's method,
top right: Nesterov's method. Bottom left: the CG method (all after 10 iterations). Bottom center: squared error.
Bottom right: squared residuals.\label{fig:test2noisy}}
\end{figure}

Note that we do not show results using the  Landweber's method proposed in~\cite{belhachmi2016direct}.  Due to the
smoothing operator  $-\Delta^{-1}$ (which is the adjoint of  the embedding $H^1_0(\Om) \hookrightarrow L^2(\Om)$),
that method is much slower than the Landweber's method presented in the present article.
On the other hand, the  application of  $-\Delta^{-1}$  may have the advantage of stabilizing the iteration.

\subsection{Test case \ref{t2}:  Partial  data, visible phantom}
\label{sec:test2}

As next  test case we investigate the case of partial data where all singularities
of the  phantom are  visible.  As before we compare iterative time reversal, Landweber's,
Nesterov's, and the CG methods using the sound speed given in \eqref{eq:ss-nontrapping}. We again take $N=200$, $R=1$, $T=1.5$ and $M=800$.
The the phantom is shown in shown left image in Figure
\ref{fig:test2}.   One notices that the partial data have  been collected on an arc with opening  angle $4\pi/3$.   The middle and right image in 
 Figure~\ref{fig:test2} show the reconstruction error and the residuals 
 depending on the iteration index $n$. One observes a similar 
 asymptotic  behavior as for the test case \ref{t1} with exact data. In particular 
 the CG iteration is again the most rapidly converging method.
Also the  convergence behavior in the first iterations is similar to the complete 
data case;  due to space limitations we do not show the corresponding pictures.

To avoid  inverse crimes and to investigate the behavior of the algorithms under  real life  scenario,
we repeated the simulations with inexact data where we  simulated the data on a different grid (using $N=350$ and $M=1300$) and further added Gaussian noise to the data (again with a standard deviation equal to $5\%$ of the $L^2$-norm of the exact data).
The reconstruction results from inexact data are shown in Figure~\ref{fig:test2noisy}. They clearly
demonstrate that all schemes provide good results. The CG method is again the fastest.  The error $     \norm{\Lnum_{N,M}\fnum - \gnum^\delta}_2$  in the data is 0.0258.  The residuals after 10 iterations are 0.2208 for the iterative time reversal,  0.0196 for the CG,  0.0199 for Nesterov's, and  0.0222 for the Landweber's methods.  This in particular also shows that the discrepancy principle yields
a well defined stopping index with an reconstruction error in the order of the data error.

\begin{figure}[tbh!]\centering
\includegraphics[width =0.3\textwidth]{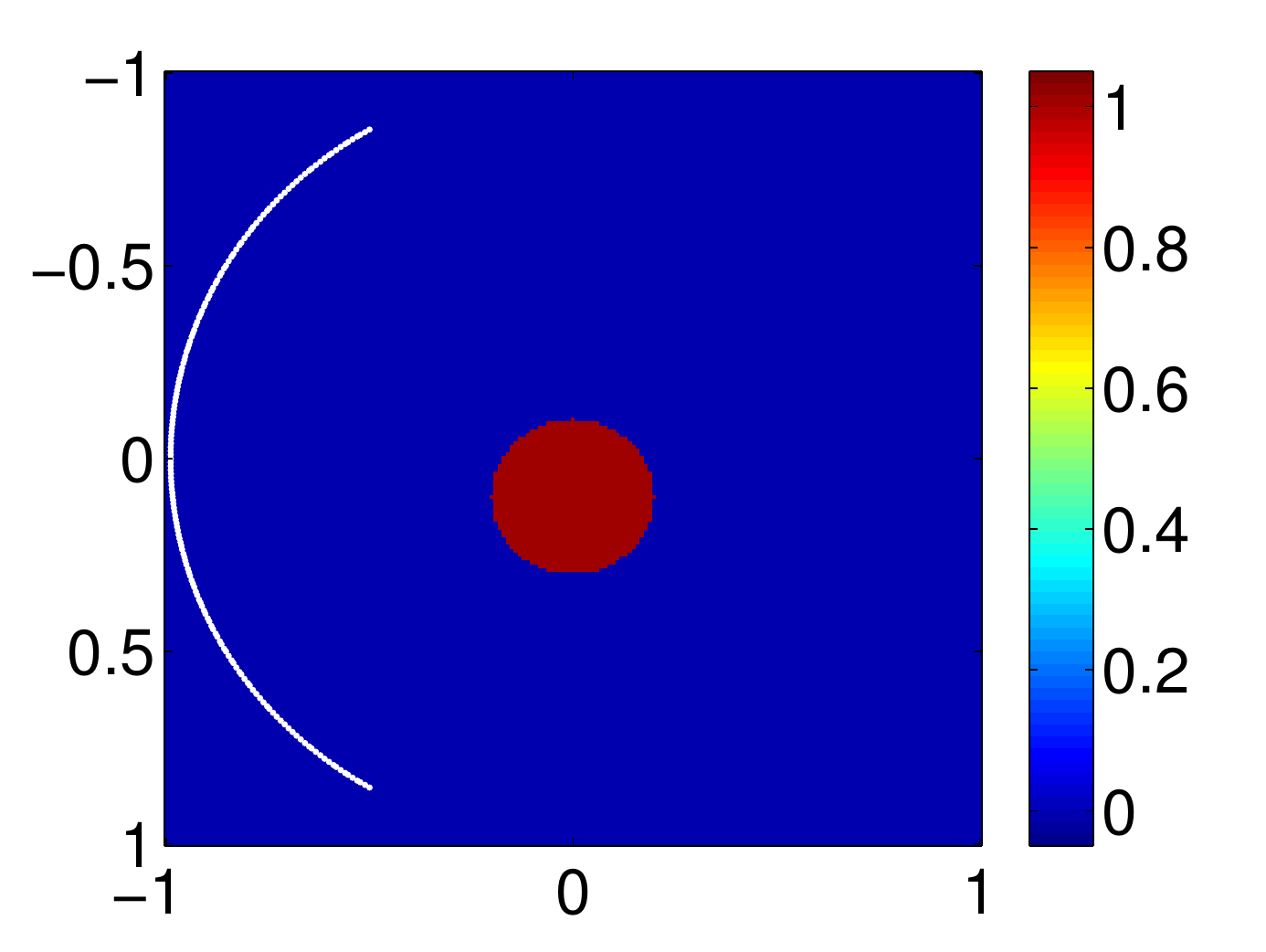}
\includegraphics[width =0.3\textwidth]{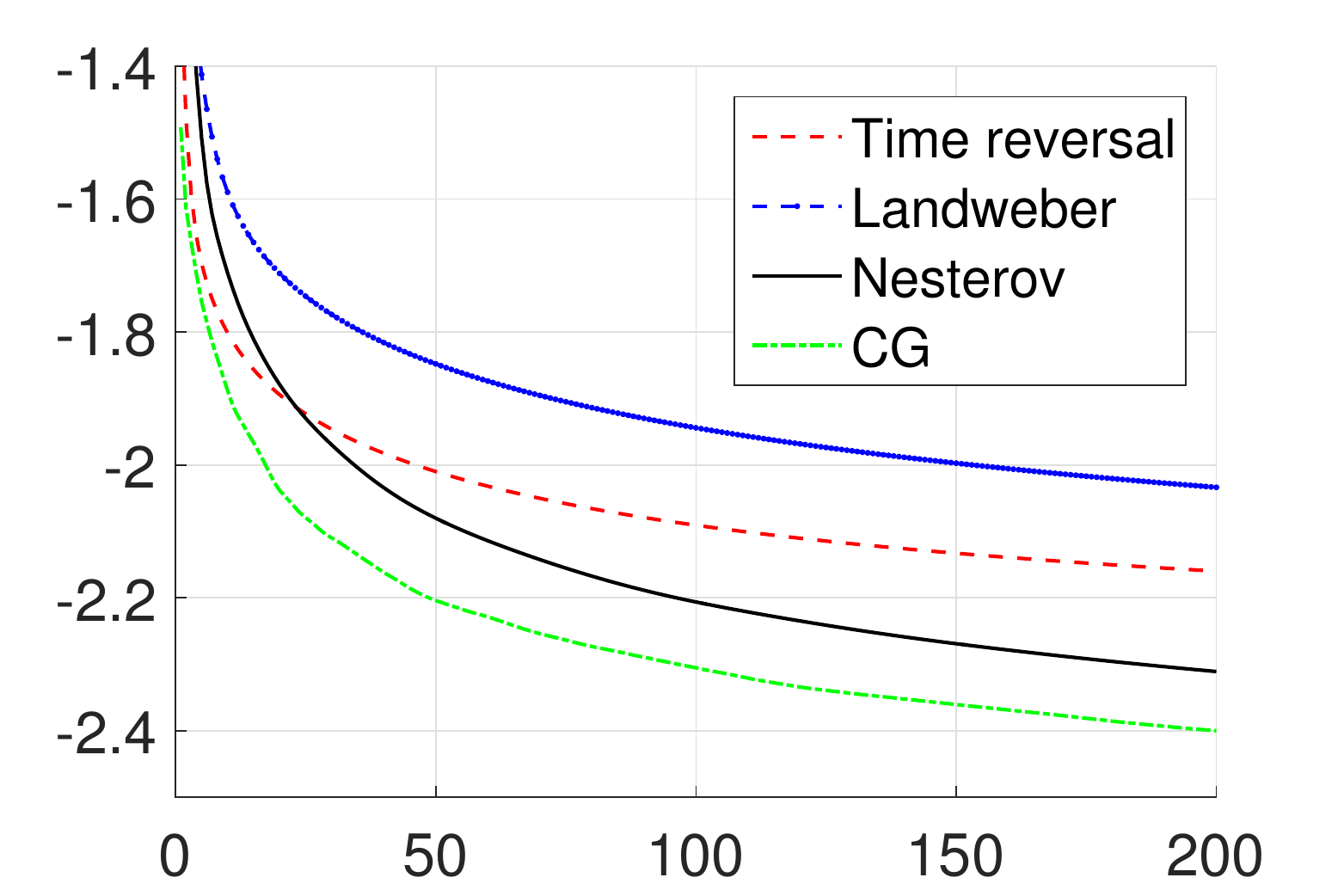}
\includegraphics[width =0.3\textwidth]{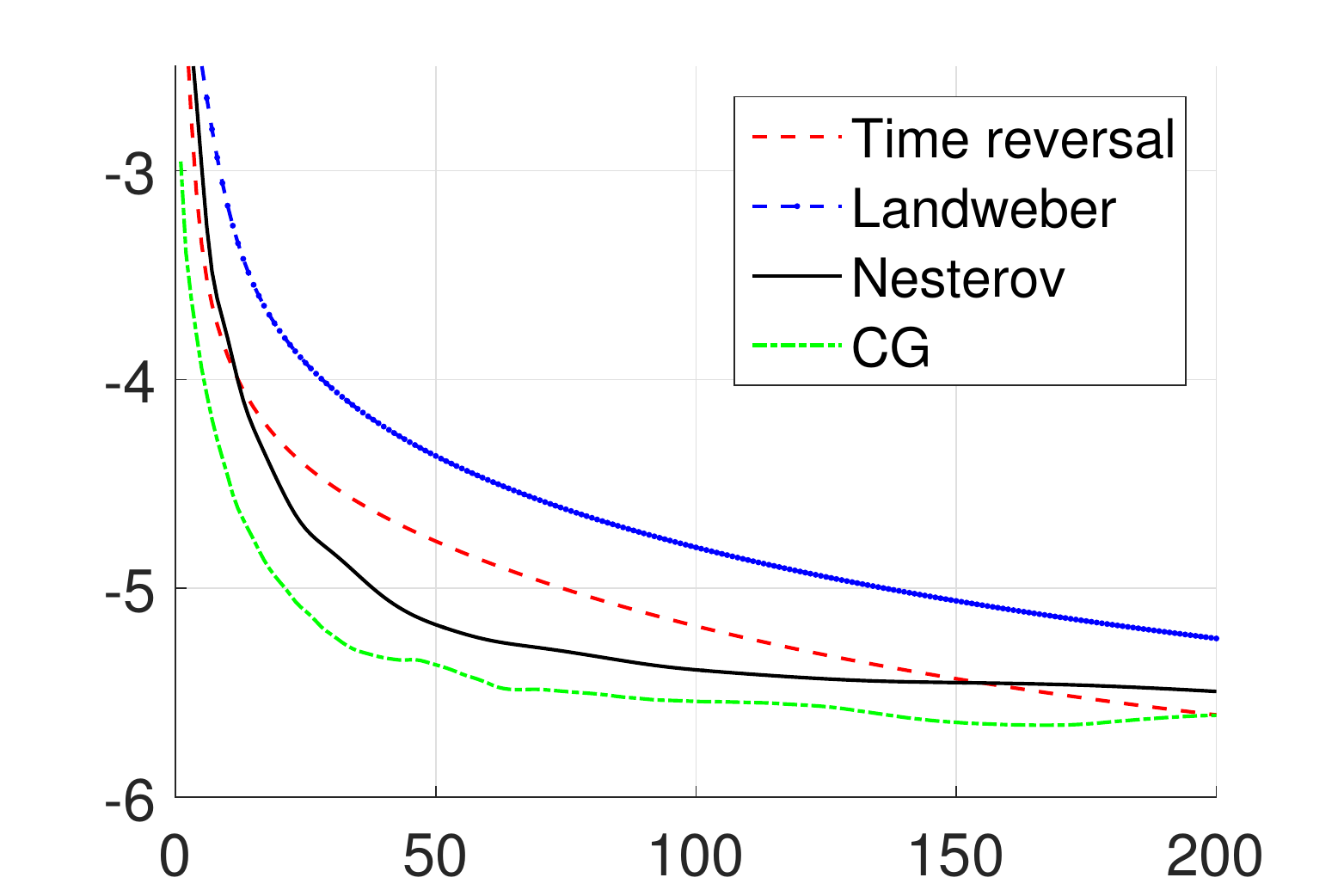}\\
\includegraphics[width =0.3\textwidth]{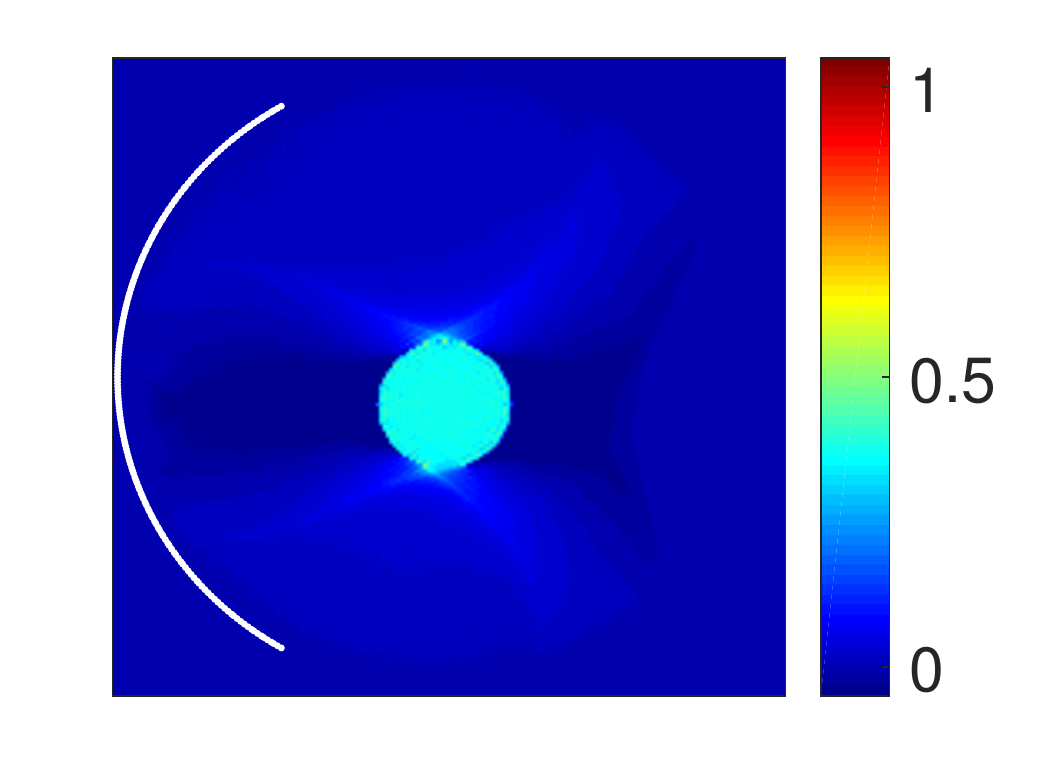}
\includegraphics[width =0.3\textwidth]{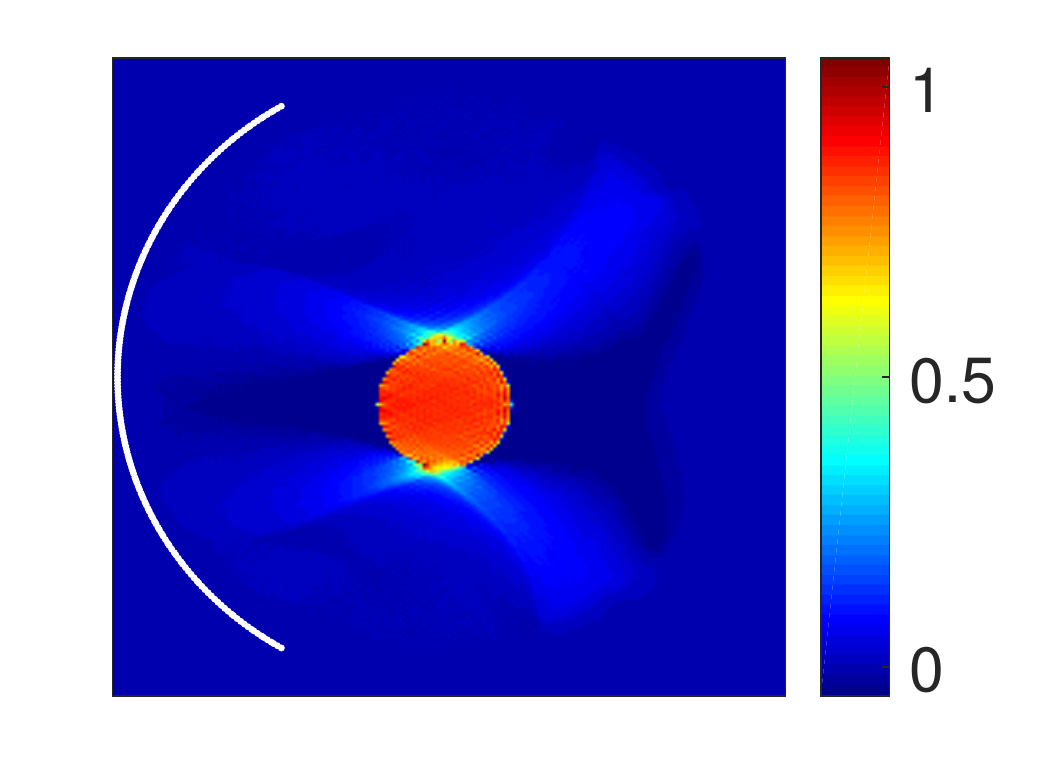}
\includegraphics[width =0.3\textwidth]{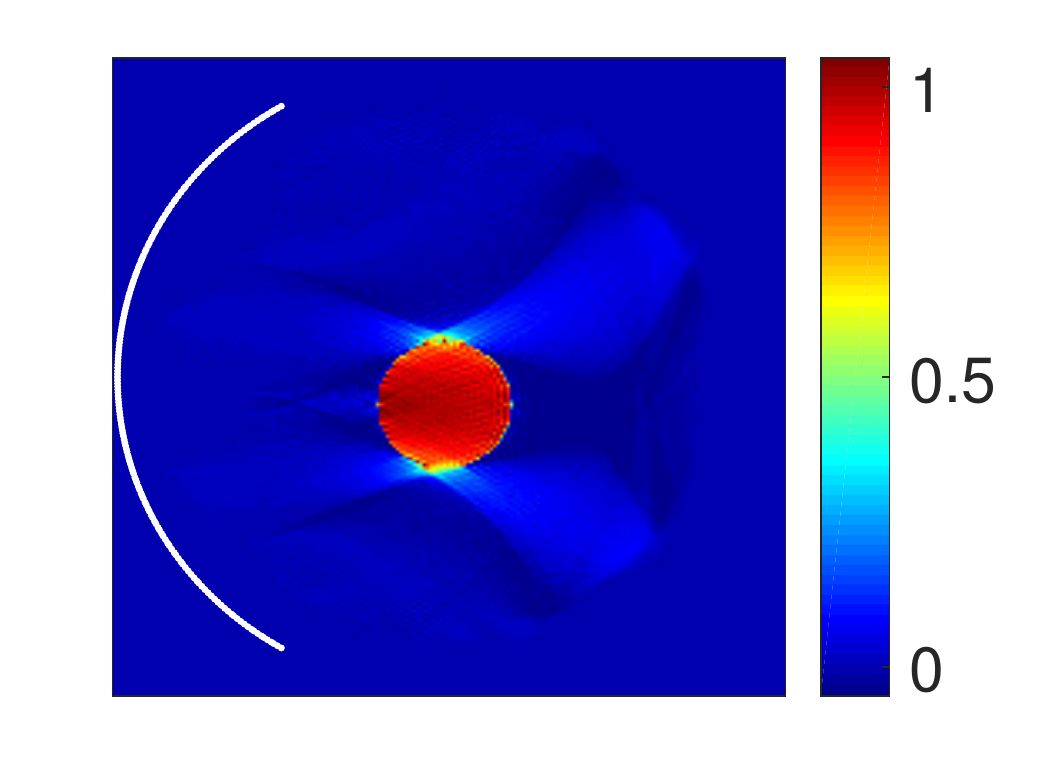}\\
\includegraphics[width =0.3\textwidth]{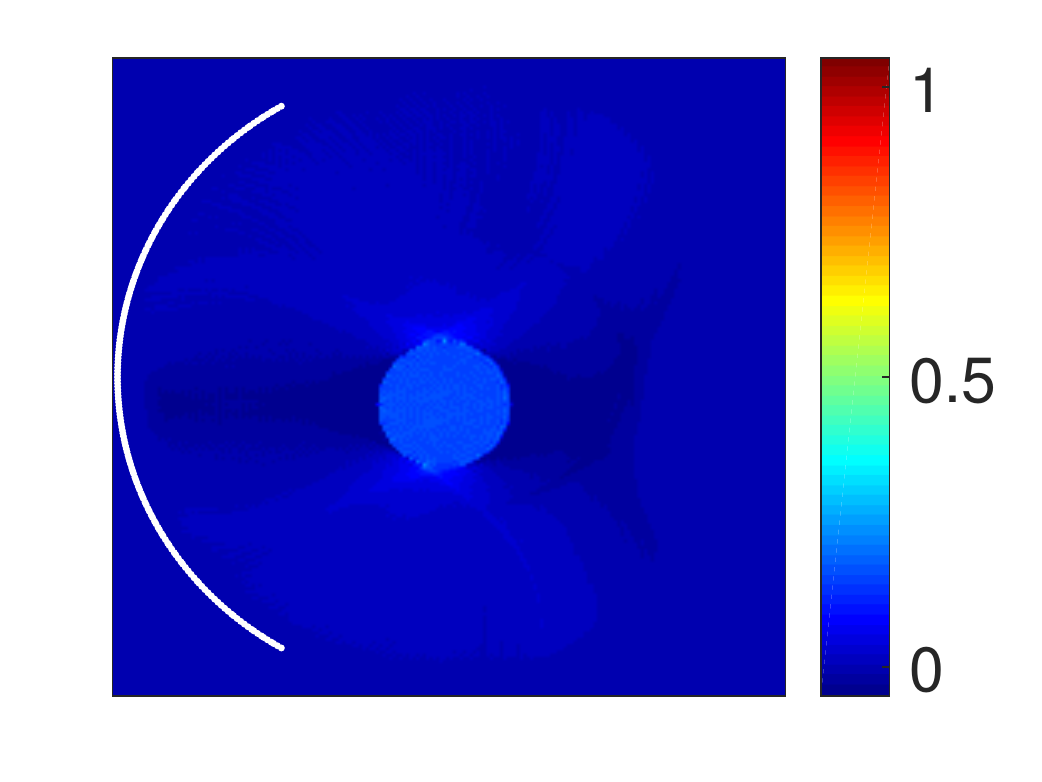}
\includegraphics[width =0.3\textwidth]{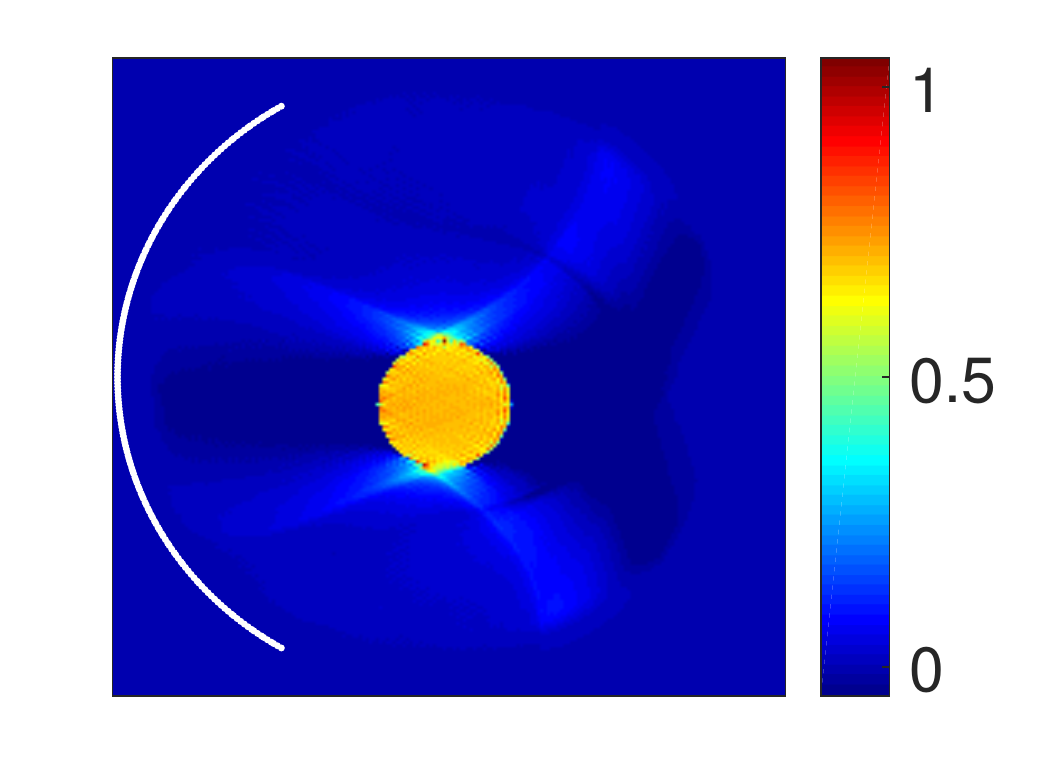}
\includegraphics[width =0.3\textwidth]{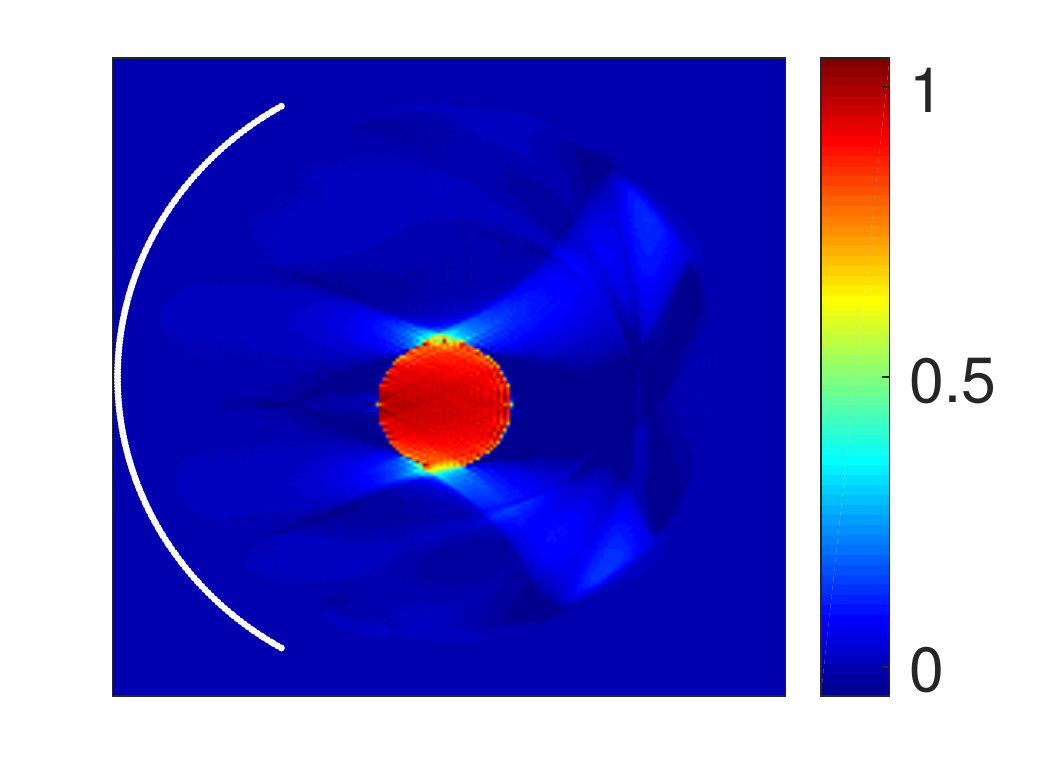}\\
\includegraphics[width =0.3\textwidth]{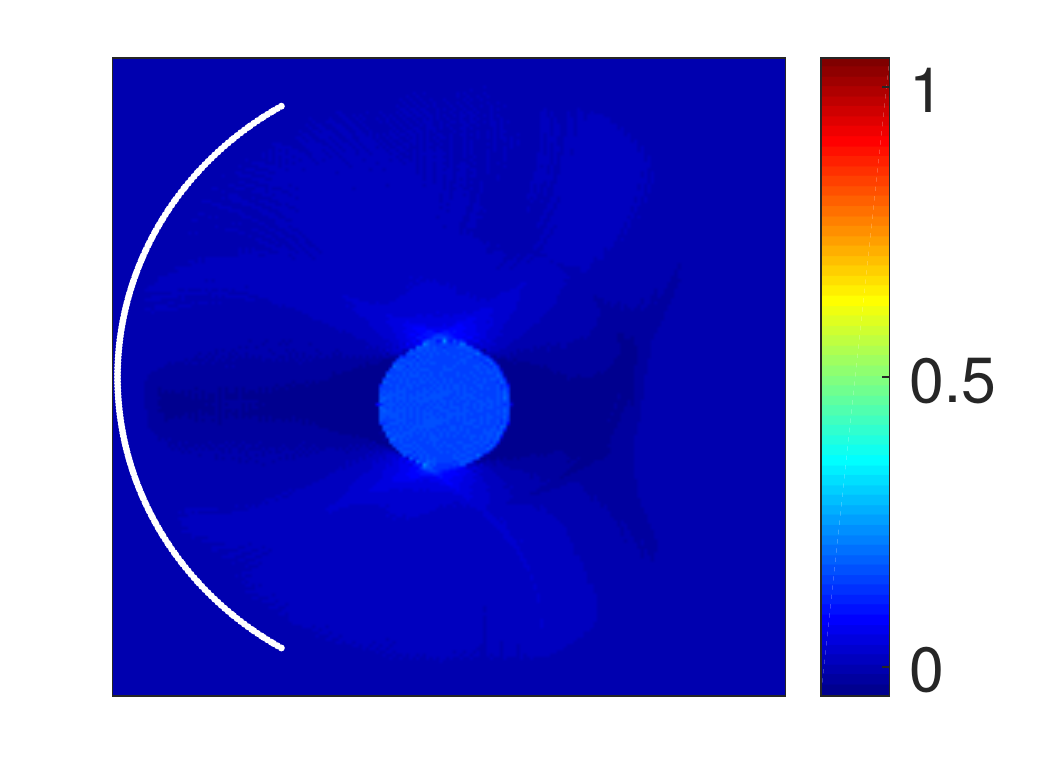}
\includegraphics[width =0.3\textwidth]{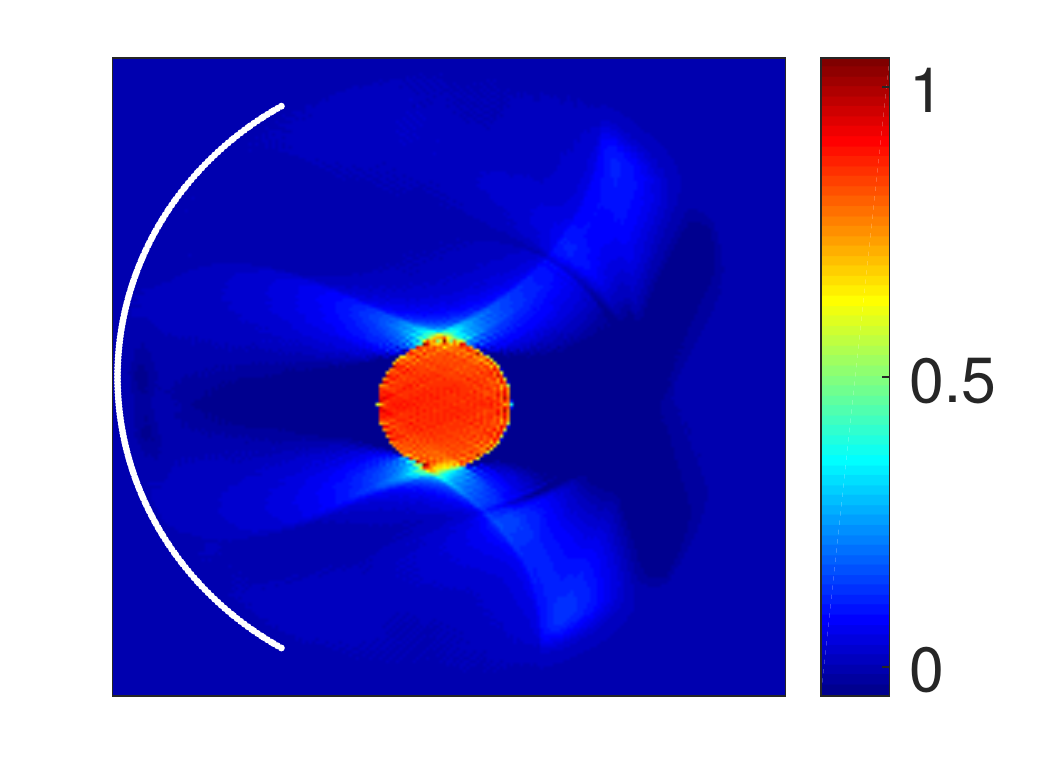}
\includegraphics[width =0.3\textwidth]{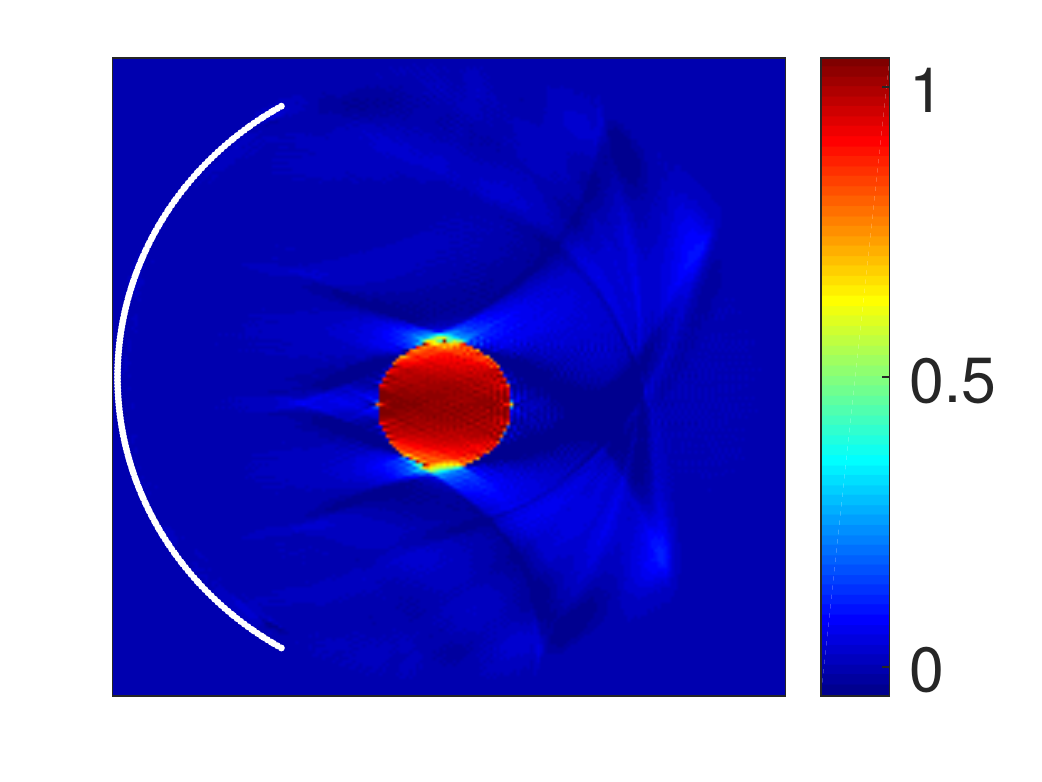}\\
\includegraphics[width =0.3\textwidth]{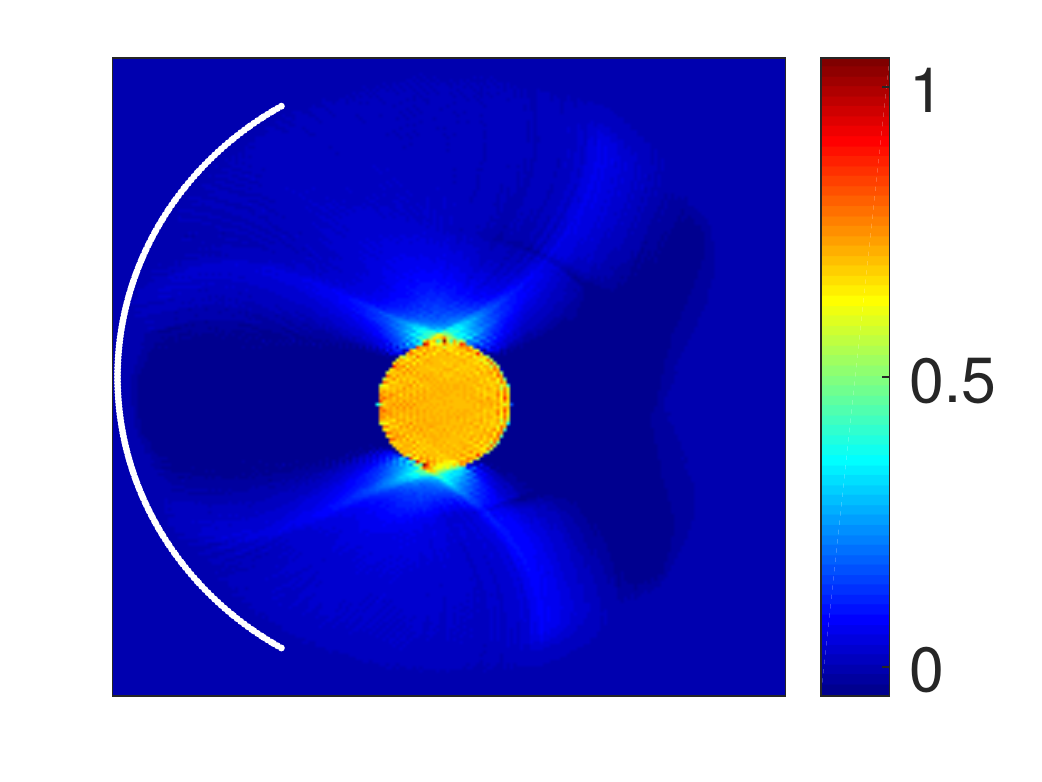}
\includegraphics[width =0.3\textwidth]{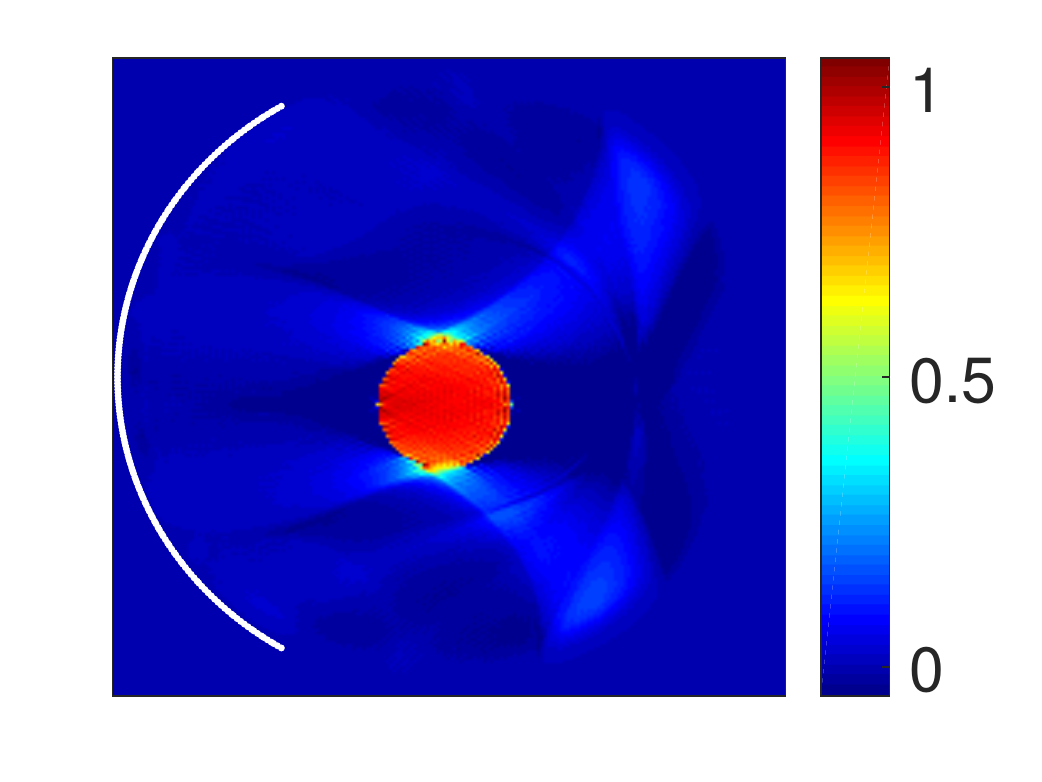}
\includegraphics[width =0.3\textwidth]{pics/case3/3211.pdf}\\
\caption{{\scshape Test case~\ref{t3}: partial data, invisible phantom}
Row 1: Initial pressure data $f$ (left), 
Logarithm of squared reconstruction error $\norm{\fnum_n - \fnum}^2_2$
(middle), and logarithm of residual $\norm{\Lnum_{N,M}\fnum_n - \gnum}^2_2$
(right).
Row 2: Iterative time reversal (after 1, 10 and 200 iterations).
Row 3: Landweber's method (after 1, 10 and 200 iterations).
Row 4: Nesterov's method (after 1, 10 and 200 iterations).
Row 5: the CG method (after 1, 10 and 200 iterations).\label{fig:test3}}
\end{figure}

\begin{figure}[tbh!]\centering
\includegraphics[width =0.3\textwidth]{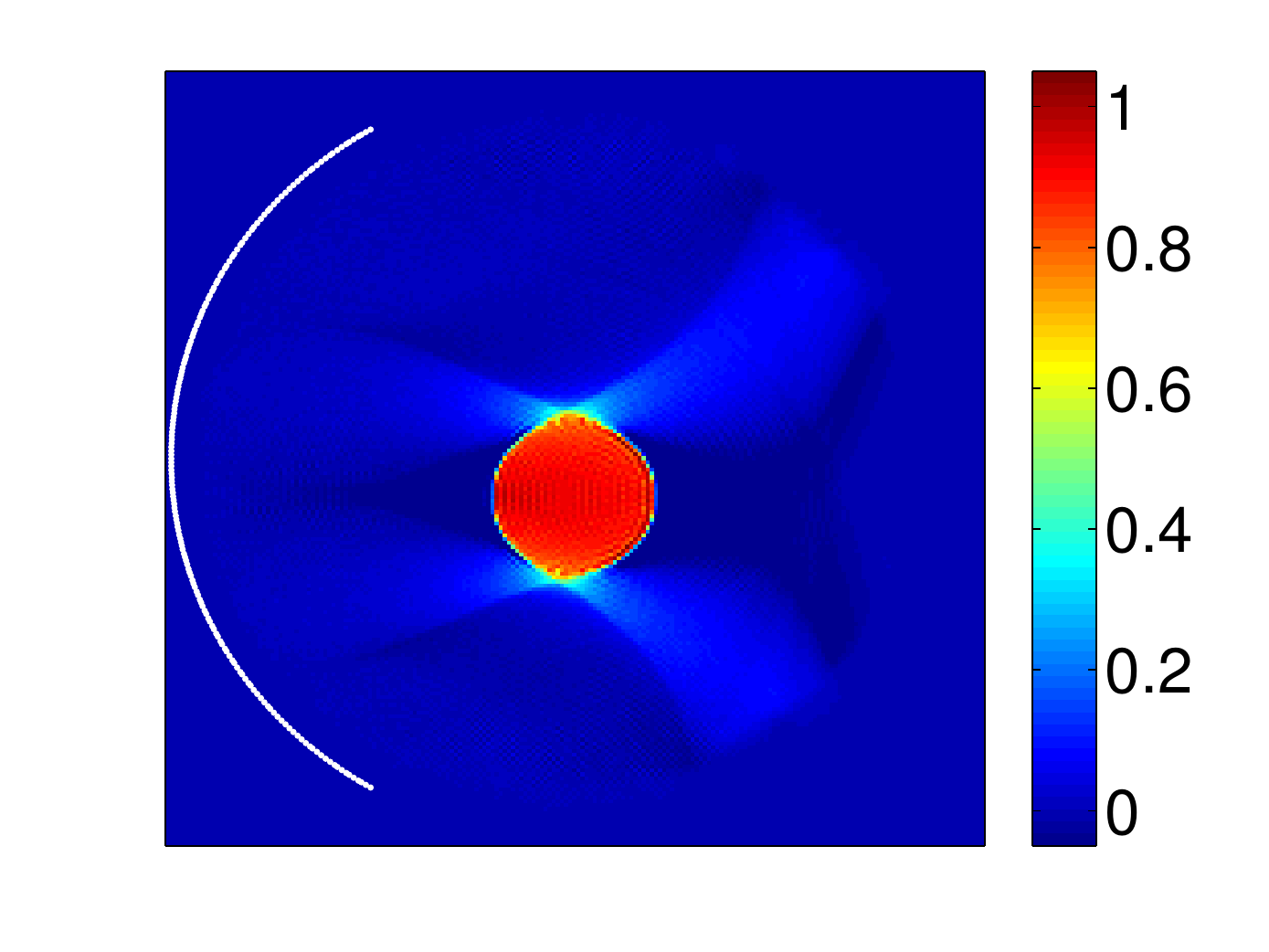}
\includegraphics[width =0.3\textwidth]{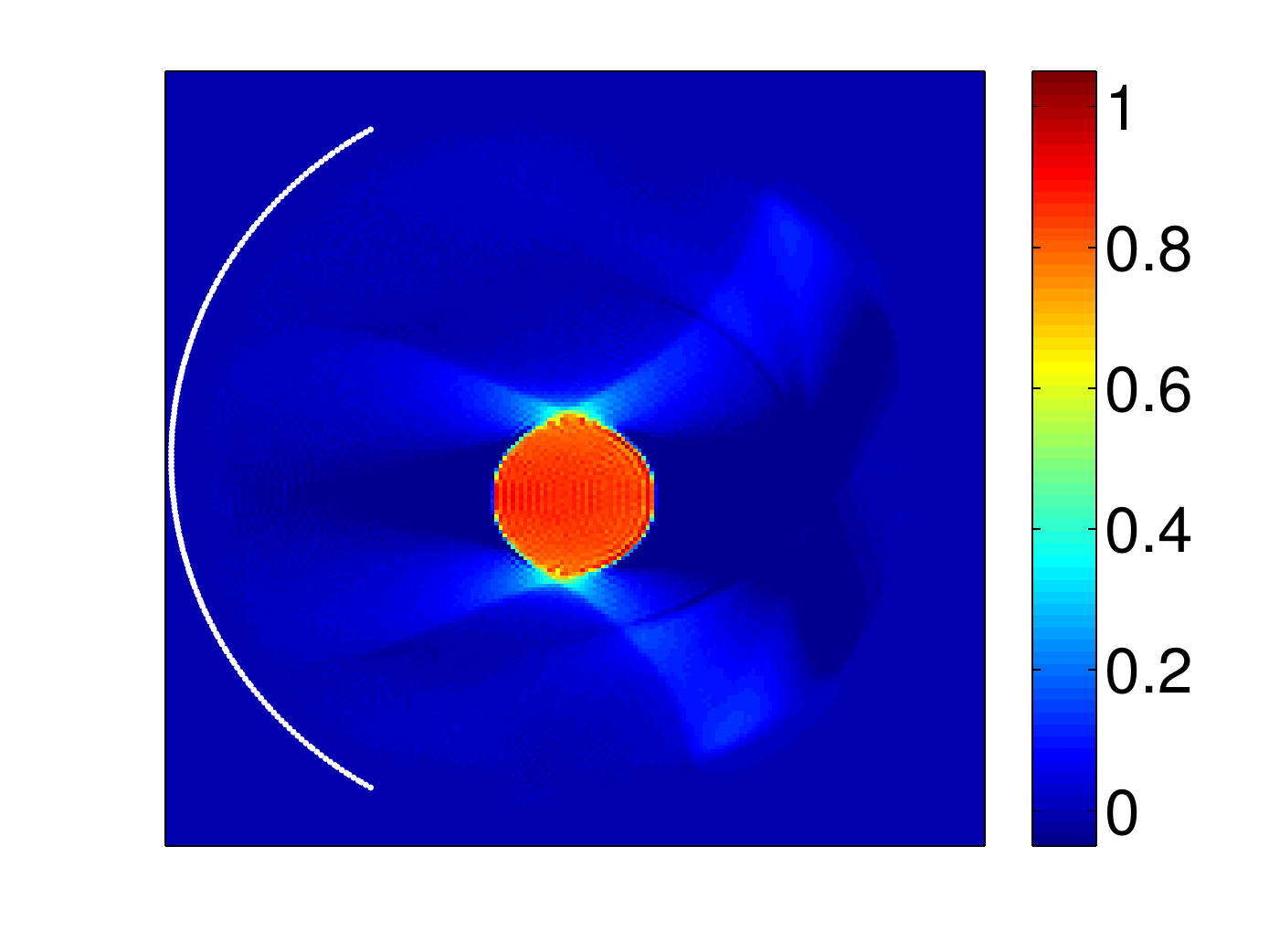}
\includegraphics[width =0.3\textwidth]{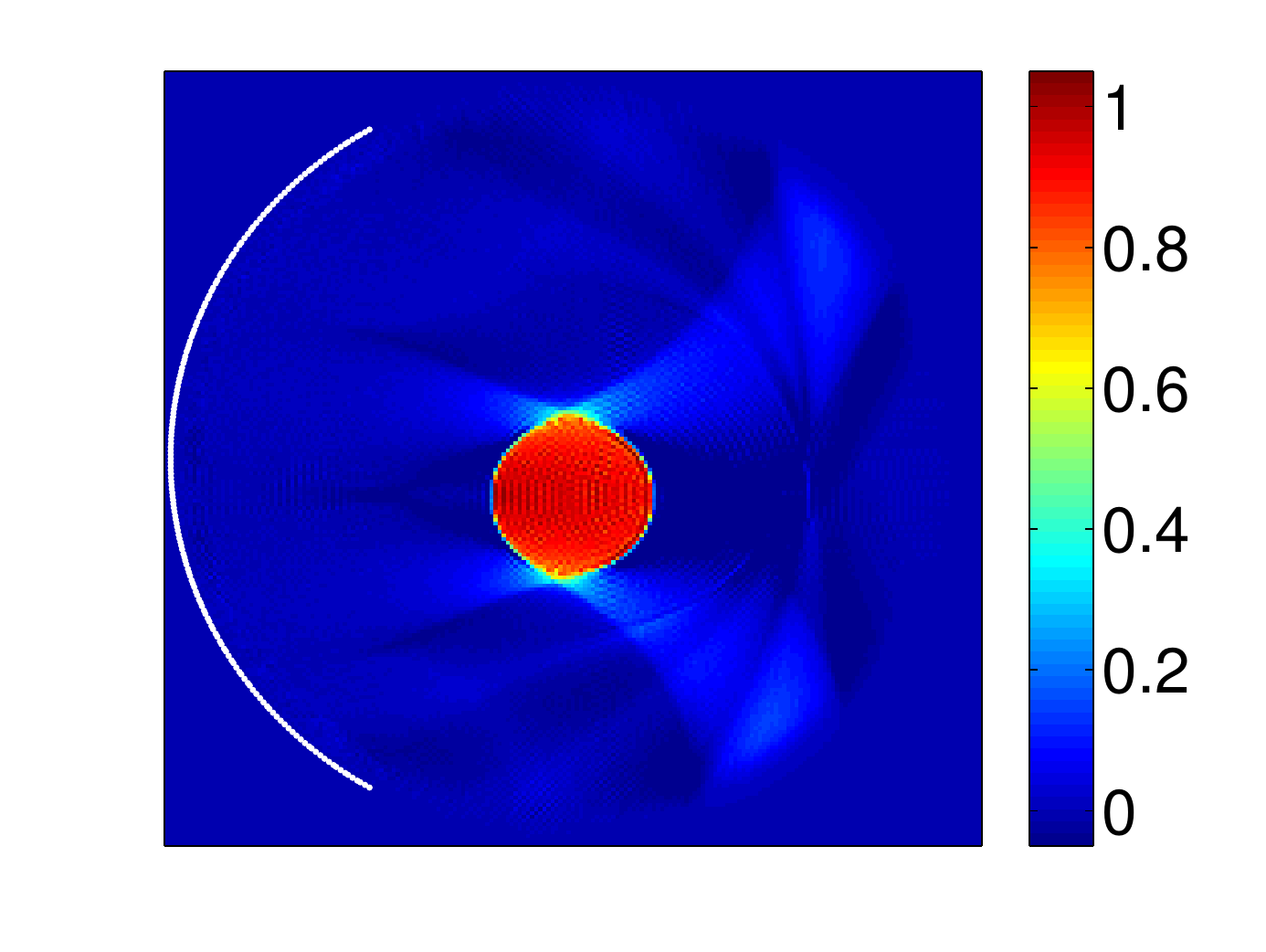}\\
\includegraphics[width =0.3\textwidth]{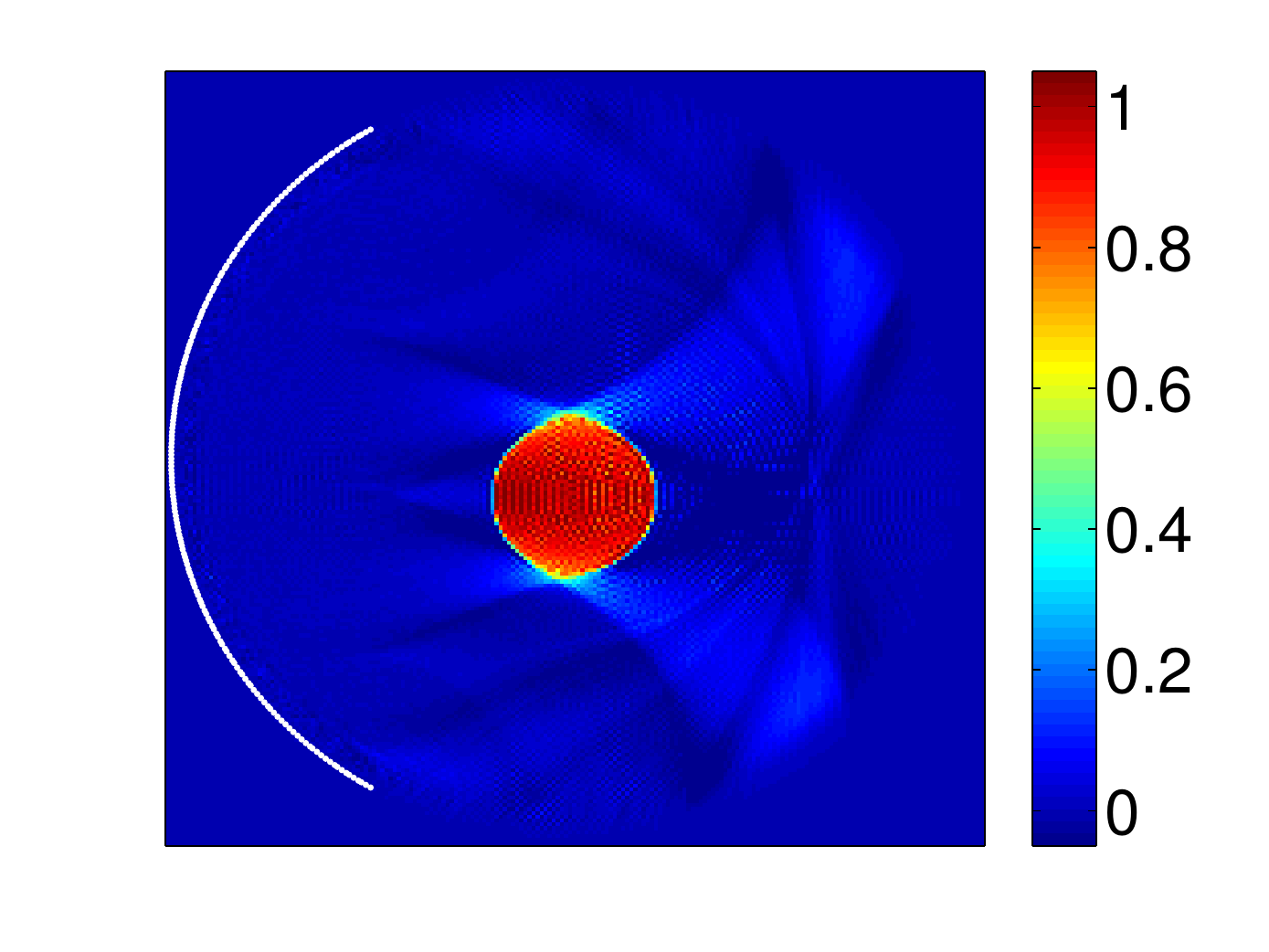}
\includegraphics[width =0.3\textwidth]{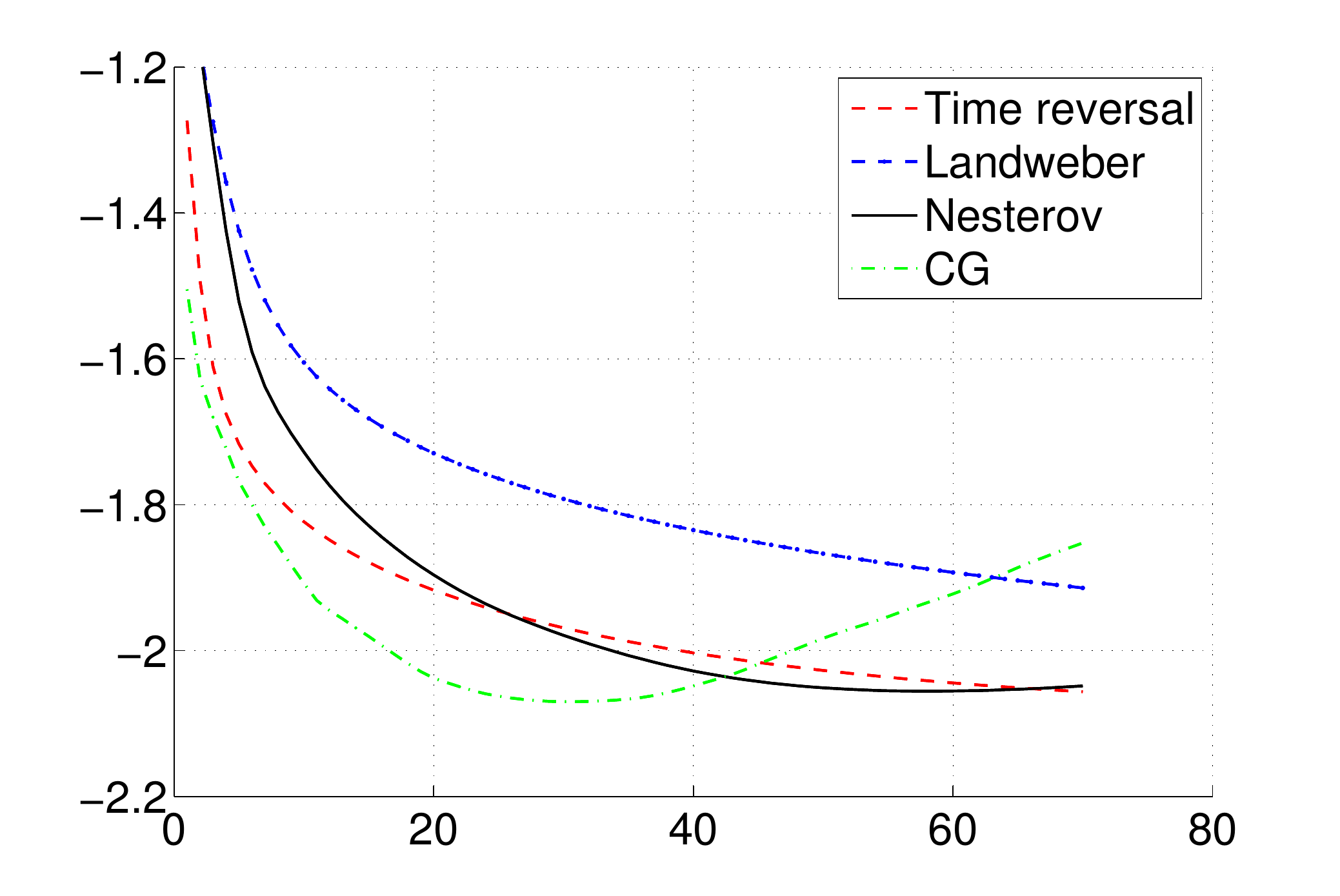}
\includegraphics[width =0.3\textwidth]{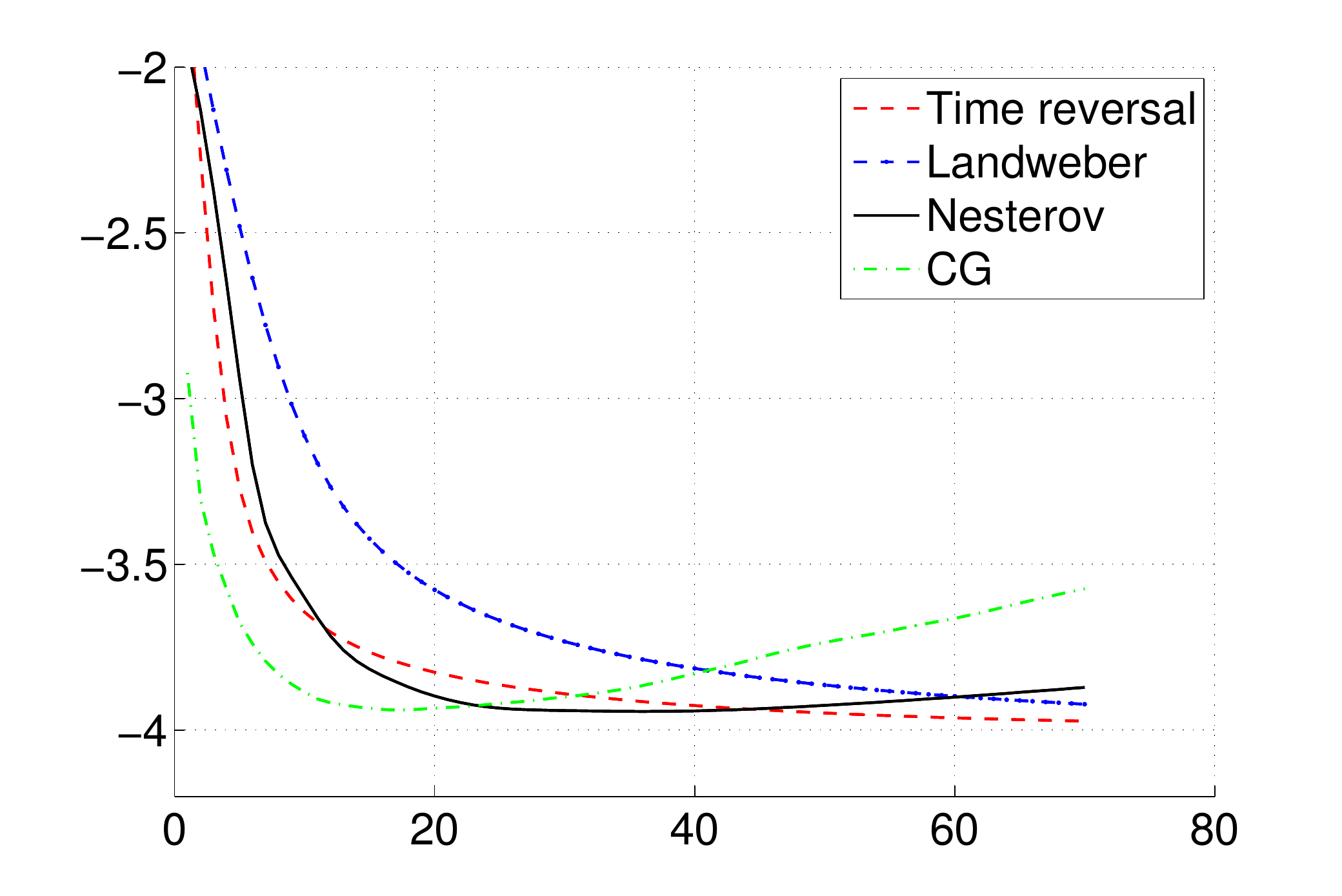}
\caption{{\scshape Test case~\ref{t3}, inexact data.} Top left: Iterative time reversal,  Top center: Landweber's method,
top right: Nesterov's method. Bottom left: the CG method (all after 30 iterations). Bottom center: squared error.
Bottom right: squared residuals.\label{fig:test3noisy}}
\end{figure}

\subsection{Test case \ref{t3}:  Partial data, invisible phantom}
\label{sec:test3}

In this case we investigate  the ill-posed problem, since the invisibility condition holds. We use again the sound speed given in \eqref{eq:ss-nontrapping}
and take $N=200$, $R=1$, $T=1.5$ and $M=800$.
The phantom is shown in the  top left image  in Figure
\ref{fig:test3}.  The partial data are collected on an arc
with opening angle $2\pi/3$ (the measurement curve). 
The invisibility condition holds in this setup.
Rows 2 to 5 in Figure~\ref{fig:test3} show reconstruction results  with  iterative time reversal,
Landweber's, Nesterov's,  and the CG methods after 1, 10 and 200 iterations.
Because the inverse problem is ill-posed no convergence rate results for the Landweber, Nesterov's, and CG methods are available. And indeed one observes that the reconstruction results are worse
compared to the  the case of a completely visible phantom. Again we investigated the  convergence  behavior more carefully. For that purpose in the top row of~\ref{fig:test3} 
we again shows the reconstruction error and the residuals   depending on the iteration index $n$. While the residuals tend to zero quite fast, reconstruction error now decreases much slower than in the previous examples.  Nevertheless also in this situation the CG iteration clearly yields the smallest reconstruction error for a given number of iterations.

Again we repeated the simulations with inexact data where we  simulated the data on a
different grid and further added Gaussian noise to the data.   Due the ill-posedness of the problem we cannot expect complete convergence for noisy data. In fact, as can be seen in Figure~\ref{fig:test3noisy} all iterations show the typical semi-convergence behavior:
The error decreases until a certain optimal index $n^\star$, after which the error starts to increase. Stopping the iteration (for example with Morozovs discrepancy principle) yields approximate but stable solutions. Incorporating additional regularization could  further improve the results. Such investigations, however, are beyond the scope of this paper.

\section{Conclusion and outlook}
In this paper we derived, analyzed, and implemented iterative algorithms for PAT with  variable sound speed. We considered the full and partial data situation. In most of the cases under consideration, Landweber's method performs almost as well as the iterative  time reversal method while Nesterov's and CG method converge faster.   Note that the  semi-convergence of the CG method is visible for all noisy data. The Landweber's and the  iterative  time reversal methods  are asymptotically much slower and therefore also the semi-convergence phenomenon appears later. 
This is also the case for the full data problem; however full data seems to further stabilize the problem above the visible data problem. Especially Landweber's and Nesterov's methods are  convenient for regularization, which we will investigate in an upcoming work.

\section*{Acknowledgement} Linh Nguyen's research is partially supported by the NSF grants DMS 1212125 and DMS 1616904. He also thanks the department of Mathematics of the University of Innsbruck  for financial support and hospitality during his visit in 2016. The authors are thankful to the anonymous referees for various helpful comments and suggestions.

\appendix

\section{Formulation of weak solution} \label{A:weak_soln} 
In this section, we define the (weak) solution $q(x,t)$ in (\ref{E:adjoint}). To that end,  we follow \cite{belhachmi2016direct}. Let us denote by $V$ the Hilbert space with the inner product
$$\left<\varphi_1, \varphi_2 \right> = \int_\Om \varphi_1(x) \, \varphi_2(x) + \int_{\R^n} \nabla \varphi_1(x)  \nabla \varphi_2(x) \, dx,$$
and $V'$ the dual space of $V$. Then, space of test functions $C_0^\infty(\R^n)$ is dense in $V$. Therefore, $V'$ is a subset of the space of distribution $\mD'(\R^n)$. Moreover,
$$V'=\{d \in \mD'(\R^n): \mbox{ there is $C>0$ such that: } |(d,f)| \leq C \|f\|_V,~ \mbox{ for all } f \in C_0^\infty(\R^n)\}.$$ 
Here and elsewhere, $(d,f)$ is the action of the distribution $d$ on the test function $f$. 
\begin{definition} Let $h \in C^\infty( \partial \Om \times (0,T))$. A weak solution of the problem
\begin{eqnarray}\label{E:weak} 
\left\{\begin{array}{l} c^{-2}(x) \, z_{tt}(x,t) - \, \Delta z(x,t) =0,
\quad (x,t) \in (\R^d \setminus \partial \Om) \times (0,T), \\[6 pt] z(x,T) =0, \quad z_t(x,T) =0, \quad x \in \R^d, \\[6 pt]
\big[ z \big](y,t) =0, \Big[\frac{\partial z}{\partial \nu} \Big](y,t) =h (y,t),  \quad (y,t) \in \partial \Om \times [0,T]. \end{array} \right.
\end{eqnarray}
is a function $z$ that satisfies:
\begin{itemize}
\item[1.] $z \in L^2(0,T;V), z' \in L^2(0,T; L^2(\R^n)), z'' \in L^2(0,T;V')$,
\item[2.] $z(T)=z'(T)=0$,
\item[3.] for any $v \in L^2(0,T; V)$, we have \begin{multline} \label{E:variation} \int_0^T  \left<c^{-2}(\edot) \, z_{tt}(\edot,t),v(\edot,t) \right>_{(V',V)} \, dt + \int_0^T \int_{\R^d} \nabla z(x,t) \, \nabla v(x,t) \,dx \, dt =
\\ - \int_0^T \int_{\partial \Om} h(y,t) \, v(y,t) \, dy \, dt.
\end{multline}
\end{itemize}
\end{definition}
 Its existence and uniqueness can be found in \cite{belhachmi2016direct}. 

Assume that $v \in C^\infty(\R^d \times \overline \R)$ such that $v(\edot,t) \in C_0^\infty(\R^d)$ for all $t \in \overline \R$. Taking integration by parts of the second term of the left hand side with respect to $x$, we can write (\ref{E:variation}) in the form
\begin{equation*}  \int_0^T  \big(c^{-2} \, z_{tt}(\edot,t) -\Delta z(\edot, t) ,v(\edot,t) \big) \, dt  = - \int_0^T \int_{\partial \Om} h(y,t) \, v(y,t) \, dy \, dt, \end{equation*}
or 
 \begin{equation*} c^{-2}(\edot) \, z_{tt}(\edot,t) -\Delta z(\edot, t)  = -  \delta_{\partial \Om}(\edot)  \, h(\edot,t).  \end{equation*}
Therefore, (\ref{E:weak}) can be formally rewritten as the nonhomogeneous wave problem
\begin{eqnarray} \label{E:reformg}
\left\{\begin{array}{l} c^{-2}(x) \, z_{tt}(x,t) - \, \Delta z(x,t) = -  \delta_{\pd \Om}(x) \,h(x,t) , \quad (x,t) \in \R^d \times(0,T), \\[6 pt] z(x,T) =0, \quad z_t(x,T) =0, \quad x \in \R^d. \end{array} \right.
\end{eqnarray}

\section{The $k$-space method for numerically solving the wave equation}
\label{sec:kspace}

In this subsection we briefly describe the $k$-space method as we use it to
numerically compute the solution of wave equation,   which is required for evaluating
the forward operator $\Lo$ and its adjoint $\Lo^*$.

Consider the solution $p \colon \R^2  \times (0, T) \to \R$ of the
two-dimensional wave equation
 \begin{align} \label{eq:wave2d}
	&c^{-2}(x) \, p_{tt}(x,t) -   \Delta p(x,t)  = s(x,t)
	&&  \text{ for }(x,t) \in \R^2 \times (0,T)  \,, \\
 	&
	p(x,0)  = f(x)
	&&  \text{ for }  x  \in \R^2  \,,\\
 	&
	p_t(x,0)  = 0
	&&  \text{ for }  x  \in \R^2  \,,
\end{align}
where $s \colon \R^2 \times   (0,T) \to \R$ is a given source term
and $f  \colon \R^2   \to \R$ the given initial pressure.
Several well investigated methods for numerically solving~\eqref{eq:wave2d}
(and analogously for the wave equation in higher
dimensions) are available and  have been used for photoacoustic tomography. This
includes finite difference methods \cite{Burgholzer2007,nguyen2016dissipative,StefanovYang2},
finite element methods \cite{belhachmi2016direct} as well as  Fourier spectral and
$k$-space methods \cite{cox2007k,huang2013full,treeby2010k}.
In this  paper we use a $k$-space method for numerically solving  \eqref{eq:wave2d}
because this method does not suffer from numerical dispersion. The $k$-space method is implemented in the
freely available $k$-wave toolbox (see \cite{treeby2010k}); in order
to be flexible in our implementations  we have developed
our own code as described below.

The $k$-space method  makes the ansatz (see \cite{cox2007k,mast2001k,tabei2002k})
\begin{equation} \label{eq:waveK1}
p(x,t) =  w(x,t) - v(x,t) \quad  \text{ for }(x,t) \in \R^2 \times (0,T) \,,
\end{equation}
where the pressure $p$ is written as linear combination of the
auxiliary quantities  $ w \coloneqq  c_0^2/c^2  \, p $ and $v  \coloneqq (1-c_0^2/c^2)  \, p $.
Here $c_0>0$ a suitable constant; we take  $c_0  \coloneqq  \max \set{ c(x) \colon x \in \R^2}$.
One easily verifies that the wave \eqref{eq:wave2d} is equivalent to the following
system of equations,
\begin{equation} \label{eq:waveK}
\left\{
\begin{aligned}
	w_{tt}(x,t) -  c_0^{2} \, \Delta w(x,t) & =
	c_0^2 \, s(x,t) - c_0^2 \, \Delta v(x,t) &&\text{ for }(x,t) \in \R^2 \times (0,T) \,,\\
	v(x,t) &= \frac{c(x)^2 - c_0^2}{c_0^2}  \, w(x,t) &&  \text{ for }(x,t) \in \R^2 \times (0,T)\,.
\end{aligned} \right.
\end{equation}
Interpreting $\Delta v$ as an additional source term, the first equation in \eqref{eq:waveK} is a standard wave equation
with  constant sound speed. This suggests the time stepping  formula
 (see \cite{cox2007k,mast2001k})
\begin{multline} \label{eq:timestep}
w(x,t + h_t)
=2  w(x,t)  -   w(x,t - h_t)
\\
- 4  \Fo_\xi^{-1} \left[  \sin(c_0 \sabs{\xi} h_t/2)^2 \Fo_x [w(x,t) - v(x,t)]     -
(c_0h_t/2)^2  \sinc(c_0 \sabs{\xi} h_t/2)^2   \Fo_x [s(x,t) ]  \right]
 \,,
\end{multline}
where $\Fo_x$ and $\Fo_\xi^{-1}$ denote  the Fourier and inverse Fourier transform in the spatial variable $x$ and the spatial frequency variable $\xi$, respectively, and  $h_t > 0$ is a time stepping size.  
 Note that the factor $ 4 \sinc(c_0 \sabs{\xi} h_t/2)^2$ is 
 a distinctive feature of the $k$-space method and replaces the factor
 $ (c_0 \sabs{\xi} h_t)^2$ arising in standard finite differences.  For  constant sound speed we have   $v=0$, in which case the solution of equation~\eqref{eq:wave2d} exactly satisfies \eqref{eq:timestep}  (see, e.g., \cite{cox2007k}).
In the case of variable sound speed there is no such equivalence  because $v$ is itself dependent
on $w$. Nevertheless, in any case  \eqref{eq:timestep} serves as the basis of an efficient and   accurate iterative time stepping scheme  for numerically computing the solution of  the wave equation.

The resulting $k$-space method for  solving~\eqref{eq:wave2d}
is summarized in Algorithm~\ref{alg:kpace}.

\begin{alg}[The $k$-space\label{alg:kpace} method]
For given initial pressure $f(x)$ and source  term $s(x,t)$
approximate  the solution  $p(x,t)$ of \eqref{eq:wave2d} as follows:
\begin{enumerate}[leftmargin=3em,label= (\arabic*)]
\item\label{k1} Define  initial  conditions 
$w(x,-h_t) = w(x,0)  = v(x,0) =  c_0^2/c^2   f(x)$;
\item\label{k2} Set $t = 0$;
\item\label{k3} Compute $w(x,t + h_t)$ by evaluating \eqref{eq:timestep};
\item\label{k4} Compute $v(x,t + h_t) \coloneqq  
\kl{c^2(x)/c_0^2- 1}
  \, w(x,t+h_t)$;
\item\label{k5} Compute  $p(x,t + h_t) \coloneqq w(x,t + h_t) -  v(x,t + h_t)$;
\item\label{k6} Set $t \gets t+h_t $ and go back to \ref{k3}.
\end{enumerate}
\end{alg}

 Algorithm~\ref{alg:kpace} can directly be used to evaluate the
forward operator $\Lo f$ by taking  $s(x,t) =0$ and  restricting
the solution to the measurement surface $S_R$, that is $\Lo f = p|_{S_r \times (0, T)}$.
Recall that the  adjoint operator is given by $\Lo^*g  = q_t(\edot,0)$,
where $q\colon \R^2 \times (0,T) \to \R$ satisfies the adjoint wave equation
 \begin{align} \label{eq:wave2ad}
	&  c^{-2}(x) \, q_{tt}(x,t) -    \Delta q(x,t)  = -    \delta_{S_R}(x) \, g(x,t)
	&&  \text{ for }(x,t) \in \R^2 \times (0,T) \\
 	&
	q_t(x,T) = q(x,T)  = 0
	&&  \text{ for }  x  \in \R^2  \,.
\end{align}
By substituting $t \gets T- t$  and taking $s(x,t)  =    g(x,T-t) \,  \delta_{S}(x)$
as source term in \eqref{eq:wave2d}, Algorithm~\ref{alg:kpace}
can also be used to evaluate the $\Lo^*$. In the partial data case where
measurements are made on a subset $S \subsetneq S_R$ only, the
adjoint  can be implemented by taking the source
 $s(x,t)  =  \chi(x,t) \,  g(x,T-t)\, \delta_{S_R}(x)$ with an
appropriate window function $\chi(x,t)$.
In order to   use all available data, in  our implementations
we  take  the window function to be equal to one on the observation
part $S$ and zero outside.  This choice of the window function is known to create streak artifacts
into the picture \cite{frikel2015artifacts,sima,BFNguyen}. However, the artifacts fade away quickly after several iterations when the problem is well-posed.

\end{document}